\documentclass[11pt]{amsart}

\usepackage{amsmath, amssymb, amscd, cancel, color, graphicx, pinlabel, soul, stmaryrd}

\usepackage{thmtools}
\usepackage{thm-restate}
\usepackage{mathdots}

\usepackage[all]{xy}
\usepackage{multirow}
\usepackage{longtable}
\usepackage{array}
\usepackage{paralist}

\newtheorem{thm}{Theorem}[section]

\newtheorem{cor}[thm]{Corollary}
\newtheorem{lem}[thm]{Lemma}
\newtheorem{prop}[thm]{Proposition}
\newtheorem{ques}[thm]{Question}

\def\cB{{\mathcal B}}

\def\I{{\mathcal I}}

\def\bR{{\mathbb R}}

\def\bZ{{\mathbb Z}}

\def\Hom{{\mathrm{Hom}}}

\def\Spc{{\mathrm{Spin^c}}}

\def\del{{\partial}}

\def\rk{{\mathrm{rk}}}

\begin{document}

\title[Fibered simple knots]%
{Fibered simple knots}

\author[Joshua Evan Greene]{Joshua Evan Greene}

\address{Department of Mathematics, Boston College\\ Chestnut Hill, MA 02467}

\email{joshua.greene@bc.edu}

\author[John Luecke]{John Luecke}

\address{Department of Mathematics, University of Texas, Austin\\ Austin, TX 78712}

\email{luecke@math.utexas.edu}

\thanks{JEG was supported on NSF Award DMS-2005619.}

\maketitle

\medskip

\noindent {\bf Abstract.}
We prove that a simple knot in the lens space $L(p,q)$ fibers if and only if its order in homology does not divide any remainder occurring in the Euclidean algorithm applied to the pair $(p,q)$.
One corollary is that if $p=m^2$ is a perfect square, then any simple knot of order $m$ fibers, answering a question of Cebanu.
More generally, we compute the leading coefficient of the Alexander polynomial of a simple knot, and we describe how to construct a minimum complexity Seifert surface for one.
The methods are direct, combinatorial, and geometric.
\medskip


\section{Introduction.}

\subsection{Background.}

Simple knots first arose in Berge's work on lens space surgeries \cite{berge:lens}, and they play an important role in Heegaard Floer homology and Dehn surgery.
Recall that a 3-dimensional lens space is a manifold besides $S^3$ and $S^1 \times S^2$ with a genus-1 Heegaard decomposition.
A pair of compressing disks, one in each Heegaard solid torus, is standard if their boundaries are in minimal position on the Heegaard torus.
A {\bf simple knot} is a knot in a lens space built from the union of a pair of properly embedded arcs, one in each of a pair of standard compressing disks.

Heegaard Floer homology assigns invariants $\widehat{\mathrm{HF}}(Y)$ and $\widehat{\mathrm{HFK}}(K)$ to a closed 3-manifold $Y$ and a knot $K \subset Y$.
When $Y$ is a rational homology sphere, they obey the inequalities
\[
\rk \, \widehat{\mathrm{HFK}}(K) \ge \rk \, \widehat{\mathrm{HF}}(Y) \ge |H_1(Y;\bZ)|.
\]
The knot $K$ is {\bf Floer minimal} if the first inequality is an equality, and the manifold $Y$ is an {\bf $L$-space} if the second inequality is an equality.
Simple knots and lens spaces are prominent examples of Floer minimal knots and $L$-spaces, respectively.

The Berge conjecture posits that a knot in a lens space with an integral surgery to $S^3$ is a simple knot \cite{berge:lens}.
Hedden and Rasmussen independently proved that a knot in an $L$-space $Y$ with an integral surgery to $S^3$ is Floer minimal, subject to a mild assumption on the knot genus which is always satisfied when $Y$ is a lens space \cite{greene:cabling,hedden:berge,r:Lspace}.
There is an analogous conjecture and unconditional result with $S^1 \times S^2$ in place of $S^3$ \cite{greene:2013-4,nivafaee}.
Thus, the Berge conjecture and its $S^1 \times S^2$ version would follow from the conjecture that Floer minimal knots in lens spaces are simple \cite{bgh:lens,r:Lspace}.
It is known which simple knots in lens spaces admit $S^3$ surgeries, while it is an open problem to determine which ones admit $S^1 \times S^2$ surgeries \cite{bbl:2016,greene:2013-4}.

Work by several researchers shows that if a knot $K$ in an $L$-space has an $S^3$ or $S^1 \times S^2$ surgery, then it is {\bf fibered}, meaning that its exterior fibers over the circle.\footnote{Some authors use the term {\em rationally fibered}.}
There are several noteworthy results in this vein.
A homology argument shows that if $K \subset Y$ has an integer surgery to $S^3$, then its homology class is primitive, while if it has an integer surgery to $S^1 \times S^2$, then its order in homology is the square-root of $|H_1(Y;\bZ)|$.
Ozsv\'ath and Szab\'o proved that any simple knot whose homology class is primitive is fibered \cite{os:lens}.
Using a computer, Cebanu observed that for all $\theta < 500$, any simple knot whose homology class has order $\theta$ in a lens space of order $\theta^2$ fibers, constituting billions of examples.
He raised the question as to whether this is always the case, without any bound on $\theta$  \cite[Question 4.0.3]{cebanuthesis}.
Ni and Wu proved that whether a Floer minimal knot in an $L$-space fibers depends only on its homology class \cite[Corollary 5.3]{niwu2014}.
Since simple knots are Floer minimal and every homology class in a lens space contains one, it follows that the determination of which Floer minimal knots in lens spaces fiber reduces to the corresponding question for simple knots.


\subsection{The main result.}
\label{ss: main result}
The preceding discussion leads us to our main result, a complete characterization of when a simple knot fibers.
To state it, we require a little more terminology.
For a pair of relatively prime positive integers $p > q$,
write out the steps of the Euclidean algorithm applied to the pair $(p,q)$:

\begin{eqnarray*}
p &=& d_1 q + r_1 \\
q &=& d_2 r_1 + r_2 \\
r_1 &=& d_3 r_2 + r_3 \\
&\vdots& \\
r_{n-2} &=& d_n r_{n-1} + 0.
\end{eqnarray*} 

\noindent
The {\bf remainders} are the values $r_1 > \cdots > r_{n-1}=1$. 
They can also be defined recursively by taking $r_{-1} = p$, $r_0 = q$, and $r_{i+1} = [r_{i-1}]_{r_i}$ for $i \ge 0$, $r_i > 0$, where $[a]_p$ denotes the least positive residue of $a \pmod p$.
The {\bf coefficients} are the values $d_1, \dots, d_n$. 
They are so named in connection with continued fractions: one writes $p/q = [d_1,\dots,d_n]$.
A {\bf harmonic} of the pair $(p,q)$ is a common divisor of $p$ and one of the remainders. 
Lastly, the {\bf order} of an (oriented) knot is the order of its homology class.

\begin{thm}
\label{thm: fiber}
A simple knot in $L(p,q)$ fibers iff its order is not a harmonic of the pair $(p,q)$.
\end{thm}

\begin{cor}
\label{cor: floer minimal}
A Floer minimal knot in $L(p,q)$ fibers iff its order is not a harmonic of the pair $(p,q)$.
\end{cor}

\begin{proof}
Immediate from Theorem \ref{thm: fiber} and
\cite[Corollary 5.3]{niwu2014}.
\end{proof}

\begin{cor}
\label{cor: fiber heredity}
If $K_1$ and $K_2$ are Floer minimal knots in a lens space,
$K_1$ is homologous to a multiple of $K_2$,
and $K_1$ fibers, then $K_2$ fibers.
\end{cor}

\begin{proof}
Immediate from Theorem \ref{thm: fiber} and Corollary \ref{cor: fiber heredity}, since the order of $[K_1]$ divides that of $[K_2]$.
\end{proof}

\noindent
Corollary \ref{cor: fiber heredity} therefore refines \cite[Corollary 5.3]{niwu2014} in the case that $Y$ is a lens space.
We wonder whether Corollary \ref{cor: fiber heredity} holds for a Floer minimal knot in an arbitrary $L$-space.

\begin{restatable}{cor}{orderbound}
\label{c: order bound}
A simple knot of order $\theta$ in a lens space of order $p$ fibers if $(\theta+1)^2 > p+1$.
In particular, a simple knot of order $\theta$ in a lens space of order $\theta^2$ fibers.
\end{restatable}

\noindent
We give the short derivation of this result from Theorem \ref{thm: fiber} in Section \ref{s: fibering} after developing a little more notation.
The second statement in Corollary \ref{c: order bound} affirmatively answers Cebanu's question.
The bound in Corollary \ref{c: order bound} is sharp, in the sense that the simple knots of order $\theta$ in $L(\theta(\theta+2),\theta+1)$ do not fiber, as the reader may check using Theorem \ref{thm: fiber}.
We wonder whether Corollary \ref{c: order bound} generalizes to a Floer minimal knot in an arbitrary $L$-space of order $p$.


\subsection{Essential notation.}
\label{ss: notation}
We deduce Theorem \ref{thm: fiber} on a combination of two other results, Theorems \ref{thm: characterization} and \ref{thm: minimizer}.
We prepare some notation for their statements.

Fix a pair of relatively prime positive integers $p > q$.
Take a genus-1 Heegaard decomposition of the lens space $L(p,q) = V_\alpha \cup V_\beta$, $V_\alpha, V_\beta \approx S^1 \times D^2$, and choose meridian disks $D_\alpha \subset V_\alpha$, $D_\beta \subset V_\beta$ so that $\alpha = \del D_\alpha$ and $\beta = \del D_\beta$ meet transversely in $p$ points of intersection.
Orient $\alpha$ and let $x_0,\dots,x_{p-1}$ denote the cyclic order of its points of intersection with $\beta$.
For $0 < k < p$, we obtain the (oriented) simple knot $K(p,q,k) \subset L(p,q)$ by taking the union of properly embedded, oriented arcs $\delta \subset D_\alpha$, $\epsilon \subset D_\beta$ with $\del \delta = - \del \epsilon = x_0 - x_k$.
Observe that there is an ambient isotopy of $L(p,q)$ that sends each $x_i$ to $x_{i+1}$, so $K(p,q,k)$ does not depend on the choice of $x_0$.
Furthermore, $[K(p,q,k)] = k [C_\beta] \in H_1(L(p,q);\bZ)$, where $C_\beta$ denotes the core of the Heegaard solid torus $V_\beta$, oriented so that $\beta$ links negatively with it.
Hence there is a unique simple knot representing each non-trivial homology class.

\begin{figure}
\includegraphics[width=6in]{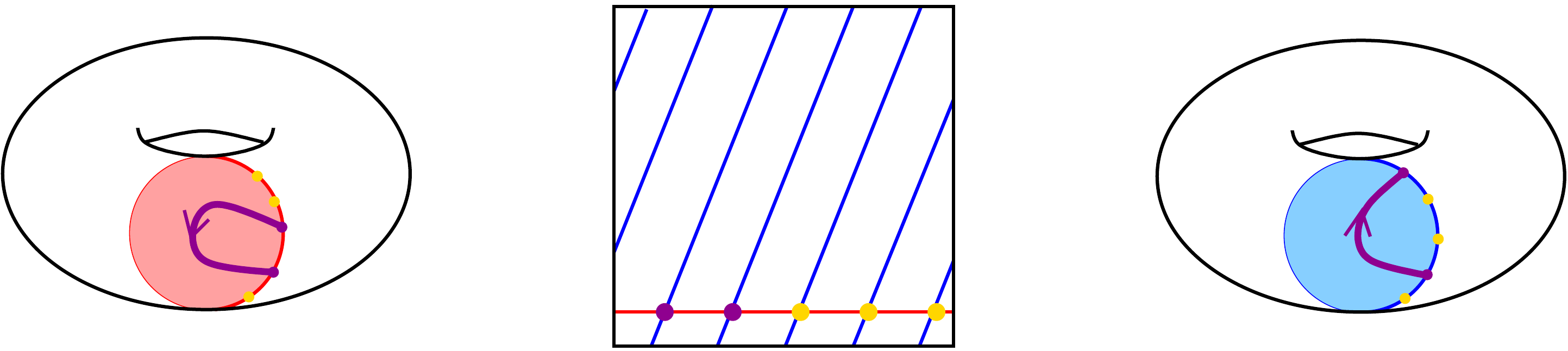}
\put(-390,30){$\delta$}
\put(-70,30){$\epsilon$}
\put(-430,85){$V_\alpha$}
\put(-290,85){$T$}
\put(-110,85){$V_\beta$}
\put(-355,21){$x_0$}
\put(-352,32){$x_1$}
\put(-260,17){$x_0$}
\put(-242,17){$x_1$}
\put(-35,21){$x_0$}
\put(-40,49){$x_1$}
\caption{The simple knot $K(5,2,1) = \delta \cup \epsilon$.}
\label{fig: simple}
\end{figure}
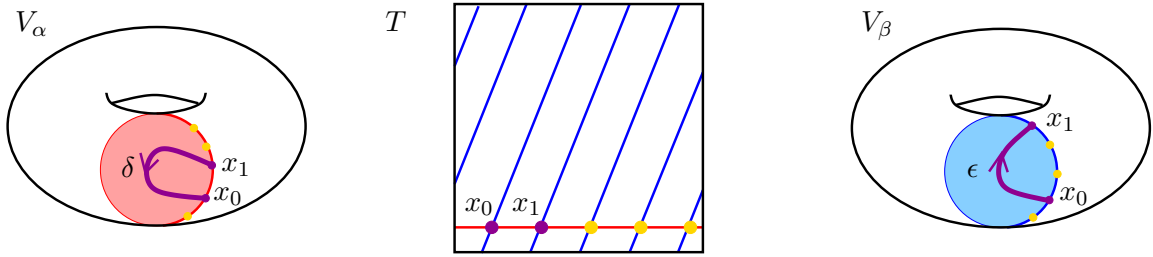

Given relatively prime positive integers $p > q$ and a positive integer $k < p$, we define a sequence $S(p,q,k)$ as follows.
Let $q'$ denote the smallest positive integer such that $q' q \equiv 1 \pmod p$.
For an integer $a$, let $[a]_p$ denote the least positive residue of $a \pmod p$.
Define
\begin{equation}
\label{eq: l theta t}
l = [q' k]_p, \quad \theta = p / \gcd(p,k), \quad \textup{and} \quad t = l / \gcd(p,l) = l / \gcd(p,k) = l \theta / p.
\end{equation}
Note that $\theta$ is the order of the simple knot $K(p,q,k)$, $0 < t < \theta$, and $\gcd(t,\theta) = 1$.
Let $S \subset \{1,\dots,p-1\}$ denote the set of values
\[
\{ [q]_p, [2q]_p,\dots, [lq]_p \}.
\]
Note that $[lq]_p = k$.
Define a function $g : \{0,\dots,p-1\} \to \bZ$ uniquely by the requirement that
\begin{equation}
\label{eq: jump function}
g(x) - g(x-1) = \begin{cases} t, & x \notin S, \\ t - \theta, & x \in S, \end{cases}
\end{equation}
for all $0 < x < p$ and $\min(g)=0$.
Extend $g$ uniquely to a $p$-periodic function on the integers by the same name.
Finally, define $S(p,q,k)$ to be the sequence of non-negative integers
$g(x)$, $x =0,\dots,p-1$.


\subsection{Secondary results.}

The sequence $S(p,q,k)$ arises in various guises, such as in connection with the Alexander polynomial $\Delta(p,q,k)$ of the knot $K(p,q,k)$:

\begin{prop}
\label{prop: coeffs}
The generating function of the sequence $S(p,q,k)$ equals $\Delta(p,q,k)$, multiplied by a monic polynomial.
In particular, the leading coefficient of $\Delta(p,q,k)$ equals the number of times the sequence $S(p,q,k)$ attains its minimum value.
\end{prop}

\begin{restatable}{thm}{characterization}
\label{thm: characterization}
The following conditions are equivalent:
\begin{enumerate}
\item
the simple knot $K(p,q,k)$ fibers;
\item
the sequence $S(p,q,k)$ assumes its minimum value once; and
\item
the Alexander polynomial $\Delta(p,q,k)$ is monic. \qed
\end{enumerate}
\end{restatable}

Both results should look familiar to experts, and we do not claim any originality for either one.
In fact, Theorem \ref{thm: characterization} holds in somewhat greater generality, as we discuss in connection with Floer homology just below.
The equivalence (1)$\iff$(2) in Theorem~\ref{thm: characterization} follows from work of Brown and Stallings \cite{brown1987,stallings1962}.
It has already been applied in this setting by Ozsv\'ath and Szab\'o, who used it to prove the case $\theta = p$ of Theorem \ref{thm: fiber} \cite[Proposition 1.8]{os:lens}.
The equivalence (2)$\iff$(3) follows from Proposition \ref{prop: coeffs}.
We prove a sharper version of Proposition \ref{prop: coeffs} in Section \ref{s: fibering}, which is needed in the proof of Theorem \ref{thm: seifert} below.

Theorem \ref{thm: fiber} follows at once on combination of Theorem \ref{thm: characterization} with:

\begin{restatable}{thm}{thmminimizer}
\label{thm: minimizer}
The number of times $S(p,q,k)$ achieves its minimum is
\[
\prod_{0 < j < n, \, \theta | r_j} {\lceil r_{j-1}/r_j \rceil}=\prod_{0 < j < n, \, \theta | r_j} (d_{j+1}+1),
\]
where $\theta$ denotes the order of $k \pmod p$ and $r_i$ and $d_i$ denote the remainders and the coefficients of the Euclidean algorithm applied to the pair $(p,q)$.
In particular, the minimum is attained uniquely iff $\theta$ is not a harmonic of the pair $(p,q)$.
\end{restatable}

Note the conventions that $r_0 = q$ and that the empty product equals $1$.

To get a geometric handle on the sequence $S(p,q,k)$, we describe in Section \ref{s: branched surfaces} how to produce a {\em rational domain} associated to a simple knot.
From it we describe how to construct a branched surface in the knot exterior and in turn a (rational) Seifert surface $F(p,q,k)$ for the knot which the branched surface fully carries.

\begin{restatable}{thm}{thmseifert}
\label{thm: seifert}
The Seifert surface $F(p,q,k)$ is taut.
\end{restatable}
\noindent
Relatedly, Ken Baker asks the following:

\begin{ques}
\label{q: unique}
Does every simple knot in a lens space admit a unique taut Seifert surface?
\end{ques}

\subsection{Complements.}

{\em Floer homology.}
We briefly discuss the relationship between the material of the previous subsection and Floer homology.
If $Y$ is a rational homology sphere and $K \subset Y$ is a knot with (rational) Seifert surface $F$, then $K$ is fibered with fiber $F$ if and only if a piece of $\widehat{\mathrm{HFK}}(K)$ determined by $F$ is infinite cyclic \cite[Theorem 2.3]{niwu2014}.
For the case of a Floer simple knot $K$, it is easy to check that this group is infinite cyclic if and only if its Alexander polynomial $\Delta(K)$ is monic.
Thus, the equivalence (1)$\iff$(3) of Theorem \ref{thm: characterization} holds more generally for Floer simple knots.
In addition, there is a related condition to Theorem \ref{thm: characterization} (1)$\iff$(2) involving the $d$-invariant.
Specifically, Ni and Wu show that a Floer simple knot $K$ in an $L$-space $Y$ fibers iff the sequence of values
\[
d(Y,s+\mathrm{PD}[K])-d(Y,s), \quad s \in \Spc(Y),
\]
attains its maximum (equivalently, its minimum) value exactly once \cite[Proposition 5.2]{niwu2014}.
Moreover, their argument and the symmetry of the Alexander polynomial shows that the number of times this sequence attains its minimum value is equal to the leading coefficient of $\Delta(K)$.
If $K$ is simple, then Proposition \ref{prop: coeffs} and the symmetry of $\Delta(K)$ shows that this value equals the number of times $S(p,q,k)$ attains its minimum value.
All told, the number of times the above differences of $d$-invariants above attain their minimum value is equal to the number of times $S(p,q,k)$ attains its minimum value.
There exist various means for calculating $d$-invariants, both in terms of explicit formulas and in terms of lengths of vectors in lattices, which have led to many applications.
We were briefly tempted to deploy them towards proving Theorems \ref{thm: fiber} and \ref{thm: minimizer}.
However, we ultimately prevailed in terms of the direct definition of the sequence $S(p,q,k)$.

{\em The discovery of Theorem \ref{thm: fiber}.} We first conjectured Theorem \ref{thm: fiber} on the basis of computer experiment.
Using Theorem \ref{thm: characterization} (1) $\iff$ (2), we wrote a short Mathematica script to determine all triples $(p,q,k)$ for which $K(p,q,k)$ fibers with $p \le 100$.
Various patterns emerged besides those we already knew from \cite{cebanuthesis} and \cite{os:lens}.
For instance, we noticed early on that $K(p,q,k)$ fibers if $p \le 100$ and $p \equiv 1 \pmod q$, independent of $k$.
The condition that $p \equiv 1 \pmod q$ is equivalent to condition that the continued fraction expansion of $p/q$ has length two and in turn that $n = 2$ in the Euclidean algorithm applied to the pair $(p,q)$.
We then looked at pairs $(p,q)$ for which the continued fraction expansion has length three, inspired by this observation and the historic  interplay between continued fractions and the topology of lens spaces.
That case involved $k$ and led us to conjecture Theorem \ref{thm: fiber}, which perfectly matched the data.
The proofs took much longer to find!

{\em Race tracks.}
Finally, we mention a passing relationship between our work and the famous race track problem, which we quote from Lov\'asz's problem book \cite[Problem 3.21]{lovasz:problems}:

\begin{quotation}
Along a speed track there are some gas-stations.
The total amount of gasoline available in them is equal to what our car (which has a very large tank) needs for going around the track.
Prove that there is a gas-station such that if we start there with an empty tank, we shall be able to go around the track without running out of gasoline.
\end{quotation}
(See \cite{lovasz:problems} for a solution.)
We imagine a speed track $p$ units in length around which there are $k$ gas-stations spaced $q$ units apart, each of which holds enough gasoline to get the car $p/k$ units around the track.
The sequence $S(p,q,k)$ records, up to an affine shift, the car's fuel reading at each unit of its circuit.
The number of gas-stations where the car can begin a full clockwise circuit on an empty tank is equal to the number of minima of the sequence $S(p,q,k)$.


\subsection{Organization.}

Theorem~\ref{thm: minimizer} and its corollary, Theorem~\ref{thm: fiber}, are proven over the course of Sections~\ref{ss: domains}, \ref{sec:findingmins}, and \ref{s: minima}.

In Section~\ref{s: branched surfaces}, we give a geometric, two-dimensional interpretation to the sequence $S(p,q,k)$.
We do this in terms of rational domains in subsection \ref{ss: domains} and in terms of branched surfaces in subsection \ref{ss: branched surface}.
The first interpretation gets used towards the proofs of Theorems~\ref{thm: fiber} and \ref{thm: minimizer}, while the second one leads to the construction of the Seifert surface $F(p,q,k)$ for $K(p,q,k)$ appearing in Theorem~\ref{thm: seifert}.

In Section~\ref{sec:findingmins}, we show that if $\theta$ is a harmonic of the pair $(p,q)$, then the sequence $S(p,q,k)$ attains its minimum value more than once.
On combination with Theorem~\ref{thm: characterization}, this establishes the ``only if" direction of Theorem \ref{thm: fiber}.
The point of view from Section~\ref{s: branched surfaces} reveals symmetries among the minima of $S(p,q,k)$ in terms of the coefficients of $(p,q)$.
In particular, Theorem~\ref{thm:lotsofzerosg} shows that the minima of $S(p,q,k)$ occur in sequences of length $\lceil r_{i-1} / r_{i} \rceil$ and on a scale of length $r_i$, where $\theta$ divides the remainder $r_i$ of $(p,q)$.
This information then gets used in the proof of Theorem~\ref{thm: minimizer}.

In Section~\ref{s: minima}, we give a recursive procedure for isolating the minima of $S(p,q,k)$ into successively smaller intervals of length $r_i$, the remainders of $(p,q)$ for which $\theta \nmid r_i$.
At the end this section we prove Theorem~\ref{thm: minimizer}, by combining the restriction of the minima when $\theta \nmid r_i$ with the sequences of minima found in Theorem~\ref{thm:lotsofzerosg} when $\theta \mid r_i$. 
The argument of Section \ref{s: minima} was first discovered and carried out under the assumption that $\theta$ is not a harmonic of $(p,q)$, from which the ``if" direction of Theorem \ref{thm: fiber} follows.
In this case, a recursive application of Lemmas~\ref{lem: minima} and \ref{lem: minima 2-2} of Sections~\ref{ss: minima prep} and \ref{ss: minima does not divide} implies the uniqueness of the minimum of $S(p,q,k)$.
We recommend going through this case at a first pass in order to understand the structure of the argument.

In Section~\ref{s: fibering}, we prove the secondary results Proposition~\ref{prop: coeffs} and Theorem \ref{thm: characterization}.
Again, this material is classical, and we do not claim any originality about the results or their proofs.
Using the information about the Alexander polynomial and a lemma from subsection~\ref{ss: branched surface}, we prove Theorem \ref{thm: seifert}.
We conclude with the short deduction of Corollary~\ref{c: order bound} from Theorem \ref{thm: fiber}.


\section*{Acknowledgments.}
We thank Faramarz Vafaee for conversations that sparked this work and Ken Baker and John Baldwin for stimulating questions along the way.


\section{Simple knots and branched surfaces.}
\label{s: branched surfaces}

Given a simple knot $K \subset L(p,q)$, we review how to construct a non-negative domain $D$ on the Heegaard torus associated with it.
Using the domain $D$, we construct a branched surface $\cB$ in the exterior of $K$.
The branched surface $\cB$ fully carries a rational Seifert surface $F$ for $K$.
Later, we show in Theorem \ref{thm: seifert} that $F$ is taut.
It then follows from Theorem \ref{thm: fiber} that $F$ is a fiber iff $D$ has a single region with coefficient 0.


\subsection{Chains and domains.}
\label{ss: domains}
The construction of a non-negative domain associated with a simple knot $K \subset L(p,q)$ follows a familiar construction in Heegaard Floer homology (see, for instance, \cite[\S5]{greene:2013-1}).
What is noteworthy is the manner of presentation of $K$ as a 1-cycle connecting a pair of generators in the Heegaard Floer chain complex.

Let $T = \bR^2 / \bZ^2$ denote a flat, oriented 2-torus.
Let $\alpha \subset T$ be a geodesic curve of slope $0/1$ and $\beta \subset T$ a geodesic curve of slope $p/q$, oriented so that $\alpha \cdot \beta = p > 0$.
Let $V_\alpha$ and $V_\beta$ be oriented solid tori with meridian disks $D_\alpha$ and $D_\beta$.
We choose diffeomorphisms $\del V_\alpha \approx T$ and $-\del V_\beta \approx T$ sending $\del D_\alpha$ to $\alpha$ and $\del D_\beta$ to $\beta$.
Gluing with these diffeomorphisms results in the lens space $L(p,q)$.
We co-orient $T \subset L(p,q)$ using the outward-pointing co-orientation to $V_\alpha$.

Let $x_0,\dots,x_{p-1}$ denote the intersection points between $\alpha$ and $\beta$, labeled in the cyclic order that they appear around $\alpha$.
The curves $\alpha$ and $\beta$ give a cell decomposition of $T$ whose 2-cells are called {\bf regions}.
Label the regions $R_0,\dots,R_{p-1}$ in the cyclic order that they appear to the right of $\alpha$, so that the top left corner of $R_i$ is $x_i$.

For a positive integer $k < p$, the oriented simple knot $K(p,q,k) \subset L(p,q)$ is the union of two oriented arcs along their boundaries, $K(p,q,k)= \delta \cup \epsilon$, where $\delta$ is properly embedded in $D_\alpha$, $\epsilon$ is properly embedded in $D_\beta$, and $\del \delta = - \del \epsilon = x_0 - x_k$.
The arcs $\delta$ and $\epsilon$ are unique up to isotopy rel. endpoints.
They are isotopic rel. endpoints within their respective meridian disks to a pair of oriented subarcs $\alpha_1 \subset \alpha$, $\beta_1 \subset \beta$.
Let $\alpha_2 \subset \alpha$, $\beta_2 \subset \beta$ denote the complementary oriented subarcs.

Let $\gamma = \alpha_1 \cup \beta_1$.
It is a 1-cycle on $T$ homologous to $K(p,q,k)$ in $L(p,q)$.
The homology group $H_1(L(p,q);\bZ)$ is isomorphic to the quotient of $H_1(T;\bZ)$ by the subgroup freely generated by the classes $[\alpha]$ and $[\beta]$.
Let $\theta$ denote the order of $[K(p,q,k)] \in H_1(L(p,q);\bZ)$; this is the same as the order of $k\pmod p$, which is $p / \gcd(p,k)$.
It follows that $\theta \cdot [\gamma] \in H_1(T;\bZ)$ belongs to the subgroup freely generated by $[\alpha]$ and $[\beta]$, so there exist unique coefficients $d, e \in \bZ$ such that $\theta \cdot [\gamma] = d [\alpha] + e [\beta] \in H_1(T;\bZ)$.
As a result, there exists a 2-chain $\sum c_i [R_i] \in C_2(T;\bZ)$ such that $\del \sum c_i [R_i] = \theta \cdot [\gamma] - d [\alpha] - e [\beta]$.
Any two such 2-chains differ by a multiple of the fundamental class $[T] = \sum [R_i]$.
The unique one obeying the additional condition that $\min (c_i)  = 0$ is the {\bf domain} $D$ associated with $K$.
Note that $D$ is $\theta$ times a {\em rational} domain associated to the pair of generators $x_0$, $x_k$ in the Floer chain complex of the Heegaard diagram $(T,\alpha,\beta)$ for $L(p,q)$.

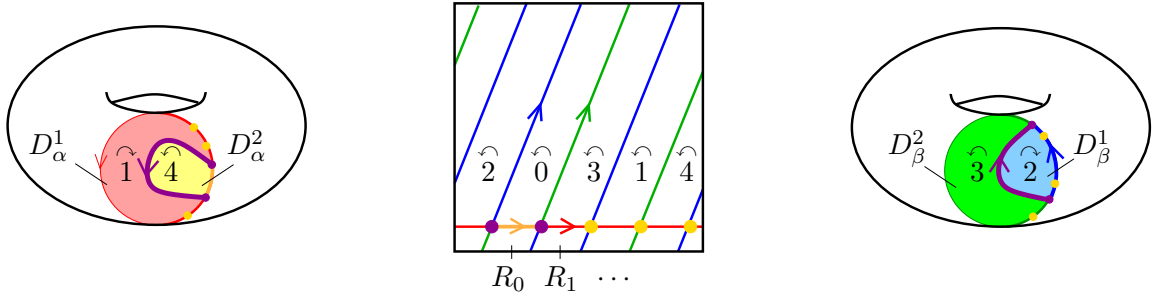
\begin{figure}
\includegraphics[width=6in]{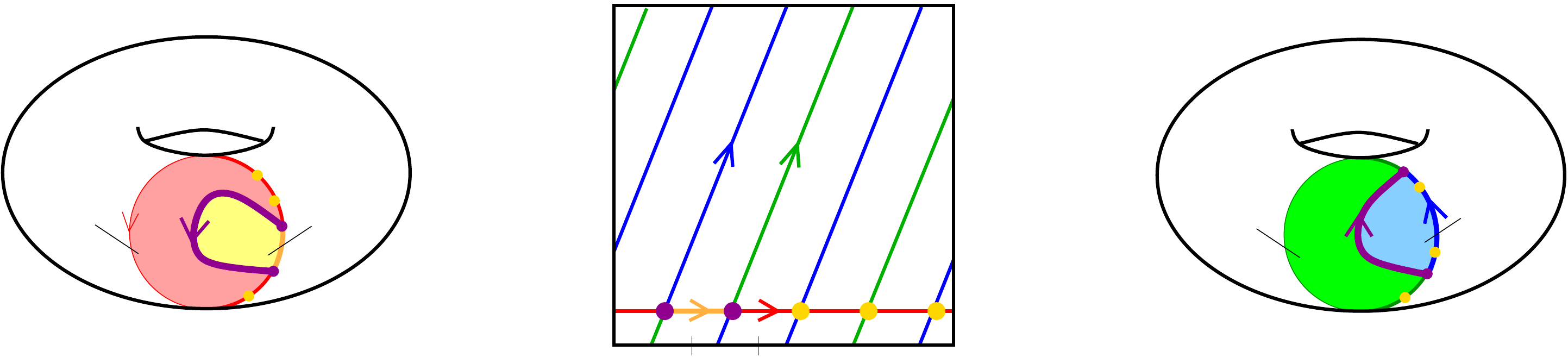}
\put(-350,40){$D^2_\alpha$}
\put(-425,40){$D^1_\alpha$}
\put(-30,40){$D^1_\beta$}
\put(-100,40){$D^2_\beta$}
\put(-392,30){$\overset{\curvearrowright}{1}$}
\put(-375,30){$\overset{\curvearrowleft}{4}$}
\put(-50,30){$\overset{\curvearrowright}{2}$}
\put(-70,30){$\overset{\curvearrowleft}{3}$}
\put(-255,30){$\overset{\curvearrowleft}{2}$}
\put(-235,30){$\overset{\curvearrowleft}{0}$}
\put(-215,30){$\overset{\curvearrowleft}{3}$}
\put(-197,30){$\overset{\curvearrowleft}{1}$}
\put(-180,30){$\overset{\curvearrowleft}{4}$}
\put(-250,-10){$R_0$}
\put(-230,-10){$R_1$}
\put(-210,-10){$\cdots$}
\caption{The domain $D$ of $K(5,2,1)$ on $T$ and its extension to a 2-chain $C$ on the spine of $L(5,2)$.  
The orientations on the 2-cells  are indicated by the arrows over their coefficients.
For quick reference: $\theta = 5$, $t = 3$, and $t' = 1$.}
\label{fig: domain}
\end{figure}

We may express the values $d$ and $e$ in terms of the parameters $p,q,k$, and $\theta$, as follows.
The points of $\alpha \cap \beta$ occur in the order $x_0, x_q, x_{2q}, \dots$ along $\beta$.
It follows that $\beta_1$ is cut into $l = [q' k]_p$ subarcs by $\alpha$, using the notation of \eqref{eq: l theta t}.
Let $\alpha'$ denote a translate of $\alpha$ to its left.
We have
\begin{equation}
\label{eq: intersection numbers}
\alpha' \cdot \del D = 0, \quad \alpha' \cdot \alpha = 0, \quad \alpha' \cdot \beta_1 = l, \quad \alpha' \cdot \beta_2 = p-l,
\end{equation}
by construction and the fact that $\del D$ is null-homologous.
We also have
\[
\del D = (\theta - d) \alpha_1 - d \alpha_2 + (\theta - e) \beta_1 - e \beta_2.
\]
Combining the indented equations gives $0 = \alpha' \cdot \del D = (\theta - e) l - e (p-l)$, leading to the identity $e = \theta l / p = t$, again using the notation of \eqref{eq: l theta t}.
Similarly, let $\beta'$ denote a translate of $\beta$ to its left.
We have
\[
\del D \cdot \beta' = 0, \quad \beta \cdot \beta' = 0, \quad \alpha_2 \cdot \beta' = k, \quad \alpha_1 \cdot \beta' = p-k.
\]
Consequently, $0 = \del D \cdot \beta' = (\theta-d)(p-k)-d k$, leading to the identity $d = (p-k) \theta / p = \theta - t'$, where $t' := k \theta/p$ by analogy to $t$.
In total, we have
\begin{equation}
\label{eq: del D}
\del D = t' \alpha_1 - (\theta - t') \alpha_2 + (\theta -t) \beta_1 - t \beta_2.
\end{equation}
Note that as $0 < k,l < p$, the coefficients on $\alpha_1$ and $\beta_1$ are positive, while the coefficients on $\alpha_2$ and $\beta_2$ are negative.

Next, we extend the domain $D$ to a 2-chain $C$ in $L(p,q)$ with boundary $\theta \cdot K(p,q,k)$.
The {\em spine} of $L(p,q)$ is the 2-complex $T \cup D_\alpha \cup D_\beta$.
We obtain a cell decomposition of the spine of $L(p,q)$ by taking the cell decomposition of $T$, attaching the disks $D_\alpha$, $D_\beta$ to $T$, and then subdividing the disks by the arcs $\delta$ and $\epsilon$, respectively.
We thereby obtain 2-cells $D^j_\alpha \subset D_\alpha$, $D^j_\beta \subset D_\beta$, $j=1,2$, with oriented boundaries
\[
\del D^1_\alpha = \delta \cup -\alpha_1, \quad \del D^2_\alpha = \delta \cup \alpha_2, \quad \del D^1_\beta = \epsilon \cup -\beta_1, \quad \del D^2_\beta = \epsilon \cup \beta_2.
\]
The orientations on $D^2_\alpha$ and $D^2_\beta$ match the orientations on $D_\alpha$ and $D_\beta$, while the orientations on $D^1_\alpha$ and $D^1_\beta$ are opposite to them.
Finally, form the 2-chain
\begin{equation}
\label{eq: 2-chain}
C= D +  t' D^1_\alpha + (\theta - t') D^2_\alpha + (\theta - t)D^1_\beta + t D^2_\beta
\end{equation}
in the cell decomposition of the spine.
By construction, $\del C = \theta \cdot (\delta \cup \epsilon) = \theta \cdot K(p,q,k)$.

Lastly, consider two consecutive regions $R_{i-1}$ and $R_i$, indices$\pmod p$.
They contain a common arc of $\beta - \alpha$ in their boundaries, and the coefficient on this arc in $\del D$ is equal to the difference $c_{i-1} - c_i$.
If it is an arc of $\beta_1$, then we have $c_i - c_{i-1} = t-\theta < 0$.
This occurs precisely when $i \in S = \{[q]_p,[2q]_p,\dots,[lq]_p\}$.
If instead it is an arc of $\beta_2$, then we have $c_i - c_{i-1} = t > 0$.
This occurs precisely when $i \notin S$.
Letting $g : \{0,\dots,p-1\} \to \bZ$ be the function of \eqref{eq: jump function} with $g(0)=c_0$, we have

\begin{equation}
\label{eq: gandc}
g(i) = c_i
\end{equation}

That is, the sequence $c_0, \dots, c_{p-1}$ is $S(p,q,k)$. 


\subsection{Construction of a branched surface and a Seifert surface.}
\label{ss: branched surface}

In this subsection, we see that 
the spine of $L(p,q)$ in the previous section gives rise to a co-oriented branched surface properly embedded
in $X_K$, the exterior of the knot $K=K(p,q,k)$.
The 2-chain $C$ of this spine induces an invariant measure on this branched surface corresponding to a rational Seifert surface for $K$ fully carried by the branched surface. This invariant measure is given in terms of $t,t',\theta$ and the sequence $S(p,q,k)$, which are easily computed from $p,q,k$ as described in Section~\ref{ss: notation}. Theorem~\ref{thm: seifert} shows that the Seifert surface resulting from this procedure is taut (i.e. Thurston norm minimizing). We refer to Oertel's paper \cite[$\S$1]{oertel84} for standard definitions, notation, and figures concerning branched surfaces.

The spine of $L(p,q)$ restricts to a properly embedded 2-complex in $X_K$ with a corresponding cell decomposition
 which we refer to as the {\em spine of $X_K$}. We will refer to the corresponding 
oriented 2-cells in the spine of $X_K$ as $R_i, D^i_{\alpha}$, and $D^i_{\beta}$, and the corresponding  oriented 1-cells
as $\alpha_i, \beta_i$. The 2-chain $C$ becomes a 2-chain, $C$, in the spine of  $X_K$. 

Let $T'=T \cap X_K$. The cells of $T'$ are co-oriented by the outward pointing normal to $V_{\alpha}$. The cells $ D^1_{\alpha}, D^2_{\alpha}, D^1_{\beta},  D^2_{\beta}$ are then co-oriented by the orientation on their boundary and the ambient orientation on $L(p,q)$.
The spine of $X_K$ can then be made into a co-oriented branched surface, $\mathcal{B'}$, properly embedded in $X_K$, by adding $ D^1_{\alpha}, D^2_{\alpha}, D^1_{\beta},  D^2_{\beta}$ to $T'$ and
smoothing so that the co-orientations of the added cells agree with that of $T'$. See Figure~\ref{fig:branched1} and Figure~\ref{fig:branchedatboundary}. The weights of the cells in the 2-chain $C$ become nonnegative weights
on the sectors of $\mathcal{B'}$ forming an invariant measure on $\mathcal{B'}$. Explicitly, these 
weights are: $c_i$ on $R_i$, $t'$ on $ D^1_\alpha$, $(\theta - t')$ on $D^2_\alpha$,  $(\theta - t)$ on $D^1_\beta$,
and $t $ on $D^2_\beta$. Recall from Section~\ref{ss: domains} that the sequence $c_0, \dots, c_{p-1}$ is the sequence $S(p,q,k)$. We then modify $\mathcal{B'}$ by removing those cells $R_i$ that
have weight zero. This gives a new co-oriented branched surface $\mathcal{B}$ and the weighting 
from the cells of $C$ gives an invariant measure on $\mathcal{B}$ with strictly positive weights.
This invariant measure on $\mathcal{B}$ corresponds to a surface $F = F(p,q,k)$ properly embedded in 
$X_K$.

\begin{figure}
\includegraphics[width=3.5in]{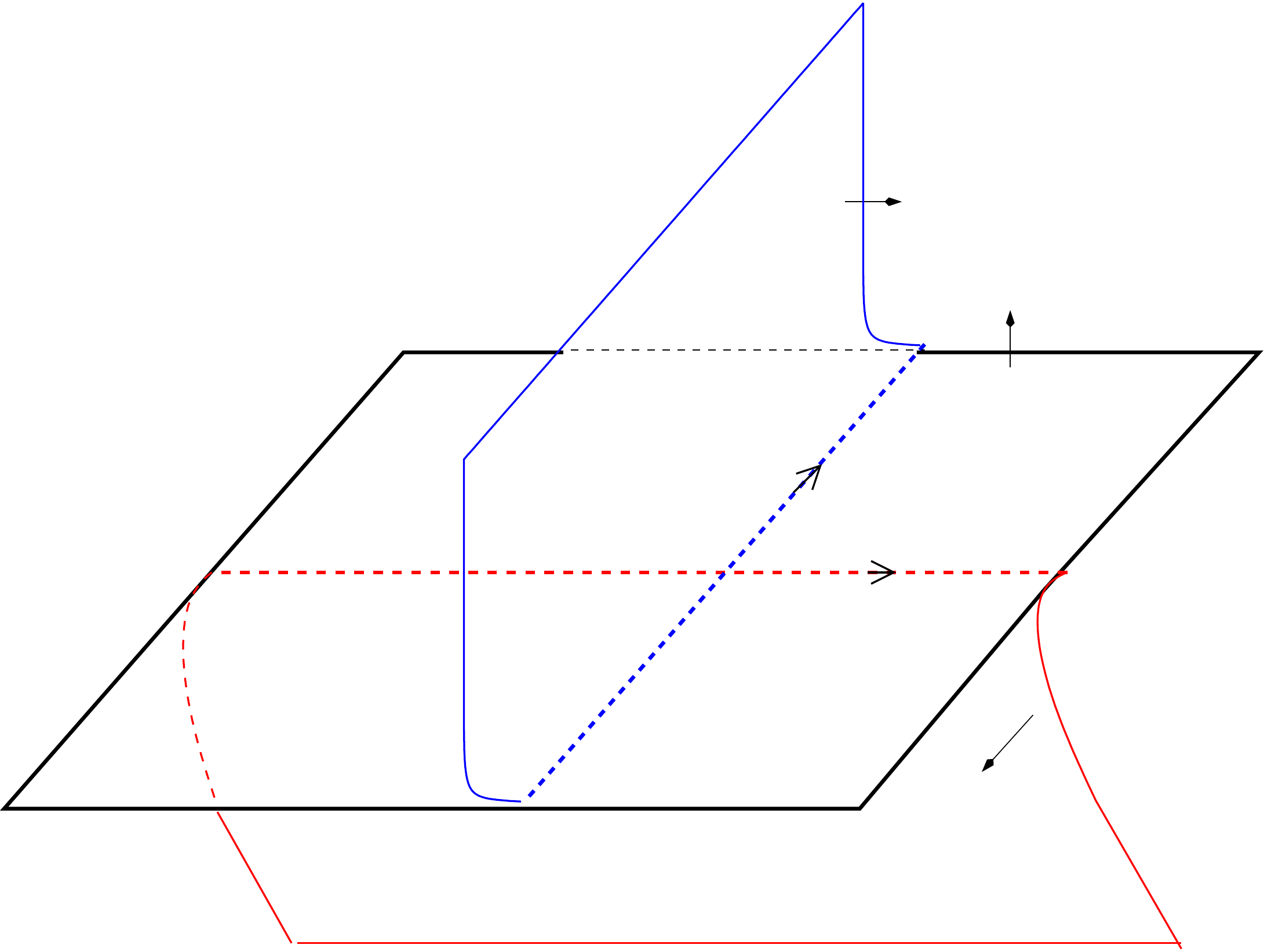}
\put(-85,95){$\beta_2$}
\put(-85,65){$\alpha_1$}
\put(-30,50){$D_\alpha^1$}
\put(-140,160){$D_\beta^2$}
\put(-175,110){\small \textcolor{magenta}{$w+t'$}}
\put(-55,110){\textcolor{magenta}{$w+t-t'$}}
\put(-100,150){\textcolor{magenta}{$t$}}
\put(-240,35){\textcolor{magenta}{$w$}}
\put(-105,35){\textcolor{magenta}{$w+t$}}
\put(-55,10){\textcolor{magenta}{$t'$}}
\caption{The branch locus of $\mathcal{B'}$ away from $\partial X_K$.}
\label{fig:branched1}
\end{figure}

\begin{figure}
\includegraphics[width=7in]{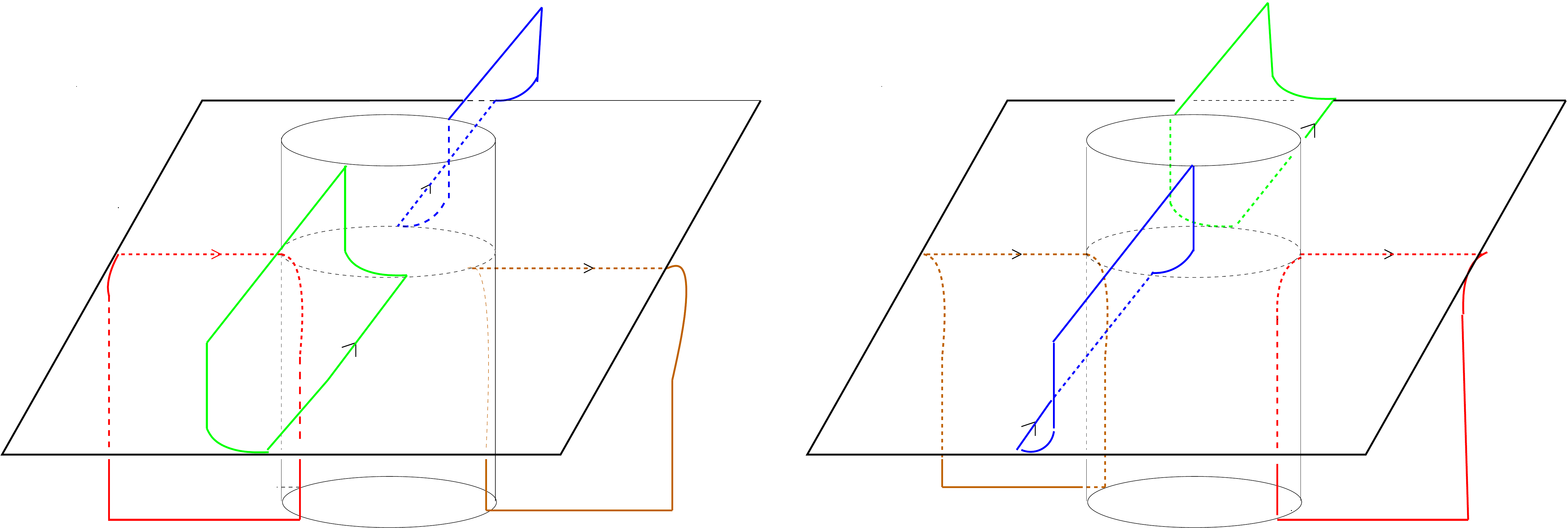}
\put(-28,5){$D_{\alpha}^1$}
\put(-450,92){$\alpha_1$}
\put(-335,88){$\alpha_2$}
\put(-72,92){$\alpha_1$}
\put(-385,105){$\beta_1$}
\put(-395,44){$\beta_2$}
\put(-77,127){$\beta_2$}
\put(-180,39){$\beta_1$}
\put(-198,92){$\alpha_2$}
\put(-490,5){$D_{\alpha}^1$}
\put(-287,5){$D_{\alpha}^2$}
\put(-220,5){$D_{\alpha}^2$}
\put(-123,160){$D_{\beta}^2$}
\put(-325,160){$D_{\beta}^1$}
\put(-370,160){\textcolor{magenta}{$\theta-t$}}
\put(-440,127){\textcolor{magenta}{$y+t'$}}
\put(-330,127){\textcolor{magenta}{$y+t+t'-\theta$}}
\put(-490,30){\textcolor{magenta}{$y$}}
\put(-345,30){\small \textcolor{magenta}{$y+t$}}
\put(-465,10){\textcolor{magenta}{$t'$}}
\put(-315,10){\textcolor{magenta}{$\theta - t'$}}
\put(-435,50){\textcolor{magenta}{$t$}}
\put(-180,127){\small \textcolor{magenta}{$x+t'$}}
\put(-170,117){\small \textcolor{magenta}{$-t$}}
\put(-40,127){\textcolor{magenta}{$x+t'$}}
\put(-105,150){\textcolor{magenta}{$t$}}
\put(-45,5){\textcolor{magenta}{$t'$}}
\put(-70,30){\textcolor{magenta}{$x$}}
\put(-240,30){\small \textcolor{magenta}{$x+\theta-t$}}
\put(-200,15){\small \textcolor{magenta}{$\theta-t'$}}
\put(-190,65){\textcolor{magenta}{$\theta-t$}}
\caption{The branch locus of $\mathcal{B'}$ at $\partial X_K$.
At left is a neighborhood of $x_0$, and at right is a neighborhood of $x_k$.}
\label{fig:branchedatboundary}
\end{figure}
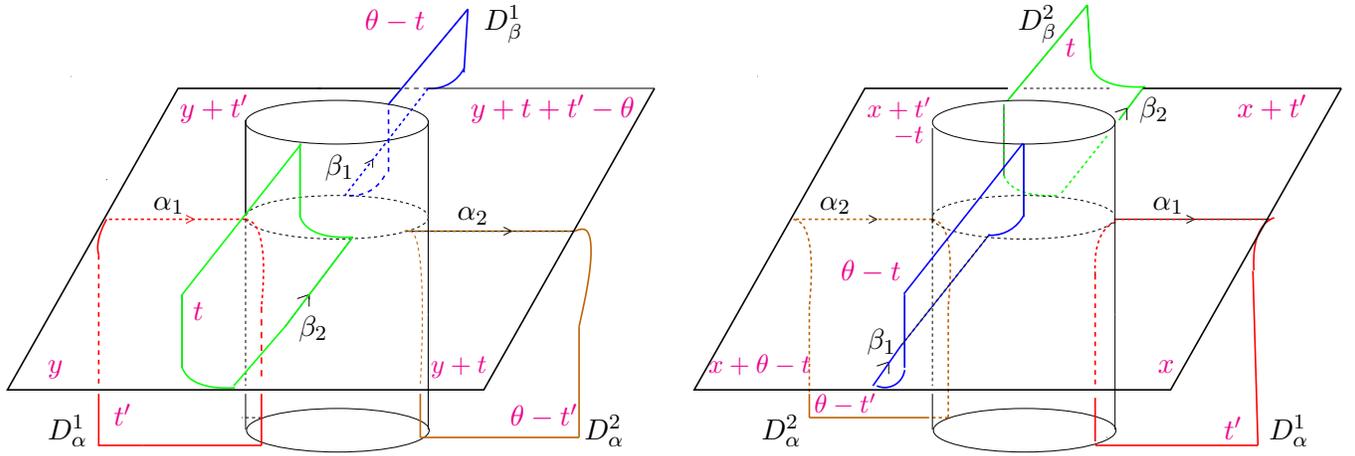

The boundary $\partial\mathcal{B}$ of the branched surface forms a co-oriented traintrack on $\partial X_K$ that is pictured in Figure \ref{fig:traintrack}, where sectors of weight zero are to be removed.
The weighting of  $\mathcal{B}$ induces a weighting of the traintrack as pictured. The co-orientation and weighting 
show that $\partial F$ consists of a 1-manifold whose intersection number on $\partial X_K$ with
a component of $\partial T'$ (slightly shifted) is $\theta$ where $\theta$ is the order of $K$ in
$H_1(L(p,q))$. Thus $F$ is a rational Seifert surface for $K$. Note that $S$ must be connected
since $K$ has order $\theta$, $\mathcal{B}$ is co-oriented, and any closed surface in $L(p,q)$
is trivial in homology (because $c_i=0$ for some $i$, one can constuct a simple closed curve in $X_K$ intersecting $\mathcal{B}$ in any particular $R_j$). 

We will use the following result to show that $F$ is taut in Theorem \ref{thm: seifert}:

\begin{lem}
\label{lem:euler} 
The Euler characteristic of $F$ is $\chi(F) = \theta - (1/4) \sum |R_i \cap \partial X_K| \cdot c_i$, where $| \cdot |$ denotes the number of components.
\end{lem} 

\begin{proof} The cell structure on the spine of $X_K$ induces a cell structure on $F$. Assign
an Euler measure to each cell giving its contribution to the number of faces, edges, vertices
of this cell decomposition of $F$. In particular assign a measure $1-e_i/4$ to each cell coming from 
$R_i$, where $e_i$
is the number of edges in $R_i$ and where  each component of $R_i \cap \partial X_K$ is counted as
an edge of $R_i$. Assign a measure of $1/2$ to each of the remaining cells, coming from $D^i_{\alpha}, D^i_{\beta}$.
Summing these Euler measures over $F$ gives the Euler characteristic of $F$. 
The Lemma then follows after noting that each cell of $F$ from $R_i$ contributes
$-|R_i \cap \partial X_K|/4$ to $\chi(F)$ and that there are a total of $2 \theta$ 
remaining cells coming from $D^i_{\alpha}$ and $D^i_{\beta}$.  
\end{proof}

\begin{figure}
\includegraphics[width=2in]{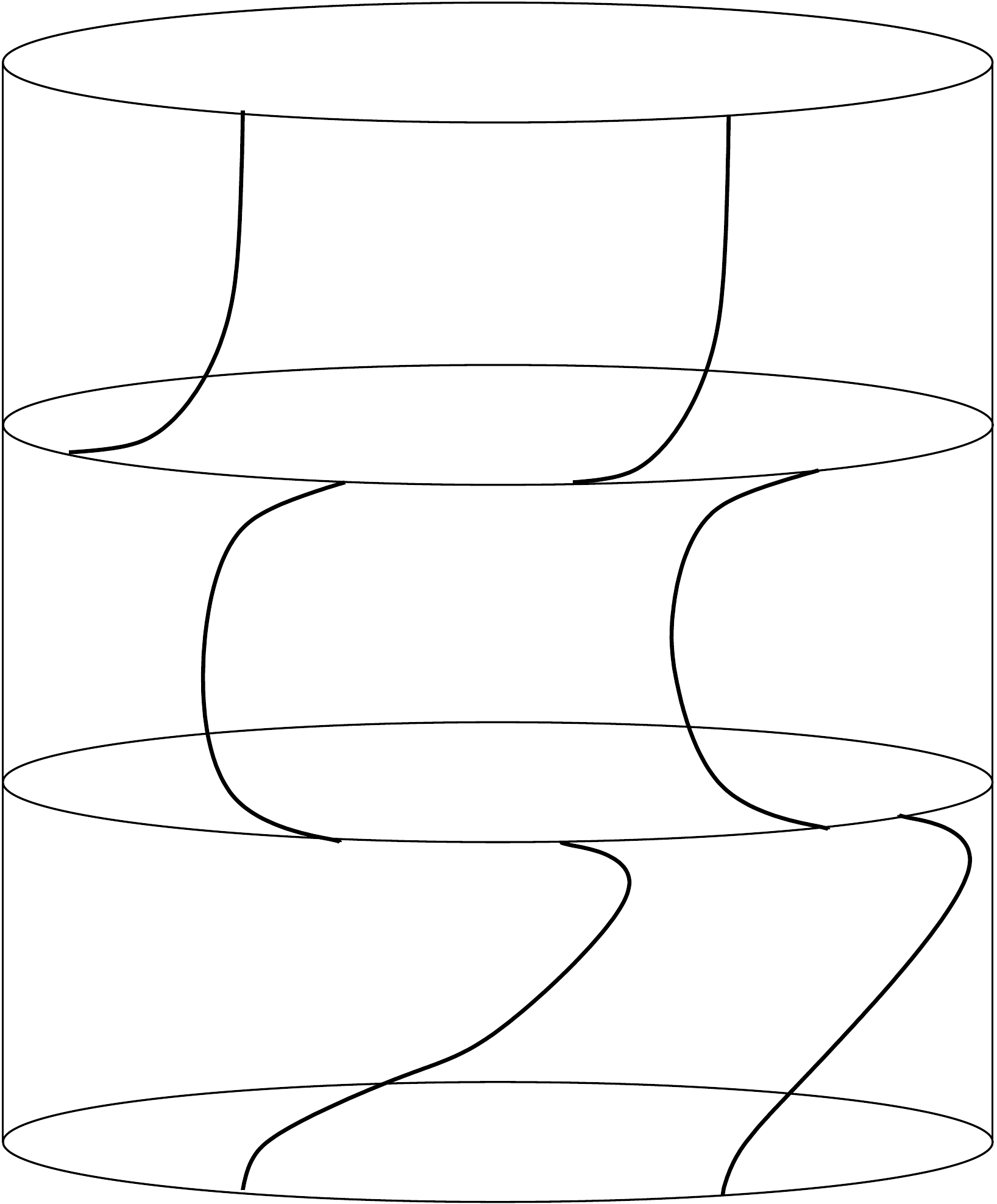}
\put(-100,140){\small $\theta-t'$}
\put(-30,80){\small $\theta-t$}
\put(-30,140){\small $t'$}
\put(-80,110){\small $x+t'$}
\put(-100,80){\small $t $}
\put(-40,110){\small $x$}
\put(-90,58){\small $y+t$}
\put(-120,58){\small $y$}
\put(-30,35){\small $t'$}
\put(-100,35){\small  $\theta-t'$}
\caption{The traintrack $\mathcal{B'} \cap \partial X_K$.}
\label{fig:traintrack}
\end{figure}
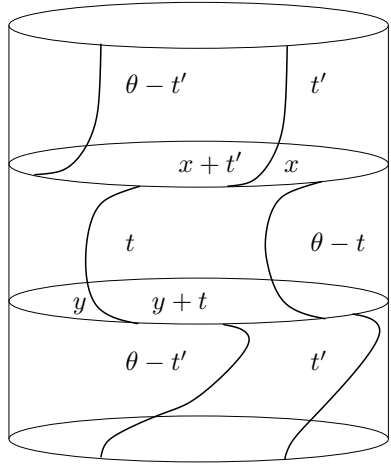


\section{Finding minima in $S(p,q,k)$.}
\label{sec:findingmins}
Let $\theta$ be the order of the simple knot $K(p,q,k)$. 
In this Section, we show that if $\theta$ is a harmonic of the pair $(p,q)$, then the sequence $S(p,q,k)$ attains its minimum value more than once. Indeed we show that the minima come in groups. 
As before, write out the steps of the Euclidean algorithm applied to the pair $(p,q)$:

\begin{eqnarray*}
p &=& d_1 q + r_1 \\
q &=& d_2 r_1 + r_2 \\
r_1 &=& d_3 r_2 + r_3 \\
&\vdots& \\
r_{n-2} &=& d_n r_{n-1} + 0.
\end{eqnarray*} 
Recall that $r_1 > \cdots > r_{n-1}=1$ are called the {\bf remainders} 
and $d_1, \dots, d_n$ the {\bf coefficients} associated with $(p,q)$.
Recall also that $S(p,q,k)$ is the sequence of non-negative integers
$g(x)$, $x =0,\dots,p-1$  where $g : \{0,\dots,p-1\} \to \bZ$ of Section~\ref{ss: notation} corresponds to Equation~(\ref{eq: jump function}), and that $g : \bZ \to \bZ$ is the $p$-periodic extension of that function to the integers.
Theorem~\ref{thm:lotsofzerosg} at the end of this section says that each minimum of $S(p,q,k)$ lies in a sequence of $\lceil r_i/ r_{i-1} \rceil$ minima that are $r_i$ steps apart, whenever $\theta$ divides the remainder $r_i$. These sequences of minima will be counted in the proof of Theorem~\ref{thm: minimizer}.
 
The argument in this section is geometric in nature.

\subsection{Continued fractions and matrices associated to $(p,q)$.}

In this subsection we define and identify properties of certain continued fractions and matrices associated to coefficients and remainders of $(p,q)$. 

For a sequence of positive integers $a_1, a_2, \dots, a_n$, we use the standard notation
\[
[a_1,a_2,\dots,a_n] = a_1 + \cfrac{1}{a_2 + \cfrac{1}{\ddots \\ + \cfrac{1}{a_n}}}
\]
for continued fractions.
When $d_1, \dots, d_n$ are the coefficients associated with $(p,q)$, we have that $p/q = [d_1,\dots,d_n]$.
For each $i=1,\dots,n$, we define the reduced fractions $p_i/q_i = [d_1,\dots,d_i]$.
Additionally, we define $q_{-1} = p_0=1$, $p_{-1} = q_0=0$, $r_{-1} = p$, and $r_0 = q$.
These values enjoy various properties:

\begin{lem}
\label{lem: cont fracs}
Let $x$ denote an indeterminate.
For all $i = 0,\dots, n-1$, we have
\begin{enumerate}
\item
$[d_1,\dots,d_i,x] = (p_i x + p_{i-1})/(q_i x + q_{i-1})$,
\item
$p_{i+1} = d_{i+1} p_i + p_{i-1}$,
\item
$q_{i+1} = d_{i+1} q_i + q_{i-1}$, and
\item
$p_{i+1} q_i - p_i q_{i+1} = (-1)^{i+1}$. 
\end{enumerate}
\end{lem}

\begin{proof}
The proof is by induction on $i$.
The assertions hold in the base case $i=0$ by direct inspection.
Assume next that $0 < i < n-1$ and all of the assertions hold for the index $i-1$. 
For the induction step, we first apply the identity $[d_1,\dots,d_i,x] = [d_1,\dots,d_{i-1},y]$, where $y = d_i + 1/x$.
By induction (assertion (1)), this equals $(p_{i-1}y + p_{i-2})/(q_{i-1}y + q_{i-2}) = ((d_i p_{i-1} + p_{i-2})x + p_{i-1})/((d_i q_{i-1} + q_{i-2})x + q_{i-1})$.
By induction (assertions (2) and (3)), this simplifies to the right side of assertion (1).
Next, set $x = d_{i+1}$.
We obtain $p_{i+1}/q_{i+1} =  (d_{i+1} p_i + p_{i-1})/(d_{i+1} q_i + q_{i-1})$.
Observe that $(d_{i+1} p_i + p_{i-1})q_i - (d_{i+1} q_i + q_{i-1})p_i = -(p_i q_{i-1} - p_{i-1} q_i) = (-1)^{i+1}$ by induction (assertion (4)).
Consequently, the numerator and denominator in this expression for $p_{i+1}/q_{i+1}$ are coprime, which gives assertions (2) and (3).
Assertion (4) now follows in turn.
This completes the induction step and finishes the proof.
\end{proof}

For each index $i$ for which the respective matrix entries are defined, we define
\[
B_i = \left( \begin{matrix} -d_i& (-1)^{i+1} \\ (-1)^i & 0 \end{matrix} \right), \quad 
C_i = \left( \begin{matrix} (-1)^i p_i & (-1)^{i+1}  q_i \\ -p_{i-1} & q_{i-1} \end{matrix} \right), \quad 
M_i = \left( \begin{matrix} (-1)^i p_i & r_i \\ -p_{i-1} & (-1)^i r_{i-1} \end{matrix} \right).
\]

\begin{lem}
\label{lem: matrix identities}
For all $i$ for which the matrices involved are defined, we have $C_{i+1} = B_{i+1} C_i$ and $M_{i+1} = B_{i+1} M_i$.
\end{lem}

\begin{proof}
The bottom rows of the two matrix identities are immediate.
The top-left entries of both identities follow from assertion (2) of Lemma \ref{lem: cont fracs}.
The top-right entry of the first identity follows from assertion (3) of Lemma \ref{lem: cont fracs}.
The top-right entry of the second identity follows from the defining equation $r_{i-1} = d_{i+1} r_i + r_{i+1}$ from the Euclidean algorithm.
\end{proof}

\begin{cor}
\label{cor: identities}
For all $i=0,\dots,n-1$, we have \\
\begin{inparaenum}
\item
$\det C_i = 1$,
\item
$M_i = C_i \left( \begin{matrix} 1 & q \\ 0 & p \end{matrix} \right)$,
\item
$(-1)^i r_i = p_i q - p q_i$, and
\item
$p = r_i p_{i-1} + p_i r_{i-1}$.
\end{inparaenum}
\end{cor}

\begin{proof}
Assertions (1) and (2) follow by induction on $i$, using Lemma \ref{lem: matrix identities} and the fact that $M_0 = \left( \begin{matrix} 1 & q \\ 0 & p \end{matrix} \right)$.
Assertion (3) follows from the top-right entry of the matrix identity in (2).
Lastly, $\det M_i = p$ follows from (1) and (2), and expanding the determinant results in (4).
\end{proof}

\subsection{Nice coordinates and a pattern in the minima of $S(p,q,k)$.}

Recall the description from Section~\ref{ss: domains} of a rational Seifert surface for $K=K(p,q,k)$ as a domain, $D$, in a Heegaard torus, $T$, of $L(p,q)$,
where $c_i$ denotes the coefficient (weight) of the region $R_i$ in $D$. 
As stated in Equation~(\ref{eq: gandc}) in that section, these weights as one travels along $\alpha$ in the Heegaard torus of $L(p,q)$ starting at $R_0$ are given by the sequence $S(p,q,k)$. This creates a geometric, two-dimensional point of view of $S(p,q,k)$. When the order $\theta$ of $K$ divides a remainder, $r_i$, of $(p,q)$, the matrix $C_i$ defined in the preceding section provides nice coordinates on $T$ revealing sequences of minima of $S(p,q,k)$ that are $r_i$ steps apart. These sequences, stated in terms of the corresponding $p$-periodic function $g$ on the integers, are captured in Theorem~\ref{thm:lotsofzerosg}, whose proof is the object of this subsection.

Assume that $\theta > 1$ is a harmonic of the pair $(p,q)$ and $1 \le i \le n-1$ is an index such that $\theta \, | \, r_i$.
Adapt coordinates on the Heegaard torus, $T$, of $L(p,q) = V_\alpha \cup_T V_\beta$ to this remainder as follows.
As in Section~\ref{ss: domains}, write $T$ as the torus $\bR^2 / \bZ^2$, endowed with the Euclidean metric.
Since $\det C_i = 1$, it descends to an orientation preserving diffeomorphism 
$\overline{C}_i : T \to T$.
By Corollary \ref{cor: identities}(2), this map takes the geodesics $\alpha$ and $\beta$, originally with slopes
 $0/1$ and $p/q$, to geodesics 
 with slopes $(-1)^{i+1} p_{i-1}/p_i$ and $(-1)^i r_{i-1}/r_i$, respectively.
 Specifically,
let $\alpha$ denote a geodesic curve of slope $s_1 = (-1)^{i+1}p_{i-1}/p_i$ and $\beta$ a geodesic curve of slope $s_2=(-1)^i r_{i-1}/r_i $ on $T$, chosen so that the intersection point $x_0 \in \alpha \cap \beta$ is the image of $(0,0)$ on $T$.
We picture $T$ using the fundamental domain $[0,1] \times [0,1]$, with $V_\alpha$ sitting below the plane of the picture and $V_\beta$ above.
Since $\theta$ divides both $r_i$ and $p$, and since $p= r_i p_{i-1} + p_i r_{i-1}$ (Corollary \ref{cor: identities}(4)) and $\gcd(r_{i-1},r_i)=1$ (from the Euclidean algorithm and the fact that $p,q$ are coprime), it follows that $\theta$ divides $p_i$. As $\gcd(r_{i-1},r_i)=1$ and $\gcd(p_{i-1},p_i)=1$ (Lemma~\ref{lem: cont fracs}(4)), 
$s_1,s_2$ both have denominators divisible by $\theta$.
Let $\gamma$ denote the vertical curve through $x_0$.
The triple intersection $\alpha \cap \beta \cap \gamma$ therefore contains $\gcd(r_i,p_i) = d \theta$ points, for some positive integer $d$.
These points are equally spaced along $\alpha$, so they are the points $x_i$ whose index $i$ is a multiple of $p / d \theta$.
In particular, $x_k \in \alpha \cap \beta \cap \gamma$.
(Note that $k \pmod p$ has order $\theta$.)

\noindent
{\em Example.}
Running the Euclidean algorithm on the pair $(p,q) = (15,4)$ yields the sequence of coefficients $(d_1,d_2,d_3) = (3,1,3)$ and remainders $(r_1,r_2) = (3,1)$.
We calculate $p_1/q_1 = [3] = 3/1$, $p_2/q_2 = [3,1] = 4/1$, and $p_3/q_3 = [3,1,3] = 15/4$, from which we obtain the matrices
$C_1 = \left( \begin{matrix} -p_1 & q_1 \\ -p_0 & q_0 \end{matrix} \right) = \left( \begin{matrix} -3 & 1 \\ -1 & 0 \end{matrix} \right)$ and 
$M_1 = \left( \begin{matrix} -p_1 & r_1 \\ -p_0 & - r_0 \end{matrix} \right) = \left( \begin{matrix} -3 & 3 \\ -1 & -4 \end{matrix} \right)$.
In the coordinates on $T$ given by the matrix $C_1$, the curves $\alpha$ and $\beta$ have slopes $p_0/p_1 = 1/3$ and $-r_0 / r_1 = -4/3$, respectively. $\alpha$ is oriented in the direction $[-3,-1]$, $\beta$ in the direction $[3,-4]$. See Figure \ref{fig: coords}.
The knot $K(15,4,5)$ has order $\theta=3$ which divides $r_1$. We see that $x_5$ lies on the vertical axis. 

In general, within the fundamental domain $[0,1] \times [0,1]$, we see one of the following two pictures under the transformation by $\overline{C}_i$, depending on the parity of $i$:

(1) {\em If $i$ is odd,} $\alpha_1$ is the part of $\alpha$ incident to the bottom left corner, $\alpha$ is oriented towards that corner, and $\alpha_2$ is incident to the top right corner; similarly, $\beta_1$ is the part of $\beta$ leaving the top left corner, $\beta$ is oriented away from that corner, and $\beta_2$  is incident to the bottom right corner.
See Figures~\ref{fig: coords} and \ref{fig:deltaiodd} .
Since $\alpha$ has slope $0 < s_1 < 1$ and $\beta$ has slope $s_2 < -1$, there exists a triangle $\Delta_j \subset T$ with one side determined by $\gamma$ and the other two by segments of $\alpha_j$ and $\beta_j$, $j=1,2$. 
Both are indicated in the Figures.
Note that $\alpha_j \cap \Delta_j$ intersects $\beta$ in $r_i+1$ points (including $x_0$).
The triangle $\Delta_j$ is embedded away from a duplicated vertex at $x_0$.
Its other vertex is $x_{(-1)^j r_i}$, where the subscript is taken mod $p$.
The condition that $s_2 < -1$ ensures that $\Delta_1 \cap \Delta_2 = \gamma$.
It follows that $P = \Delta_1 \cup \Delta_2$ is an embedded parallelogram away from the duplicated vertex at $x_0$; in particular, it has three distinct vertices $x_{-r_i}, x_0, x_{r_i}$.

\begin{figure}
\includegraphics[width=2in]{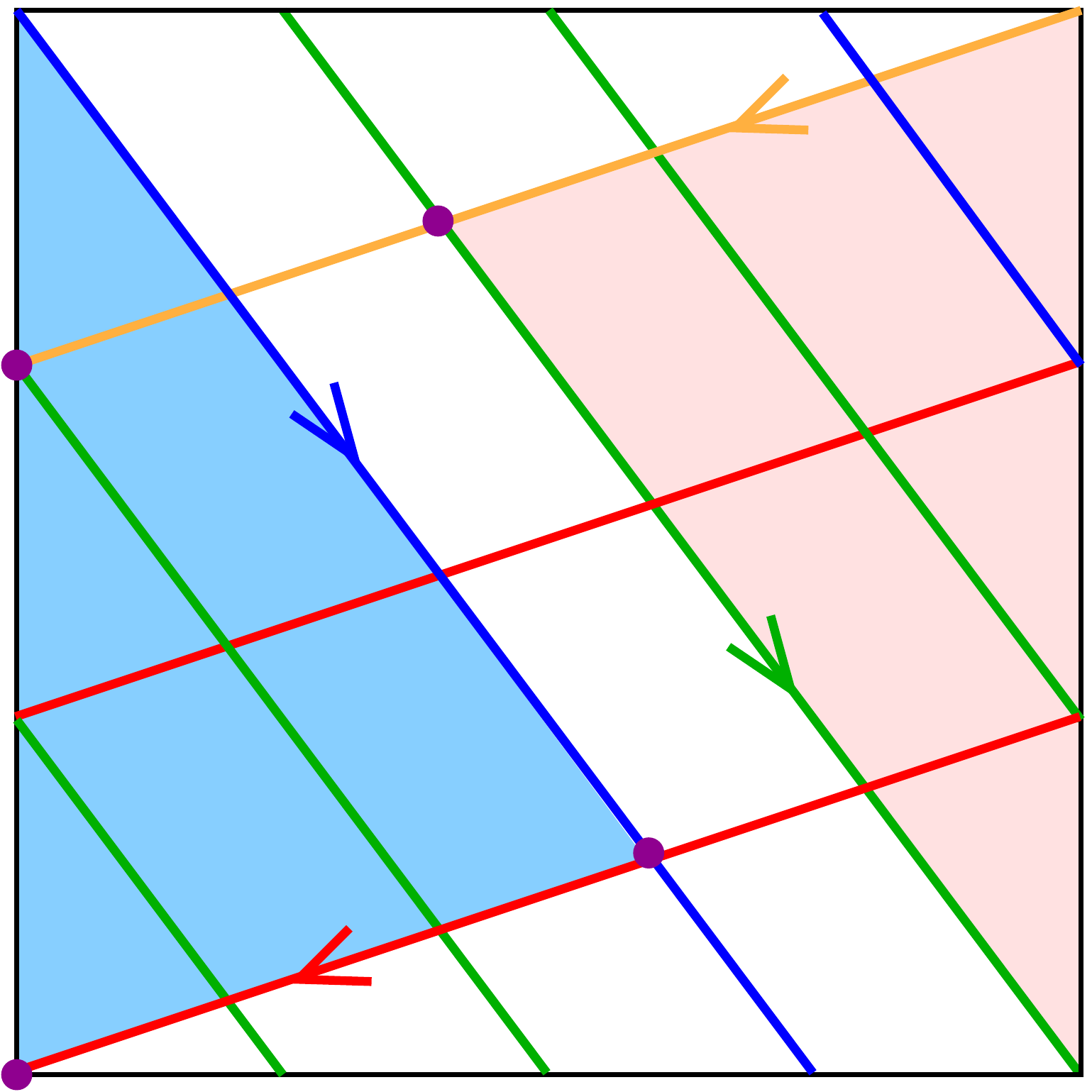}
\put(-160,0){$x_0$}
\put(-160,90){$x_5$}
\put(-92,105){$x_3$}
\put(-50,25){$x_{-3}$}
\put(-65,98){$\Delta_2$}
\put(-93,40){$\Delta_1$}
\caption{Heegaard diagram of the lens space $L(15,4)$ coming from $C_1$.
It has attaching curves $\alpha = \alpha_1 \cup \alpha_2$ of slope $1/3$ and $\beta = \beta_1 \cup \beta_2$ of slope $-4/3$.
The Heegaard torus meets the simple knot $K(15,4,5)$ in points $x_0, x_5 \in \alpha \cap \beta \cap \gamma$, where $\gamma$ has slope $1/0$.
The triangles $\Delta_1$ and $\Delta_2$ are indicated as well.
}
\label{fig: coords}
\end{figure}

\begin{figure}
\includegraphics[width=2in]{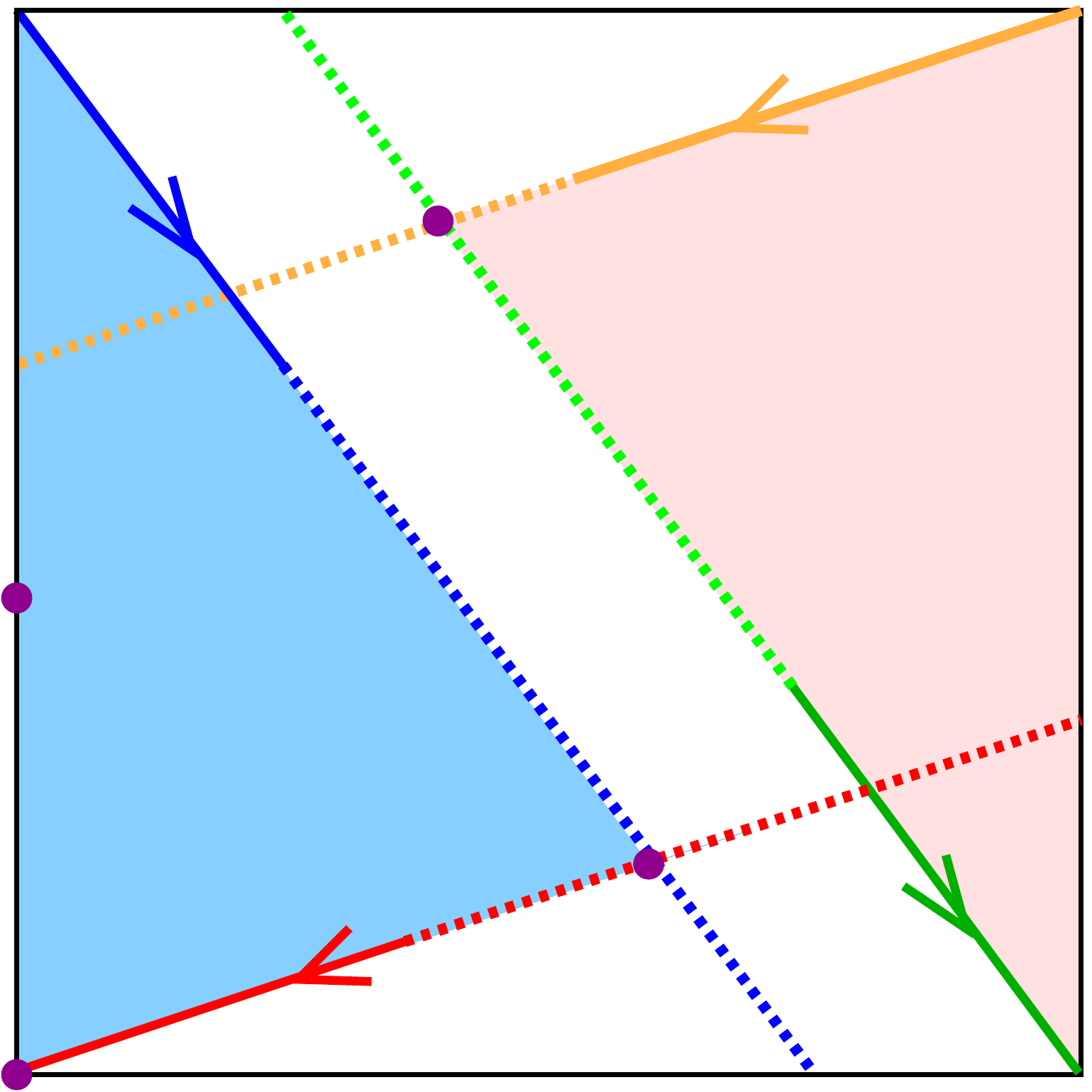}
\put(-20,130){$\alpha_2$}
\put(-130,130){$\beta_1$}
\put(-13,18){$\beta_2$}
\put(-130,14){$\alpha_1$}
\put(-92,105){$x_{r_i}$}
\put(-50,25){$x_{-r_i}$}
\put(-160,0){$x_0$}
\put(-160,63){$x_k$}
\put(-40,80){$\Delta_2$}
\put(-120,60){$\Delta_1$}
\caption{The curves $\alpha_1, \alpha_2,\beta_1,\beta_2$ in the coordinates coming from $C_i$ when $i$ is odd, along with the triangles $\Delta_1,\Delta_2$.
}
\label{fig:deltaiodd}
\end{figure}

\begin{figure}
\includegraphics[width=2in]{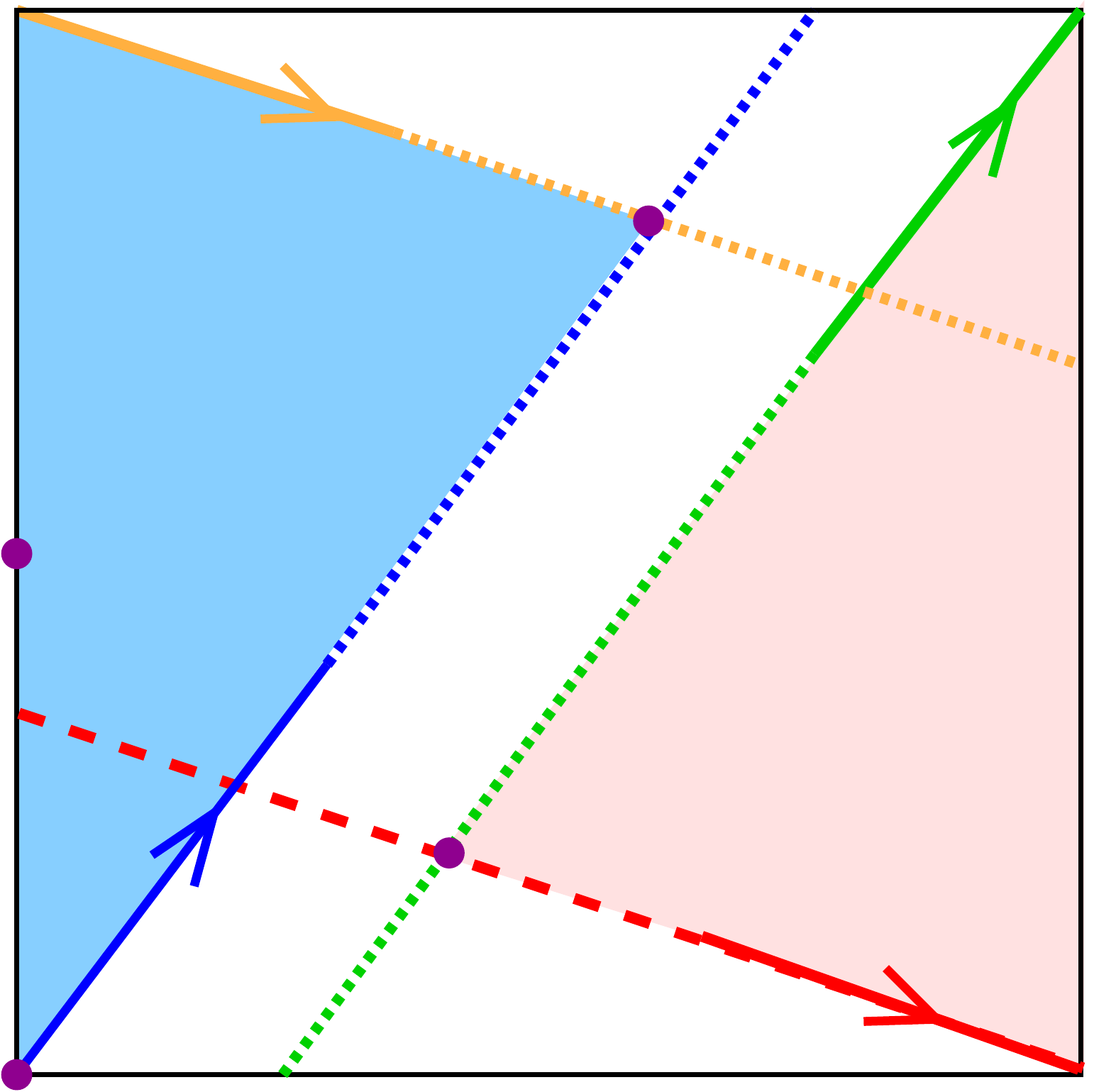}
\put(-22,134){$\beta_2$}
\put(-95,132){$\alpha_2$}
\put(-18,17){$\alpha_1$}
\put(-125,17){$\beta_1$}
\put(-76,112){$x_{r_i}$}
\put(-77,32){$x_{-r_i}$}
\put(-160,0){$x_0$}
\put(-160,63){$x_k$}
\put(-25,60){$\Delta_2$}
\put(-120,80){$\Delta_1$}
\caption{The curves $\alpha_1, \alpha_2,\beta_1,\beta_2$ in the coordinates coming from $C_i$ when $i$ is even, along with the triangles $\Delta_1,\Delta_2$. 
}
\label{fig:deltaieven}
\end{figure}

\begin{figure}
\includegraphics[width=2in]{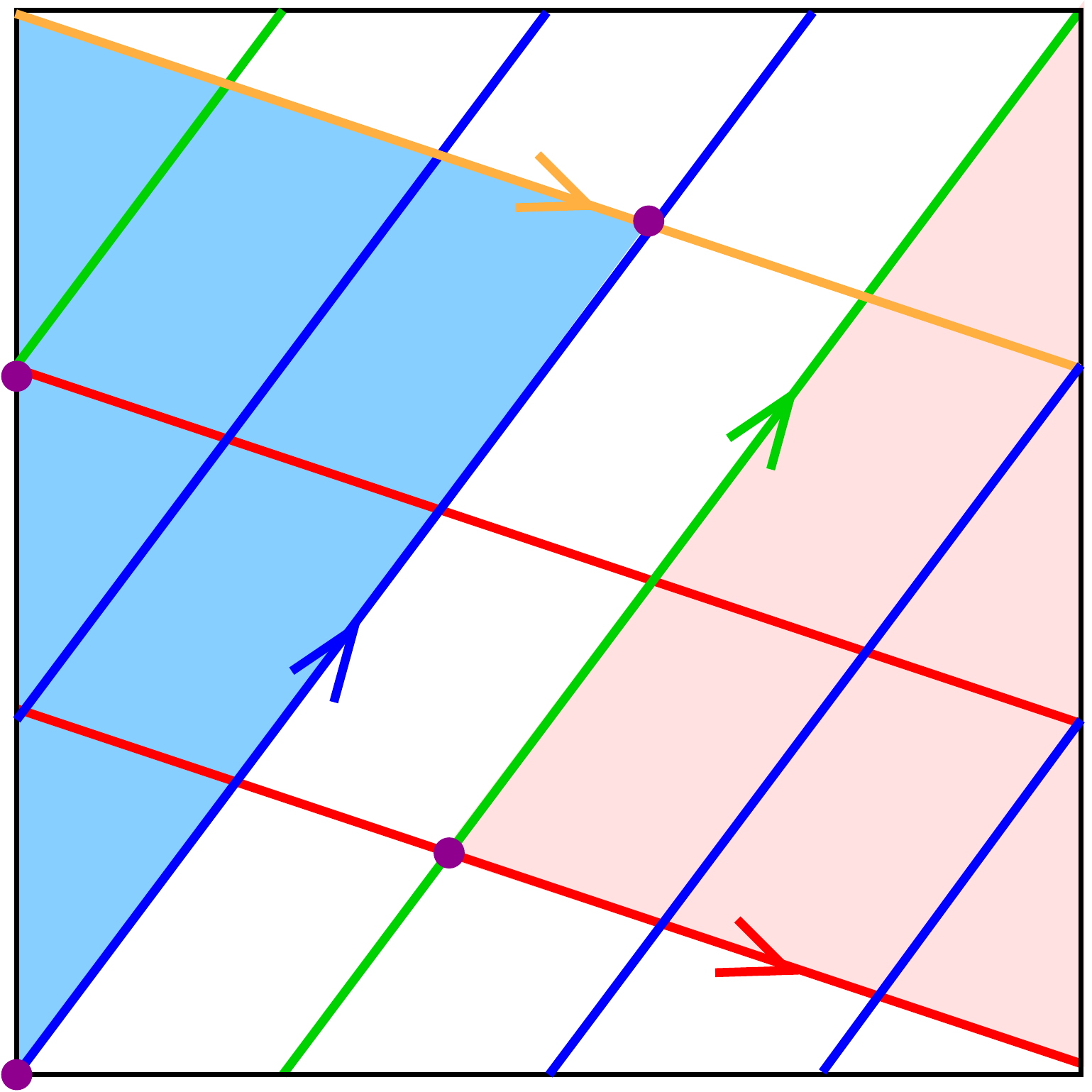}
\put(-160,0){$x_0$}
\put(-160,90){$x_5$}
\put(-76,110){$x_3$}
\put(-77,32){$x_{-3}$}
\put(-18,62){$\Delta_2$}
\put(-118,65){$\Delta_1$}
\caption{Heegaard diagram of the lens space $L(15,11)$ coming from $C_2$.
It has attaching curves $\alpha = \alpha_1 \cup \alpha_2$ of slope $-1/3$ and $\beta = \beta_1 \cup \beta_2$ of slope $4/3$.
The Heegaard torus meets the simple knot $K(15,11,5)$ in points $x_0, x_5 \in \alpha \cap \beta \cap \gamma$, where $\gamma$ has slope $1/0$.
The triangles $\Delta_1$ and $\Delta_2$ are indicated, as well.
}
\label{fig:coords15-11}
\end{figure}

(2) {\em If $i$ is even,} $\beta_1$ is the part of $\beta$ leaving the origin from the bottom left corner, $\beta$ is oriented away from that corner, and $\beta_2$ is incident to the top right corner; similarly, $\alpha_2$ is the part of $\alpha$ leaving the top left corner, $\alpha$ is oriented away from that corner, and $\alpha_1$ is incident to the bottom right corner. 
Since $\alpha$ has slope $-1 < s_1 < 0$ and $\beta$ has slope $s_2 > 1$, there exists a triangle $\Delta_j \subset T$ with one side determined by $\gamma$ and the other two by segments of $\alpha_l$ and $\beta_j$, $\{j,l\}=\{1,2\}$. 
See Figure~\ref{fig:deltaieven} and Figure~\ref{fig:coords15-11}.  
Note that $\alpha_l \cap \Delta_j$, $\{j,l\}=\{1,2\}$, intersects $\beta$ in $r_i+1$ points (including $x_0$).
The triangle $\Delta_j$ is embedded away from a duplicated vertex at $x_0$.
Its other vertex is $x_{(-1)^{j+1}r_i}$.
The condition on the slopes ensures that $\Delta_1 \cap \Delta_2 = \gamma$.
It follows that $P = \Delta_1 \cup \Delta_2$ is an embedded parallelogram away from the duplicated vertex at $x_0$; in particular, it has three distinct vertices $x_{-r_i}, x_0, x_{r_i}$.

Let $R_m$ be a region of $T$ with coefficient 0 in the domain $D$ of $K$.
We find another region with coefficient 0, as follows.
Translate $P=\Delta_1 \cup \Delta_2$ to a parallelogram $P' = \Delta_1' \cup \Delta_2'$ whose duplicated vertex occurs in the interior of $R_m$ and does
not lie on $\gamma$.
Its other vertices occur in the interiors of $R_{m-r_i}$ and $R_{m+r_i}$
The regions $R_{m-r_i}$, $R_m$, $R_{m+r_i}$ are distinct.
Recall that $c_{m \pm r_i}$ denotes the coefficient on $R_{m \pm r_i}$ for the domain $D$.

\begin{prop}
\label{prop: zero coeff}
If $i$ is odd, then $c_{m-r_i}$ vanishes unless the interior of $\Delta_1'$ contains $x_0$, and $c_{m+r_i}$ vanishes unless  the interior of $\Delta_2'$ contains $x_0$.
If $i$ is even, then $c_{m+r_i}$ vanishes unless the interior of $\Delta_1'$ contains $x_k$, and $c_{m-r_i}$ vanishes unless the interior of $\Delta_2'$ contains $x_k$.

\end{prop}

The proof of Proposition~\ref{prop: zero coeff} relies on two lemmas concerning linking numbers.
Let $\gamma'$ denote the translate of $\gamma$ contained in $P'$. 
It is a vertical geodesic curve on $T$ disjoint from $x_0$ and $x_k$, so it describes a knot in $L(p,q)$ disjoint from $K$.
Orient $\gamma'$ in the direction of increasing height.
Recall that the linking number $lk(K,\gamma')$ between $K$ and $\gamma'$ can be calculated as $1/n$ times the intersection number between $\gamma'$ and a 2-chain with boundary $n \cdot K$, for any $n \in \bZ$, $n \ne 0$.

\begin{lem}
\label{lem: linking 1}
The linking number between $K$ and $\gamma'$ equals zero.
\end{lem}

\begin{proof}
The curve $\gamma'$ is isotopic in the complement of $K$ to a copy $\gamma''$ of $\gamma$ displaced just to one side of $\gamma$, so $lk(K,\gamma') = lk(K,\gamma'')$.
The curve $\gamma''$ intersects $\beta$ in $r_i$ points, all with the same sign, since $\gamma''$ has slope $1/0$, $\beta$ has slope of denominator $r_i$, and both are geodesics.
(The precise sign will depend on the parity of $i$, but it will not matter.)
Recall the definition of $l$ and $t$ of \eqref{eq: l theta t} of Section~\ref{ss: notation}. As we are dealing with Euclidean geodesics, the points of $\gamma \cap \beta$ are equally spaced along $\beta$. The subarc $\beta_1$ has its endpoints $x_0,x_k$ on $\gamma$.  
Furthermore, $\beta_1$ has length $l/p = t / \theta$ times that of $\beta$ ($\beta_1$ is cut into $l = [q' k]_p$ subarcs and $\beta$ into $p$ subarcs by $\alpha$). Similarly, the arc $\beta_2$ has its endpoints on $\gamma$ and has length $(\theta-t)/\theta$ times that of $\beta$.
It follows that $\gamma''$ meets $\beta_1$ in $(t/ \theta) r_i$ points and $\beta_2$ in $((\theta - t)/ \theta) r_i$ points of intersection, all with the same sign.
Isotope $\gamma''$ slightly into $V_\beta$.
Its intersection with the 2-chain $C$ of \eqref{eq: 2-chain} in Section~\ref{ss: domains} comes from 
$(t/ \theta) r_i$ points of intersection with $D^1_\beta$, all of one sign, and $((\theta - t)/ \theta) r_i$ points of intersection with $D^2_\beta$, all of the opposite sign: the reason the two signs are opposite is that $\beta_1$ is oriented as $- \del D^1_\beta$, while $\beta_2$ is oriented as $\del D^2_\beta$.
As the weight of $D^1_\beta$ in $C$ is $(\theta-t)$ and of $D^2_\beta$ is $t$, the intersection number of $\gamma''$ and $C$ is $\pm (-(t/ \theta) r_i \cdot (\theta - t) + ((\theta-t)/ \theta) r_i  \cdot t) = 0$.
As $\del C = \theta \cdot K$, it follows that $lk(\gamma',K) = lk(\gamma'',K) = (1/\theta) \cdot 0  = 0$, as claimed.
\end{proof}

Let $\gamma_j$ denote the union of the non-vertical sides of $\Delta_j'$.
It is a simple closed curve on $T$ in the complement of $\alpha \cap \beta$.
Orient $\gamma_j$ in the direction of decreasing   
height. 

\begin{lem}
\label{lem: linking 2}
The linking number between $K$ and $\gamma_j$ equals $(1/\theta)(-1)^{i+j-1} c_{m+(-1)^{i+j-1} r_i}$ for $j \in \{1,2\}$.

\end{lem}

\begin{proof}
Isotope $\gamma_j$ by pushing the interior of its side of slope $s_1$ slightly into $V_\alpha$ and the interior of its side of slope $s_2$ slightly into $V_\beta$.
The intersection between the isotoped copy of $\gamma_j$ and the spine of $L(p,q)$ consists of one point of $R_m$ and one point of  $R_{m+(-1)^{i+j-1} r_i}$.
As the normal orientation to $C$ at $T$ points from $V_\alpha$ to $V_\beta$, a check shows that $\gamma_j$ meets $R_{m+(-1)^{i+j-1} r_i}$ in a point of sign $(-1)^{i+j-1}$. Figure~\ref{fig:linkingiodd} may aid in checking this in the case that $i$ is odd, ignoring the features incident with $x_k$.
Since $R_m$ has coefficient 0, it follows that the intersection number of $\gamma_j$ with $C$ is $(-1)^{i+j-1} c_{m+(-1)^{i+j-1} r_i}$.
Hence $lk(\gamma_j,K)$ equals the value stated in the Lemma.
\end{proof}

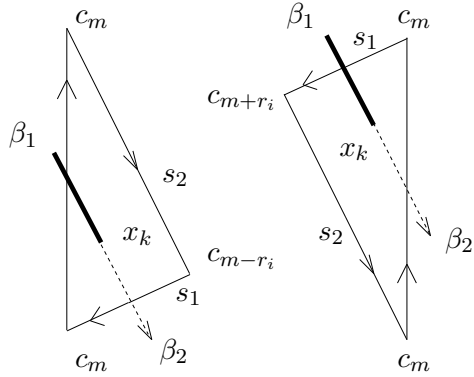
\begin{figure}
\includegraphics[width=2in]{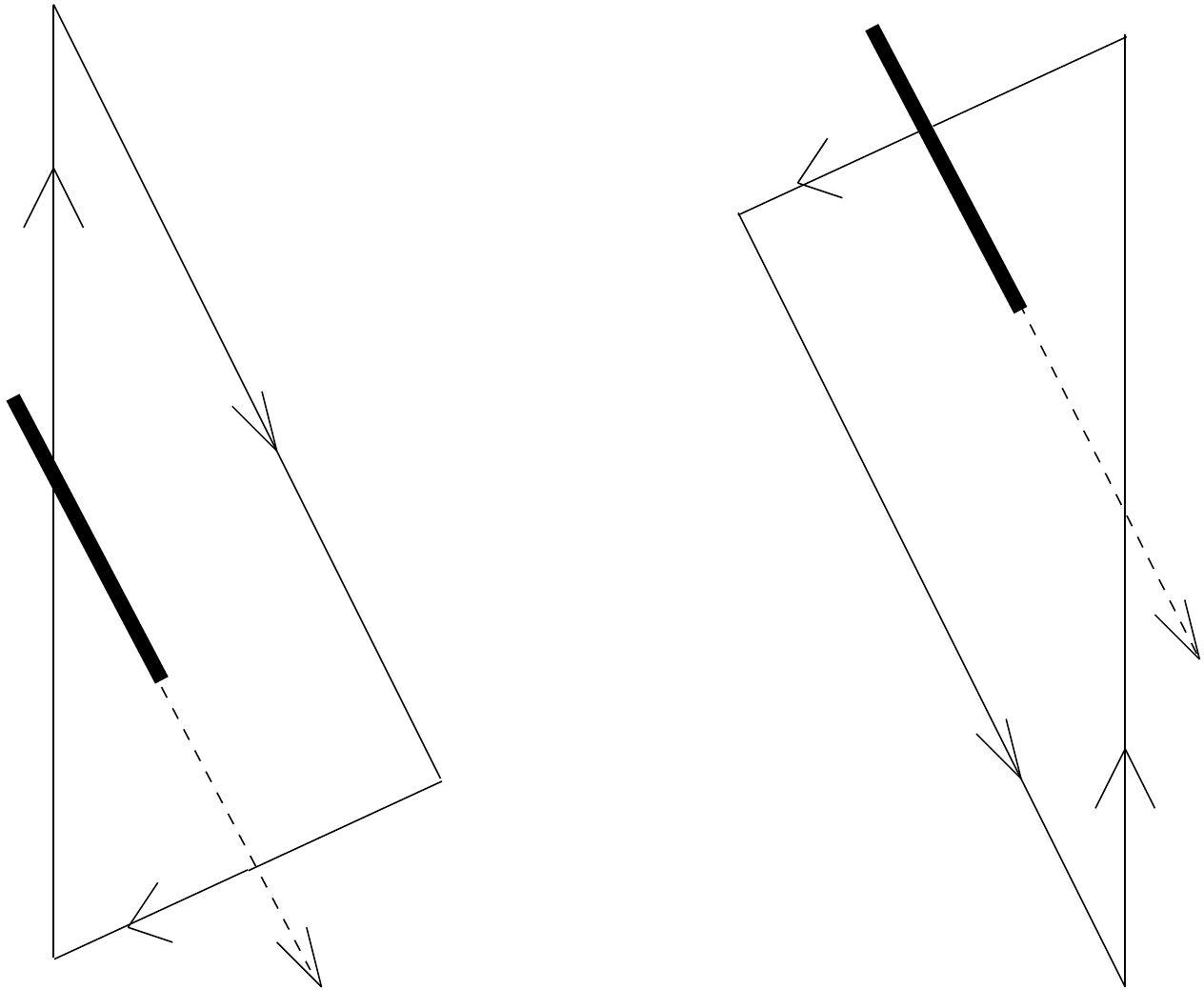}
\put(-160,75){$\beta_1$}
\put(-55,120){$\beta_1$}
\put(5,35){$\beta_2$}
\put(-103,-7){$\beta_2$}
\put(-30,112){$s_1$}
\put(-135,120){$c_m$}
\put(-135,-10){$c_m$}
\put(-13,120){$c_m$}
\put(-43,38){$s_2$}
\put(-13,-10){$c_m$}
\put(-97,15){$s_1$}
\put(-85,90){$c_{m+r_i}$}
\put(-85,30){$c_{m-r_i}$}
\put(-102,60){$s_2$}
\put(-117,38){$x_k$}
\put(-35,70){$x_k$}
\caption{ $\Delta_1'$ and $\Delta_2'$ pictured for $i$ odd and the case that $x_k$ is contained in the interior of one or the other.
Observe that $K$ is oriented into the plane of the diagram at $x_k$.
}
\label{fig:linkingiodd}
\end{figure}

\begin{proof}
[Proof of Proposition \ref{prop: zero coeff}]
Recall that the weights on the regions are all nonnegative.
We first assume $i$ is odd.
Assume for contradiction that $c_{m-r_i} > 0$ and $\Delta_1'$ does not contain $x_0$ in its interior.
If $\Delta_1'$ does not contain $x_k$ either, then $\Delta_1'$ guides an isotopy of $\gamma_1$ to $-\gamma'$ in the complement of $K$.
Hence the linking number of $\gamma_1$ with $K$ is zero by Lemma~\ref{lem: linking 1}, but then Lemma~\ref{lem: linking 2} contradicts our assumption that $c_{m-r_i} \ne 0$.  
Thus, we may assume that $\Delta_1'$ contains $x_k$ in its interior.
The triangle $\Delta_1'$ shows that $\gamma_1 + \gamma'$ is homologous in $T - K$ to a meridian of $K$ that orients clockwise around $x_k$. 
See Figure~\ref{fig:linkingiodd}.
Since $K$ orients from $V_{\beta}$ to $V_{\alpha}$ at $x_k$, the linking number of this meridian with $K$ is $1$.
(Note that the sign convention for linking number uses the right-hand rule and the convention that $V_\alpha$ is below the page.)
Along with Lemma~\ref{lem: linking 1}, this shows that $lk(\gamma_1,K) = 1$.
But then Lemma~\ref{lem: linking 2} gives $c_{m-r_i}=-\theta < 0$, a contradiction.

The case in which $i$ is odd and we assume that $c_{m+r_i} > 0$ is identical, mutatis mutandis.
In particular, $\gamma_2 + \gamma'$ is homologous to a meridian of $K$ with the counterclockwise orientation.
It has linking number $-1$ with $K$, and we steer to the same contradiction that $c_{m+r_i} = -\theta < 0$. 
This proves part one of the Proposition. 

The case in which $i$ is even follows similarly.
The only change is that we consider the possibility that $x_0$ is contained in the interior of $\Delta_1'$ or $\Delta_2'$, noting that $K$ is oriented from $V_{\alpha}$ to $V_{\beta}$ at $x_0$.
\end{proof}

Proposition~\ref{prop: zero coeff} shows that moving up or down along $\alpha$ by $r_i$ units finds new weight zero regions until the $\Delta_j'$ capture $x_0$ or $x_k$.
This allows us to find a sequence of $\lceil r_{i-1} / r_{i} \rceil$ weight zero regions about any given weight zero region.  

\begin{prop}\label{prop:lotsofzeros}  Let $\delta$ be the distance along $\alpha$ between consecutive points of $\alpha \cap \beta$. Assume $R_m$ has 
coefficient zero. Then $R_m$ lies in a sequence of $\lceil r_{i-1} / r_{i} \rceil$ regions along $\alpha$ in steps of length $\delta r_i$ 
all of which have weight zero.  That is, there is an integer $M$ and a collection of indices $I_m=\{ M+h r_i  \, | \,  h = 0, \dots, \lfloor r_{i-1} / r_{i} \rfloor \}$ with the properties that $m \in I_m$ and if $j \in I_m$ then $c_{j}=0$ (subscripts taken mod $p$). 
\end{prop}  

\begin{proof}
Given $R_m$ whose weight is zero, 
let $M$ be an integer mod $p$ such that  $\{R_{M+h r_i} | h = 0, \dots, N \}$ is a maximal sequence  in steps of length $\delta r_i$ along $\alpha$ of 
regions whose weight is zero
 and which includes $R_m$. We will show that $N \geq  
 \lfloor r_{i-1} / r_{i} \rfloor$.
 
 Let $\rho:\widetilde T \to T$ be the universal cover of $T$ with the lifted metric.
 Lift the sequence $\{ R_{M+h r_i} \, | \, h = 0, \dots, N \}$ to a sequence
 in $\widetilde T$ along a component $\widetilde \alpha$ of $\rho^{-1}(\alpha)$. Denote these regions by the same  name.
For each $h = 0, \dots, N$, pick a point $O_h$ in the interior of $R_{M+h r_i}$ in $\widetilde T$ such that the distance between $O_h$ and $O_{h+1}$ is $\delta r_i$ and $\overline{O_hO_{h+1}}$ is parallel to $\widetilde \alpha$. 

For each $h = 0, \dots, N$, let $O_h'$ be the point in $\widetilde T$ that is vertically one unit up from $O_h$. Thus $O_h'$ and $O_h$ project to the same point in $T$.
Lift the triangles $\Delta_1, \Delta_2$ to $\widetilde T$, and let $P$ be a lifted parallelogram $\Delta_1 \cup \Delta_2$.
For each $h = 0, \dots, N$, let $P_h$ be the parallelogram $P$ in $\widetilde T$ translated so $\gamma=\overline{O_hO_{h+1}}$  and let $\Delta_1^h, \Delta_2^h$ be the corresponding constituent triangles of $P_h$, which are translates of $\Delta_1, \Delta_2$.
The union of the $P_h$ over $h = 0, \dots, N$ is a parallelogram, $\mathcal{P}$.

We assume that $i$ is odd. The argument for $i$ even is analogous.

By Proposition~\ref{prop: zero coeff} and the maximality of 
$\{ R_{M+h r_i} | h = 0, \dots, N \}$, $\Delta_1^0, \Delta_2^N$ must 
each contain a (single) preimage of $x_0$. Call these points $X_0,X_N$
respectively. For example, see Figure~\ref{fig:OmorethanN}.  

Assign $x,y-$coordinates to points in $\widetilde T$. For a point $P$ in $\widetilde T$,
let $(x(P),y(P))$ be its coordinates. 

Note that $y(X_N) \leq y(X_0)$.
Otherwise, the horizontal line through $X_N$ meets the line segment $\overline{O_0O_0'}$ and at a point higher than the horizontal line through $X_0$ meets it ($\overline{O_NO_N'},\overline{O_0O_0'}$ have the same length and $\alpha$ has postive slope).
As $X_0, X_N$ are lifts of $x_0$, we reach the contradiction
$1=|\overline{O_0O_0'}|> y(X_N)-y(X_0) \geq 1$, where $| \cdot |$ denotes length. See Figure~\ref{fig:OmorethanN}.  

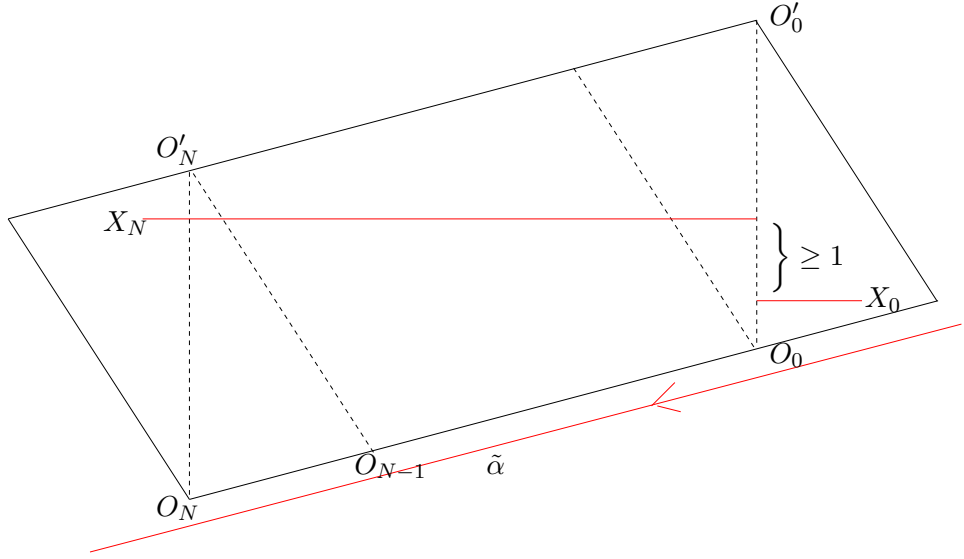
\begin{figure}
\includegraphics[width=5in]{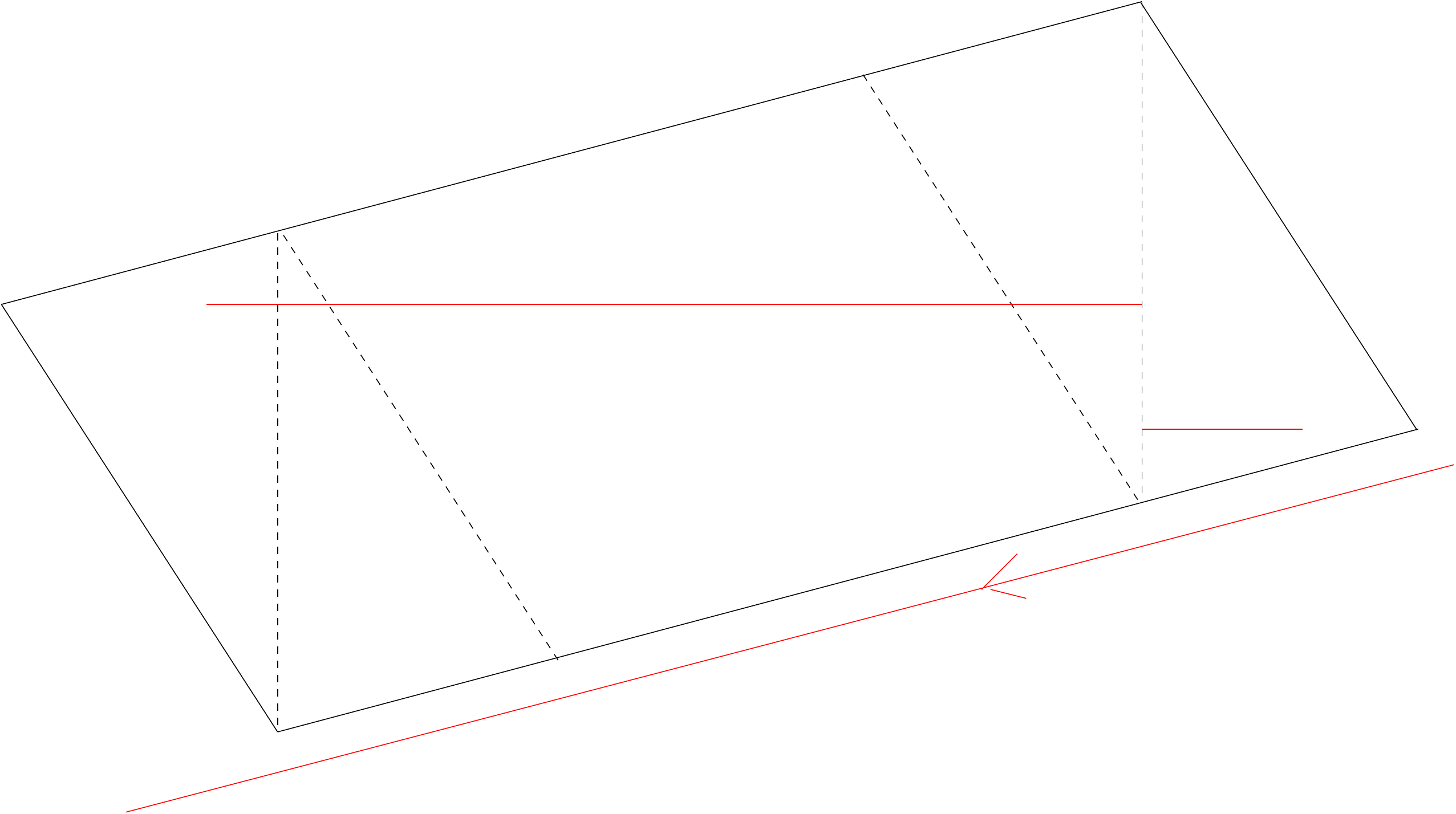}
\put(-305,14){$O_N$}
\put(-305,150){$O_N'$}
\put(-230,30){$O_{N-1}$}
\put(-325,122){$X_N$}
\put(-180,30){$\tilde \alpha$}
\put(-37,93){$X_0$}
\put(-73,200){$O_0'$}
\put(-73,109){$\bigg\}  \geq 1$}
\put(-73,72){$O_0$}
\caption{The parallelogram $\mathcal{P}$ is pictured for the case that $i$ is odd. Points $X_N, X_0 \in \rho^{-1}(x_0)$ are shown with $y(X_N) > y(X_0)$, driving for a contradiction.}
\label{fig:OmorethanN}
\end{figure}

Let $x(X_0)-x(X_N) = w \geq 1$.
Consider the line segment $\overline{X_NX_0}$ lying within $\mathcal{P}$.
Since $y(X_0) \ge y(X_N)$, the line segment $\overline{X_NX_0}$ has non-negative slope, while each component of $\rho^{-1}(\beta)$ is a line of negative slope $s_2=-r_{i-1}/r_i$. Let $X_N' = (x(X_N),y(X_0))$.
It follows that each component of $\rho^{-1}(\beta)$ with a point of intersection on the horizontal line $\overline{X_N'X_0}$ also has a point of intersection on $\overline{X_NX_0}$.
Hence $|\overline{X_NX_0} \cap \rho^{-1}(\beta)| \ge |\overline{X_N'X_0} \cap \rho^{-1}(\beta)| = w r_{i-1} +1$, including endpoints in the count.
The last equality comes from the fact that a horizontal side of a fundamental domain $[0,1] \times [0,1]$ with corners in $\rho^{-1}(x_0)$ intersects $\beta$ in $r_{i-1}+1$ points, see Figure~\ref{fig:lengthestimate}.
Since $\mathcal{P}$ contains $\overline{X_NX_0}$ and has sides of slope $s_2$, each component of $\rho^{-1}(\beta)$ that intersects $\overline{X_NX_0}$ intersects the bottom side of $\mathcal{P}$ as well.
Thus, the bottom side of $\mathcal{P}$ meets $\rho^{-1}(\beta)$ in at least $w r_{i-1} +1$ points. 
See Figure~\ref{fig:0lessthanN}.
Thus, $(N+1) \delta r_i >  w \delta r_{i-1}$, recalling that the sides of $\Delta_1$, $\Delta_2$ of slope $s_1$ have length $\delta r_{i}$.
Thus, $N+1 > r_{i-1}/r_i$, and $N \geq \lfloor r_{i-1} / r_{i} \rfloor$ as required.
\begin{figure}
\includegraphics[width=2in]{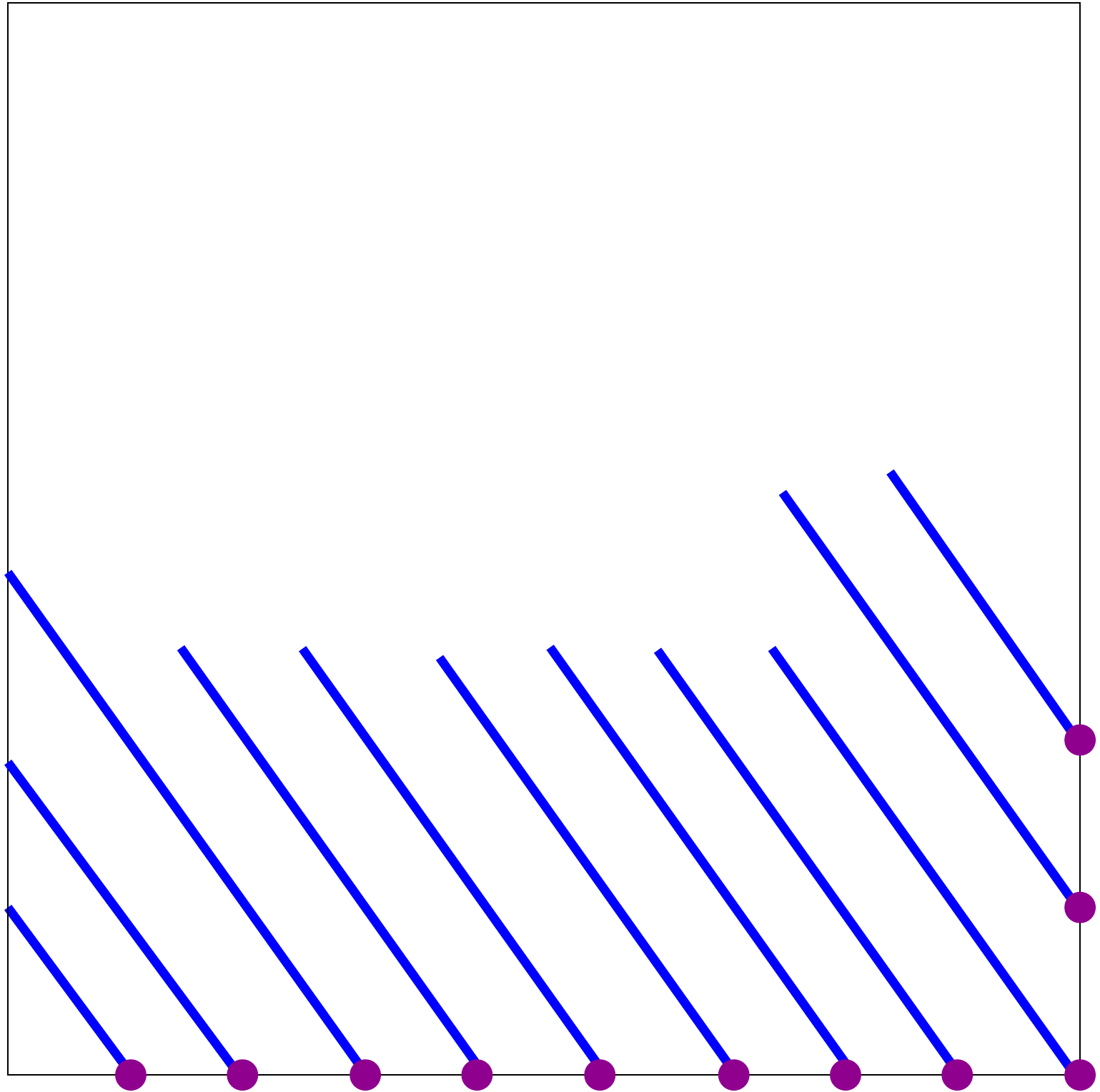}
\put(-80,-10){$r_{i-1}$}
\put(-150,-10){$X_N'$}
\put(-10,-10){$X_0$}
\caption{Here $x(X_0)-x(X_N) = w = 1$. In this case, $\overline{X_N'X_0}$
 corresponds to the bottom side of a fundamental domain $[0,1] \times [0,1]$. Thus $|\overline{X_N'X_0} \cap \rho^{-1}(\beta)| =  r_{i-1} +1$. 
}
\label{fig:lengthestimate}
\end{figure}
\end{proof}

\begin{figure}
\includegraphics[width=5in]{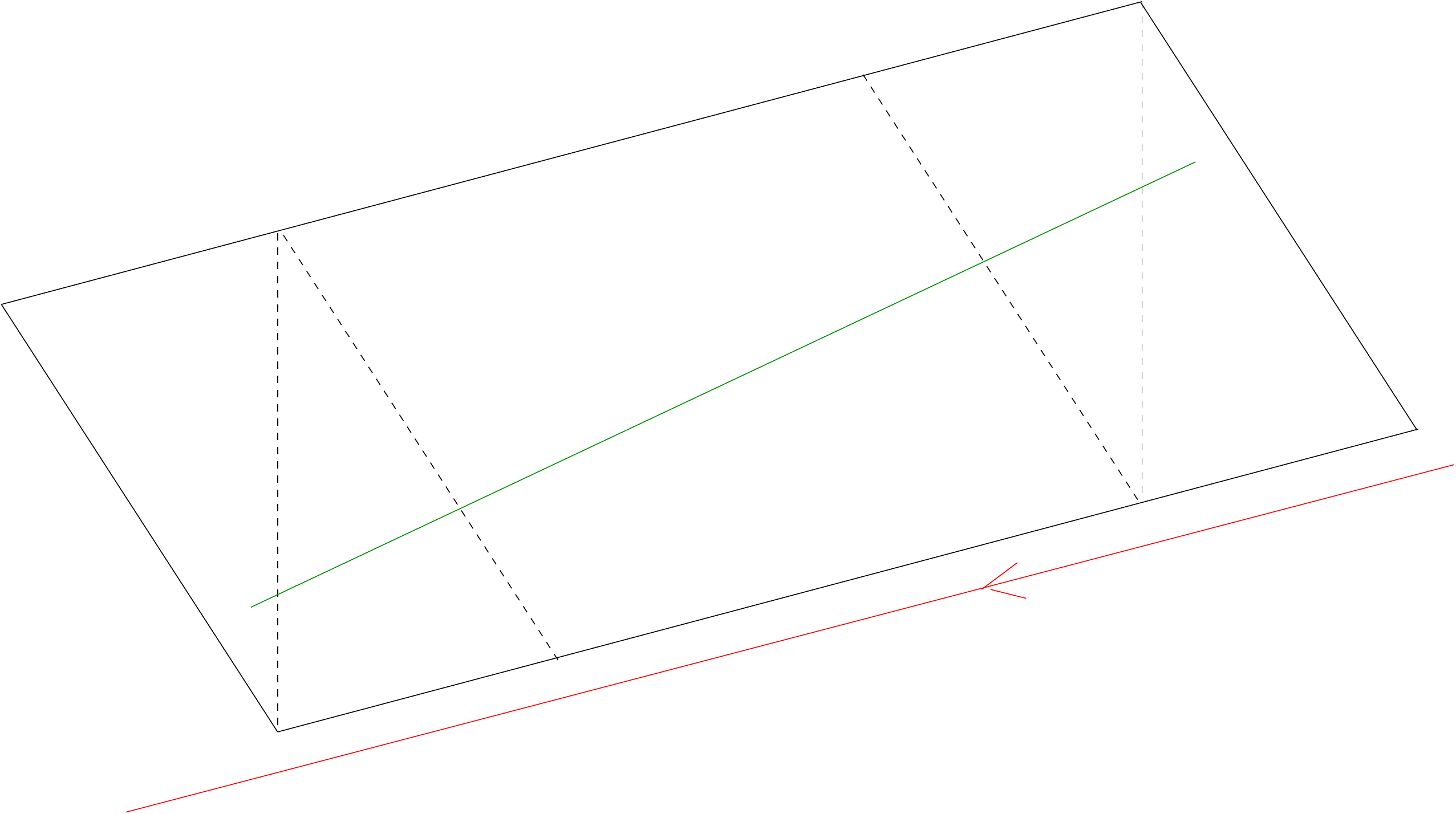}
\put(-305,14){$O_N$}
\put(-305,150){$O_N'$}
\put(-230,30){$O_{N-1}$}
\put(-309,44){$X_N$}
\put(-180,30){$\tilde \alpha$}
\put(-68,150){$X_0$}
\put(-73,200){$O_0'$}
\put(-73,72){$O_0$}
\caption{The bottom edge of the parallelogram $\mathcal{P}$ intersects $\rho^{-1}(\beta)$ in at least as many points as $\overline{X_NX_0}$.}
\label{fig:0lessthanN}
\end{figure}
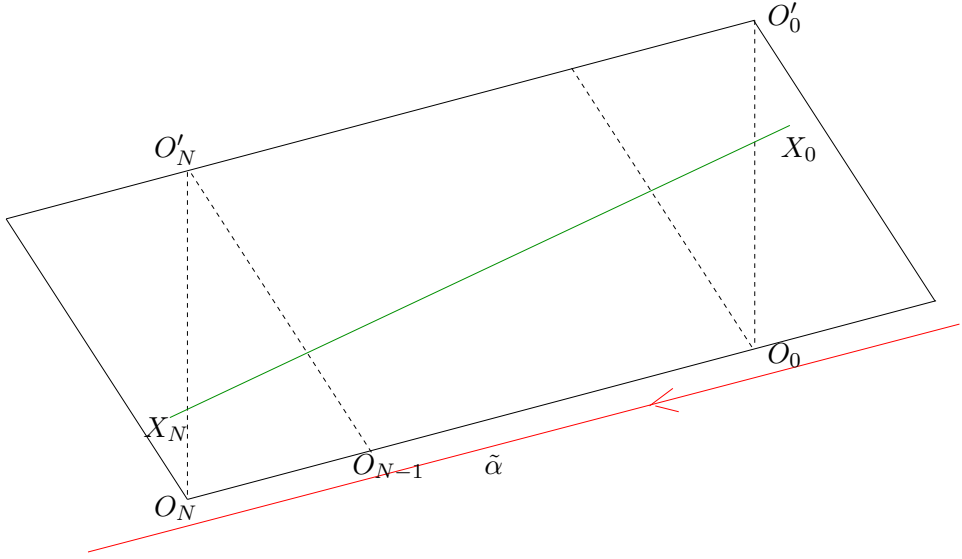

Let $g : \bZ \to \bZ$ be the $p$-periodic extension to the integers of the function $g : \{0,\dots,p-1\} \to \bZ$ of Section~\ref{ss: notation} corresponding to Equation~(\ref{eq: jump function}). 
As stated in Equation~(\ref{eq: gandc}) of Section~\ref{ss: domains}, $g(j)=c_j$ is the weight of the region $R_j$ for the domain $D$
when $j$ takes values in $\{0, \dots, p-1\}$. Note that $0$ is the minimum value assumed by $g$. 

For a function $g:\bZ \to \bZ$ bounded from below, define
\[
m(g)=\{x \in \bZ \, | \, g(x)=\min(g)\},
\]
the set of {\bf minimizers} of $g$.
Let $a>b>0$ be relatively prime integers.
We say that {\bf $g$ has an $(a,b)$-symmetry on its minima} if $m(g)$ is a union of arithmetic progressions of length $\lceil a / b \rceil$ and common difference $b$.
That is, whenever $j \in m(g)$,
there exists an integer $M_j$ such that  $j$ lies in a sequence of $\lceil a / b \rceil$  elements of $m(g)$: $M_j, M_j+b, \dots, M_j+h b, \dots, M_j+\lfloor a / b \rfloor \cdot b$.

 Proposition~\ref{prop:lotsofzeros} then takes on the following form:

\begin{thm}\label{thm:lotsofzerosg}
Let $\theta$ be the order of $K(p,q,k)$. Assume that $\theta \, | \, r_i$ for some remainder $r_i$ of the pair $(p,q)$. Let $g: \bZ \to \bZ$ be the function associated to $K(p,q,k)$ described in Section~\ref{ss: notation}. 
Then $g$ has an $(r_{i-1},r_i)$-symmetry on its minima.
\qed
\end{thm}


\section{A formula for the number of minima of $S(p,q,k)$.}
\label{s: minima}


The goal of this section is a proof of Theorem~\ref{thm: minimizer}.
This will be done by a recursive argument, and we begin in Subsection \ref{ss: dynamics} by setting up an appropriate context in which to work, that of an affine progression over a half open interval of integers.
For notational ease, we use the variables $p$ and $q$ to phrase our results throughout this Section.
When we prove Theorem~\ref{thm: minimizer} in Subsection \ref{ss: harmonics}, we fix the pair of $p$ and $q$ and then call the preparatory lemmas of Subsections \ref{ss: minima prep}-\ref{ss: minima divides} with the variables $p$ and $q$ in their statements replaced by a pair of consecutive remainders for the fixed pair.
Thus, for emphasis, we do not assume any relationship between $\theta$ and $p$ until the proof of Theorem~\ref{thm: minimizer}.
Moreover, Subsections \ref{ss: minima does not divide} and \ref{ss: minima divides} handle the cases that arise when $\theta \nmid q$ and $\theta \mid q$, respectively.

\subsection{Dynamics and modular arithmetic}
\label{ss: dynamics}

For $a,p \in \bZ$, $p > 0$, let $[a]_p$ denote the least positive residue of $a \pmod p$. 
For an interval $I = [s,s+p)$, $s \in \bZ$, an {\bf affine progression} in $I$ of {\bf difference} $\epsilon q$, {\bf length} $l$, and {\bf width} $p$ is a subset $S \subset I$ such that $S \smallsetminus \{ s \}$ takes the form
\[
\{ s +[\epsilon q]_p, s + [2 \epsilon q]_p,\dots, s +[l \epsilon q]_p \},
\]
where $q,l \in \bZ$, $\gcd(p,q)=1$, $0 <  q < p$, $0 \leq l < p$, and $\epsilon= \pm 1$.
Note that $s$ may or may not belong to $S$.
Let $b = s +  [l \epsilon q]_p$ denote its {\bf break point}.
In the extremal case that $l=0$, we allow ourselves to set the difference as we see fit, including equal to 0, and $b=s$ may or may not belong to $S$.

The principal example is the set $S$ of Subsection \ref{ss: notation} of the Introduction, which is an affine progression in the interval $[0,p)$ with difference $q$ and length $l$. 

An affine progression in $I$ can be characterized as the set of images of $s$ under iterations of a simple arithmetic transformation.
Define the {\bf cycle}-$I$-{\bf by}-$\epsilon q$ map $f = f_{I,\epsilon q}$ by
\begin{equation}
\label{e:shift}
f : I \to I, \quad f(x) =[x-s+\epsilon q]_p+s.
\end{equation}
Thus, $f$ translates $x$ by $-s$ into $[0,p)$, cycles by $\epsilon q \pmod p$, and translates by $s$ back to $I$.
The affine progression $S$ therefore consists of the images of $s$ under the first $l$ iterates of $f$, possibly along with $s$.
It is also the set of images of $f(b)$ under the first $l$ iterates of $f^{-1} = f_{I,- \epsilon q}$ (possibly along with $s$).
Note that the set of images of any point under {\em all} iterates of $f$ is the whole of $I$, since $\gcd(p,q)=1$.
Note also that we can alter the domain of $f$ to any subset of $\bZ$, and it is convenient to extend it to the closure $\overline{I} = [s,s+p]$.

Fix a non-empty interval $J \subset I$ and a point $x \in \overline{I}$.
There exists a smallest positive integer $t_1=t_1(x)$ for which $x_1 : =f^{(t_1)}(x) \in J$.
For $k \ge 2$, inductively define
\[
t_k = t_k(x) = t_{k-1}(x) + t_1(x_{k-1}) \quad \textup{and} \quad x_k = f^{(t_k)}(x).
\]
The point $x_k$ is the {\bf $k$-th return} of $x$ to $J$ under $f$ and $t_k$ its {\bf $k$-th return time}.

For the following Lemmas, write $p = dq + r$,  $r = [p]_q$, and $J = [y,y+q) \subset I$ for some $y \in \bZ$.
The proof of the first Lemma is a straightforward exercise.
Note that if $x \in J$ and $y$ is the first return of $x$ to $J$ under $f_{I,- \epsilon q}$, then $x$ is the first return of $y$ to $J$ under $f_{I,\epsilon q}$. 

\begin{lem}
\label{lem: first return}
The first return of $x \in \overline{J}$ to $J$ under $f_{I,q}$ is $f_{J,-r}(x)$. 
The first return of $x \in \overline{J}$ to $J$ under $f_{I,-q}$ is $f_{J,r}(x)$. \qed
\end{lem}

\noindent

The next Lemma asserts that affine progressions restrict to affine progressions on nice subintervals, enabling the inductive argument of Theorem \ref{thm: minimizer}.

\begin{lem}
\label{lem: dynamics}
Suppose that $S$ is an affine progression in $I$ with difference $\epsilon q$, $q > 0$, $\epsilon = \pm 1$, with break point $b$ and length $l$.
Then $S \cap J$ is an affine progression if $y \in \{s,b,s+p-q\}$.
It has difference $\epsilon r$ if $y=b$ and difference $- \epsilon r$ otherwise.
\end{lem}

\noindent
Note that if $y=b$, then $b \le s+p-q$, by the running assumption that $J = [y,y+q) \subset I$.
On any interval $J \subset I$ of width $q$, $S \cap J$ consists of the images of {\em some} point under the first number of iterates of $f_{\pm r,J}$.
However, this point is not always its left endpoint, which is part of the important content of Lemma \ref{lem: dynamics}.

\noindent
{\em Example.} Take $I =[0,33)$, $q=10$, and $l=11$.  We obtain the affine progression $S \subset I$ shown here:

\begin{center}
\includegraphics[width=5in]{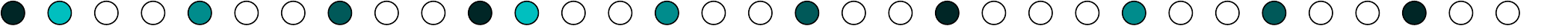}
\put(-362,10){$0$}
\put(-350,10){$1$}
\put(-317,10){$4$}
\put(-287,10){$7$}
\put(-258,10){$10$}
\put(-245,10){$11$}
\put(-215,10){$14$}
\put(-183,10){$17$}
\put(-149,10){$20$}
\put(-106,10){$24$}
\put(-73,10){$27$}
\put(-43,10){$30$}
\end{center}

\noindent
We shade the points of $S$ according to how many times we have to ``carry" 0 on iterating $f_{I,10}$.
In this example, $b = 11$ and $r=3$.
We see that $S$ meets
\begin{itemize}
\item
$[s,s+q)=[0,10)$ in an affine progression with difference $-3$;
\item
$[p-q,p) = [23,33)$ in an affine progression with difference $-3$ where $s \notin S$; and
\item
$[b,b+q) = [11,21)$ in an affine progression with difference $+3$.
\end{itemize}
The length is 3 in each case.

\begin{proof}
If $l=0$, then the restriction to any subinterval of $I$ is affine with length $0$, and we can set the difference as desired. 
If $q = 1$, then $S \cap J$ is an affine progression with length $0$, so we can set the difference to be $r=0$, as claimed.  We assume then that $l>1, q >1$. For the proof, we treat the case $\epsilon = +1$ for ease of notation; the case $\epsilon = -1$ follows with minor changes of notation.

We will identify $S \cap J \smallsetminus \{y\}$ with the set of images of $y$ under the first $k=|S \cap J \smallsetminus \{y\}|$ iterates of $f_{J,-r}$, when $y = s$ or $y+q=p$, and its inverse $f_{J,r}$, when $y = b$.

First, suppose that $y=s$.
Since $S \smallsetminus \{y\}$ is the set of images of $y$ under the first $l$ iterates of $f_{I,q}$, Lemma \ref{lem: first return} shows that $S \cap J$ takes the stated form.

Next, suppose that $y + q = s+p$.
Observe that the first return of $s$ to $J$ and of $y$ to $J$ under $f_{I,q}$ are the same.   Applying Lemma \ref{lem: first return} to $y$ shows that the first return of $s$ to $J$ is the same as the first iterate of $y$ under $f_{J,-r}$.
The desired conclusion about $S \cap J$ now follows as before from Lemma \ref{lem: first return}. 

Lastly, suppose that $y = b \le s+p-q$. In this case, we recognize $S \smallsetminus \{b,s\}$ as the set of images of $b$ under the $l-1$ iterates of $f^{-1}_{I,q} = f_{I,-q}$. Lemma \ref{lem: first return} now shows that $S \cap J \smallsetminus \{y\}$ is the set of images of $y$ under the first $k$ iterates of $f_{J,r}$.
\end{proof}

The following simple Lemma expresses the periodicity of an affine progression.
It is used to establish a basic property of $S$-jump functions in Lemma \ref{lem: minima}.

\begin{lem}
\label{lem: pattern}
Suppose that $S \subset I = [s,s+p)$ is an affine progression of difference $q$ and break point $b$, and suppose that $x,x+q \in I$.
\begin{enumerate}
\item
If $x \ne s,b$, then $x \in S \iff x+q \in S$.
\item
There is a $c$ such that
\begin{equation}
\label{eq: c points}
|S \cap (x,x+q]| = 
\begin{cases}
c, & x < b \\
c-1, & x \ge b.
\end{cases}
\end{equation}
\end{enumerate}
\end{lem}

\begin{proof}
(1)($\implies$)
Suppose that $x \in S$.
Since $x+q \in I$, $x+q = f_{I,q}(x)$.
Since $x \ne b$, $f_{I,q}(x) \in S$.
Thus, $x+q \in S$.

\noindent
(1)($\impliedby$)
Suppose that $x+q \in S$.
Since $x \in I$, $f_{I,q}^{-1}(x+q) = x$.
Since $x \ne s$, $f_{I,q}^{-1}(x+q) \in S$.
Thus, $x \in S$.

\noindent
(2)  If the length of $S$ is $0$ then $c=0$ works. So assume the length of $S$ is not $0$. Take $c = |S \cap (s,s+q]|$. Then $s < b$ and the expression holds for $s$. The proof proceeds by induction on $x \ge s$, using part (1), and noting that if $b+q \in I$, then $b+q \notin S$. 
\end{proof}


\subsection{Minima of jump functions: preparation.}
\label{ss: minima prep}

Suppose that $S$ is an affine progression in $I = [s,s+p)$ and that $\theta > t$ are relatively prime positive integers.
An $S$-{\bf jump function} (with parameters $\theta$, $t$) is a function $g : \bZ \to \bZ$ with the properties that
\[
g(x) - g(x-1) \in \{t,t-\theta\}, \quad \forall x \in \bZ,
\]
and
\[
g(x) - g(x-1) = \begin{cases} t, & x \notin S, \\ t - \theta, & x \in S, \end{cases}
\]
for all $s \le x < s+p$.
Note that the values of $g$ are otherwise unconstrained.
If $J \subset I$ is a subinterval for which $S \cap J$ is an affine progression, then an $S$-jump function is also an $(S \cap J)$-jump function with the same parameters $\theta$, $t$.

The principal example is the $p$-periodic function $g:\bZ \to \bZ$ from Subsection \ref{ss: notation} of the Introduction.
It is an $S$-jump function for the affine progression $S$ over the interval $I=[0,p)$ and parameters $\theta$ and $t$ introduced there.
The one sensitive point is to check that $g(-1)=g(p-1)$, in conformity with $p$-periodicity and the defining relations that $g$ obeys on $\{-1,0,\dots,p-1\}$.
To verify it, express $g(p-1)-g(-1)$ as the telescoping sum $\sum_{x=0}^{p-1} (g(x)-g(x-1))$.
The sum consists of $|S|$ terms equal to $t-\theta$ and $p-|S|$ terms equal to $t$.
Since $|S| = l$ and $t = l \theta / p$ in the notation of Subsection \ref{ss: notation}, the sum becomes $l(t-\theta) + (p-l)t = 0$, as desired.

Recall from the end of Section \ref{sec:findingmins} that for a function $g : \bZ \to \bZ$ bounded from below, $m(g) = \{ x \in \bZ \, | \, g(x) = \min(g) \}$ denotes the set of minimizers of $g$.
Our goal in this and the next several subsections is to describe conditions under which we may pass from an interval $I$ containing elements of $m(g)$ to a proper subinterval $J \subset I$ containing $m(g) \cap I$.
The technical lemmas that we establish enable an inductive argument for pinpointing $m(g)$ in the proof of Theorem \ref{thm: minimizer}.
For instance, the following simple result characterizes $m(g)$ for two extreme types of affine progression.
 
\begin{lem}
\label{lem:length0}
Let $I = [s,s+p)$ and let $S$ be an affine progression in $I$ with length $l$. 
Assume that $m(g) \cap I \neq \emptyset$.
If $l=0$, then $m(g) \cap I = \{ s \}$, while if $l=p-1$, then $m(g) \cap I = \{ s+p-1 \}$. 
\end{lem}

\begin{proof} 
If $l=0$, then $(I - \{ s \}) \cap S = \emptyset$.
Thus, $g$ is increasing on $I$, and $m(g) \cap I = \{ s \}$.
If instead $l=p-1$, then $(I - \{ s \}) \cap S = I - \{ s \}$.
Thus, $g$ is decreasing on $I$, and $m(g) \cap I = \{ s+p-1 \}$.
\end{proof}

Suppose that $S$ is an affine progression in $I = [s,s+p)$ with difference $\pm q$ ($q>0$) and $g$ is an $S$-jump function.
We say that $I$ is $q$-{\bf isolated} for $g$ if $m(g) \cap [s-q,s+p+q) = m(g) \cap I$: that is, no additional minima of $g$ appear within a $q$-neighborhood of $I$.
Let $S'$ be an affine progression in $J=[w,z)$, and let $g$ be an $S'$-jump function with parameters $t,\theta$. 
We say that
\begin{itemize}
\item
$J$ is {\bf left-like} if $g(z)-g(w)>0$ and $g(z)-g(z-1)=t-\theta$, and
\item
$J$ is {\bf right-like} if $g(z-1)-g(w-1)<0$ and $g(w)-g(w-1) = t$.
\end{itemize}
These conditions are essential for the inductive argument we carry out for locating the minima of an $S$-jump function, and the terminology stems from Lemma \ref{lem: minima} below.
Note that an interval $J$ may be left-like, right-like, neither, or both. Note also that the second condition of $J$ being right-like is equivalent to $w \notin S'$ by the defining condition of an $S'$-jump function. We will often think of the second condition of right-likeness in this way.
In what follows, the affine progression $S'$ in $J$ will typically come from the restriction to $J$ of an affine progression $S$ of a larger interval $I$ (that is, $J \subset I$, $S'=S \cap J$,  and $g$ an $S$-jump function with the same parameters $t,\theta$).
In that case, the left-like condition that $g(z)-g(z-1)=t-\theta$ follows if $z \in S$.
Typically, this will be the way in which this condition is verified.

Suppose that $S$ is an affine progression in $I$, and suppose that $g$ is an $S$-jump function that attains its minimum in $I$.
The next result identifies a subinterval $J \subset I$ that contains $m(g) \cap I$ and to which $S$ restricts to an affine progression with particular qualities.
It explains the use of our terminology ``left-like" and ``right-like", since they convey where $J$ is located within $I$.

\begin{lem}
\label{lem: minima}
Let $I = [s,s+p)$, let $S$ be an affine progression in $I$ with difference $q>0$ and length $l$, and let $g$ be an $S$-jump function on $I$ with parameters $t,\theta$.
Set $r=[p]_q$.
Assume that $m(g) \cap I \neq \emptyset$. 
If $\theta \nmid  q$, then either
\begin{enumerate}
\item
$m(g) \cap I \subset [s,s+q)$, $[s,s+q)$ is left-like, and $[s,s+q) \cap S$ is an affine progression with difference $-r$;
\item
$m(g) \cap I \subset [b,b+q) \subset I$, and $[b,b+q) \cap S$ is an affine progression with difference $r$;
\item
$m(g) \cap I \subset [s+p-q,s+p)$, $[s+p-q,s+p)$ is right-like, and $[s+p-q,s+p) \cap S$ is an affine progression with difference $-r$; or
\item 
$l=0$ and $m(g) \cap I =\{s\} \subset [s, s+q)$ and $[s,s+q) \cap S$ is an affine progression with difference $r$ and length $0$.
\end{enumerate}
If $I$ is $q$-isolated for $g$, then the corresponding subinterval above is $r$-isolated for $g$.

{\em Addendum:} Let $b$ be the break point of the affine progression $S \cap I$.
If $b > s+p-q$, then either
\begin{itemize}
\item $g(x+q)-g(x)>0$ for $s \leq x <s+p-q$ and (1) holds; or
\item $g(x+q)-g(x)<0$ for $s \leq x <s+p-q$ and (3) holds.
\end{itemize}

\end{lem}

\noindent
{\em Example.} 
Consider once more the arithmetic progression $S \subset I = [0,33)$ of difference 10 and length 11, and let $g$ be the 33-periodic $S$-jump function.
We report three choices for $(\theta,t)$, one corresponding to each of the outcomes (1)-(3) of Lemma \ref{lem: minima}, and in each case the value $\min(g)$, the set $m(g) \cap I$, and the guaranteed subinterval:
\begin{enumerate}
\item
$(\theta,t)=(3,2)$: $\min(g)=-1$, $m(g) = \{1\} \subset [0,10)$;
\item
$(\theta,t) = (3,1)$: $\min(g) =-4$, $m(g) = \{11,14,17,20\} \subset [11,21)$;
\item
$(\theta,t)=(4,1)$: $\min(g)=-14$, $m(g) = \{ 30 \} \subset [23,33)$.
\end{enumerate}

\begin{proof}
If $l=0$, then Lemma~\ref{lem:length0} shows that $m(g) \cap I =\{s\} \subset [s,s+q)=[b,b+q)$. We can take $[b,b+q) \cap S$ as an affine progression with difference $r$ and length $0$. Thus conclusion $(2)$ of the Lemma is satisfied. If $l=p-1$, and hence $b=s+p-q$, then Lemma~\ref{lem:length0} shows that  $m(g) \cap I =\{s+p-1\} \subset [s+p-q,s+p)=[b,b+q)$.
By Lemma~\ref{lem: dynamics}, $[b,b+q) \cap S$ is an affine progression with difference $r$ (and length $q-1$).
This is again conclusion $(2)$. 
We hereafter assume $0 < l < p-1$.
By Lemma \ref{lem: pattern}(2), $S$ meets every interval $(x,x+q] \subset I$ in the same number of points $c$ for $x < b$ and in the same number of points $c-1$ for $x \ge b$.
Setting $\sigma = (q-c)t + c(t - \theta) = q t - c \theta$, we therefore have
\[
g(x+q) - g(x) = \begin{cases} \sigma, & s \le x < b; \\ \sigma + \theta, & b \le x < s+p-q. \end{cases}
\]
Since $\theta \nmid q$ and $\gcd(t,\theta) = 1$, we have $\sigma, \sigma + \theta \ne 0$.
There are three possibilities to consider, each of which leads to the corresponding numbered outcome of the Lemma:

\begin{enumerate}
\item
If $\sigma > 0$, then $g(x+q) > g(x)$ for all $x,x+q \in I$.
It follows in this case that $m(g) \cap I \subset [s,s+q)$ and that $g(s+q)-g(s) > 0$.
By Lemma~\ref{lem: dynamics}, $[s, s+q) \cap S$ is an affine progression with difference $-r$.
Since $l > 0$, we have $s+q \in S$, and $[s,s+q)$ is left-like. 
\smallskip
\item
If $\sigma < 0 < \sigma+\theta$, then $g(x+q) < g(x)$ for all $x, x+q \in I$, $x < b$, while $g(x+q) > g(x)$ for all $x,x+q \in I$, $x \ge b$. 
Thus, $m(g) \cap I \subset [b,b+q)$. By Lemma~\ref{lem: dynamics}, $[b,b+q)$ is an affine progression with difference $r$. As noted above, this includes the cases when $l=0,p-1$. It also includes the case that $q=1$ (with $r=0$). 
\item
If $\sigma+\theta < 0$, then $g(x+q) < g(x)$ for all $x,x+q \in I$.
It follows in this case that $m(g) \cap I \subset [s+p-q,s+p)$ and that $g(s+p-1) - g(s+p-q-1) < 0$. If $s+p-q \in S$, then $I \cap S$ has length $p-1$, which is assumed not to be the case.
Thus $s+p-q \notin S$. Applying Lemma~\ref{lem: dynamics}, 
we conclude that $[s+p-q,s+p) \cap S$ is a right-like affine progression with difference $-r$. 

\smallskip
\end{enumerate}
Since each of these subintervals contains all of $m(g) \cap I$, which is non-empty, and since $r < q$, the $q$-isolation of $I$ for $g$ immediately implies the $r$-isolation of the subinterval for $g$.

To verify the Addendum, let $b > s+p-q$ be the break point of the affine progression. Then $g(x+q) - g(x) =  \sigma$ for $s \leq x <s+p-q$. That is, 
$g(x+q) - g(x) =  \sigma$ for $x, x+q \in I$. Note that $b > s+p-q$ implies $l \neq 0, p-1$.
If $\sigma > 0$, apply the argument of possibility $(1)$ above to reach conclusion $(1)$. If $\sigma < 0$, apply the argument of possibility $(3)$ above to reach conclusion $(3)$.
\end{proof}


\subsection{Minima of jump functions: $\theta \nmid q$.}
\label{ss: minima does not divide}

As in the previous subsection, let $S$ be an affine progression on the interval $I=[s,s+p)$ with difference $\epsilon q$ with $\epsilon= \pm 1$ and $q>0$. Let $g$ be an $S$-jump function with parameters $\theta , t$. Under the assumption that the parameter $\theta$ does not divide the difference $\epsilon q$, 
Lemma \ref{lem: minima} provides the basis for an inductive argument for locating the minima of $g$ within $I$ when the difference is positive ($q$). 
This subsection is devoted to the case when the difference is negative ($-q$), again under the assumption that $\theta$ does not divide the difference, treated in Lemma \ref{lem: minima 2-2} below. When the difference is negative, we need to make the extra assumption that the interval is either left-like or right-like. 
In order to prove Lemma \ref{lem: minima 2-2}, we require the following version of Lemma \ref{lem: pattern}, whose proof is analogous (and is one place in this subsection that does not require that $\theta \nmid q$).

\begin{lem}
\label{lem: pattern 2-2}
Suppose that $S \subset I = [s,s+p)$ is an affine progression of difference $-q$ and break point $b$.
\begin{enumerate}
\item
If $x,x-q \in I$ and $x \ne s+q,b$, then $x \in S \iff x-q \in S$.
\item
If $x,x+(p-q) \in I$ and $x \ne s,b$, then $x \in S \iff x+(p-q) \in S$.
\item
There is a $c$ with the following property. If $x,x-q \in I$ then
\[
\label{eq: c'' points}
|S \cap (x-q,x]| = 
\begin{cases}
c, & x < b \\
c+1, & x \ge b. 
\end{cases}
\]
\end{enumerate}
\end{lem}

\begin{proof}

The proof of part $(1)$ is the same as for Lemma \ref{lem: pattern}(1).

For part $(2)$, note that  $x,x+(p-q) \in I$ means that $f_{I,-q}(x)=x+(p-q)$. Thus $x \in S$ means that $x+(p-q) \in S$ if $x \neq b$. If $x \neq s$, then $x+p-q \in S$ means $x \in S$. 

To prove (3), assume $x,x-q, x+1, x-q+1 \in I$. Note that since $x-q \geq s$, $x+1 \neq s+q$. If $x+1 <b$ then $|S \cap (x-q,x]|=|S \cap (x-q+1,x+1]|$ by part $(1)$. The same holds if $x \geq b$. Thus we consider the case when $x+1=b$.  By definition of $b$, $x-q+1 \notin S$, but $x+1 \in S$. That is, $|S \cap (x-q,x]|+1=|S \cap (x-q+1,x+1]|$. Together, these verify $(3)$.

\end{proof}

\begin{lem}
\label{lem: minima 2-2}
Let $I= [s,s+p)$, and let $S$ be an affine progression in $I$ with difference $-q<0$.
Assume that $I$ is either left-like or right-like.
Set $r=[p]_q$.
Assume that $m(g) \cap I \ne \emptyset$. 
If $\theta \nmid  q$, then either
\begin{enumerate}
\item
$m(g) \cap I \subset [s,s+q)$, and $[s,s+q) \cap S$ is an affine progression with difference $r$; or
\item
$m(g) \cap I \subset [s+p-q,s+p)$, and $[s+p-q,s+p) \cap S$ is an affine progression with difference $r$.
\end{enumerate}
If $I$ is $q$-isolated for $g$, then the subinterval from (1) or (2) above is $r$-isolated for $g$.
\end{lem}

\begin{proof}
[Proof of Lemma \ref{lem: minima 2-2}]
First note by Lemmas~\ref{lem: dynamics} and \ref{lem:length0} that we may assume that the length of $I \cap S$ is strictly between $0$ and $p-1$.

We begin as in the proof of Lemma \ref{lem: minima}.
Let $c$ be as in Lemma \ref{lem: pattern 2-2}(3) and set $\sigma = (q-c) t + c (t - \theta) = q t - c \theta$.
Lemma \ref{lem: pattern 2-2}(3) gives
\begin{equation}
\label{eq: r difference}
g(x) - g(x-q) = \begin{cases} \sigma, & s+q \le x < b; \\ \sigma - \theta, & b \le x < s+p. \end{cases}
\end{equation}
Note that if $b \le s+q$, then the first outcome does not occur.
Since $\theta \nmid q$ and $\gcd(t,\theta) = 1$, we have $\sigma, \sigma-\theta \ne 0$.
As in the proof of Lemma \ref{lem: minima}, if $\sigma < 0$, then we conclude that $m(g) \cap I \subset [s+p-q,s+p)$, while if $\sigma - \theta > 0$, then $m(g) \cap I \subset [s,s+q)$.

The remaining possibility is that $b > s+q$ and $\sigma - \theta < 0 < \sigma$.
In this case, \eqref{eq: r difference} only implies that $m(g) \cap I \subset [s,s+q) \cup [s+p-q,s+p)$.
We must work harder to constrain $m(g) \cap I$ to one of the two constituent intervals of width $q$.
In particular, we must invoke the left- or right-likeness of $I$.

Note that $b > s+q$ implies that $S \cap I$ has positive length and that $S$ does not hit every point of $I - \{s\}$.
Consider the difference $g(x+(p-q)) - g(x), x \in [s,s+q)$.
By Lemma \ref{lem: pattern 2-2}(2) and the fact that $b > s+q$, it follows that this value is constant on $[s,s+q)$.
If it is positive, then $m(g) \cap I \subset [s,s+q)$, while if it is negative, then $m(g) \cap I \subset [s+p-q,s+p)$, resulting in the required conclusions.

The remaining possibility is that
\begin{equation}
\label{eq: q-r difference-2}
g(x+(p-q)) - g(x) = 0, \quad \forall x \in [s,s+q),
\end{equation}
which we now assume in pursuit of a contradiction.  Using \eqref{eq: q-r difference-2}, we first establish the following comparison:
\begin{equation}
\label{eq: q difference-2}
g(s+p)-g(s) = \begin{cases} \sigma, &g(s+p)-g(s+p-1)=t; \\ \sigma - \theta, & g(s+p)-g(s+p-1)=t-\theta. \end{cases}
\end{equation}

First, suppose that $g(s+p)-g(s+p-1)=t$.
Note that $s+q \notin S$, since otherwise $S$ meets every point of $I - \{s\}$, which we already noted does not occur.
In particular, $g(s+q)-g(s+q-1)=t$.
Consequently,
\[
g(s+p)-g(s+q)=g(s+p-1)-g(s+q-1)=0,
\]
taking $x = s+q-1$ in \eqref{eq: q-r difference-2}, and so
\[
g(s+p)-g(s)= (g(s+p)-g(s+q)) + (g(s+q)-g(s)) = \sigma
\]
using \eqref{eq: r difference}.

Second, suppose that $g(s+p)-g(s+p-1)=t-\theta$.
Since $S \cap J$ has positive length, $s+p-q \in S$ and $g(s+p-q)-g(s+p-q-1)=t-\theta$.
This along with \eqref{eq: r difference} shows
\[
g(s+p)-g(s+p-q)= g(s+p-1)-g(s+p-q-1)= \sigma - \theta,
\]
and so
\[
g(s+p)-g(s)=(g(s+p)-g(s+p-q))+(g(s+p-q)-g(s)) = \sigma - \theta,
\]
taking $x = s$ in \eqref{eq: q-r difference-2}.
These two cases verify \eqref{eq: q difference-2}.

Now we invoke \eqref{eq: q difference-2} and the left- / right-like hypotheses to reach the desired contradiction.
If $I$ is left-like, then $g(s+p)-g(s+p-1)=t-\theta$ and $g(s+p) - g(s) > 0$, contradicting \eqref{eq: q difference-2}.
So assume $I$ is right-like.
Then $g(s)-g(s-1)=t$.
If $g(s+p)-g(s+p-1)=t$, then $g(s+p)-g(s) = g(s+p-1)-g(s-1)<0$ contradicting \eqref{eq: q difference-2}.
If instead $g(s+p)-g(s+p-1)=t-\theta$, then $g(s+p)-g(s) = g(s+p-1)-g(s-1)-\theta$.
On the other hand, the left side is $\sigma - \theta$ by \eqref{eq: q difference-2}, so $g(s+p-1)-g(s-1)= \sigma > 0$.  This contradicts that $I$ is right-like.

Thus we conclude that either $m(g) \cap I \subset [s,s+q)$ or $m(g) \cap I \subset [s+p-q,s+p)$.
Lemma \ref{lem: dynamics} with $\epsilon = -1$ shows that each of these intervals intersects $S$ in an affine progression with difference $r$. If $I$ is $q$-isolated for $g$, then $m(g) \cap I = m(g) \cap [s-q, s+p+q)$. Since $r <q$ whichever of these subintervals  contains $m(g) \cap I$ will be $r$-isolated for $g$. 
\end{proof}

\subsection{Minima of jump functions: $\theta \mid q$.}
\label{ss: minima divides}

As in the preceding subsections, suppose that $S$ is an affine progression in $I=[s,s+p)$ with positive length, difference 
$\pm q$, $q > 0$, and break point $b$.
Assume that $p$ and $q$ are relatively prime.
Let $r=[p]_q$, and let $r'=[q]_r$. Let $g$ be an $S$-jump function with parameters $t,\theta$.
In the preceding subsections, we often assumed that the parameter $\theta$ did not divide the difference of the affine progression.
In this subsection, we address when it does. We use the results of this subsection in the proof of Theorem~\ref{thm: minimizer} where they are applied to the case when $p=r_{i-1}, q=r_i$, and $\theta | r_i$ (note then that $\theta \nmid r= r_{i+1}$). 
We assume that $g$ has a $(p,q)$-symmetry in its minima, as defined at the end of Section \ref{sec:findingmins}.
This assumption is justified in the eventual proof of Theorem~\ref{thm: minimizer} by way of Theorem~\ref{thm:lotsofzerosg}.
Thus, this subsection is devoted to locating the minima of $g$ in the interval $I$ under the assumptions that $\theta \mid q$, hence that $\theta \nmid r$, and that $g$ has a $(p,q)$-symmetry in its minima. 

Recall that $I$ is $q$-isolated for the function $g$ if $m(g) \cap [s-q,s+p+q) = m(g) \cap I$. 
Define subintervals $I_\ell=[s+\ell q,s+ \ell q+r)$ for $\ell = 0, \dots, \lfloor p/q \rfloor$.

\begin{lem}
\label{lem:Icontainsall}
Suppose that $g$ has a $(p,q)$-symmetry in its minima and $I$ is $q$-isolated for $g$.
Then each element of $m(g) \cap I$ lies in a sequence of $\lceil p/q \rceil$ elements of $m(g)$ within $I$ that are $q$ units apart.
Each such sequence has a unique representative in $I_\ell$ for $\ell = 0, \dots, \lfloor p/q \rfloor$.
Thus, $m(g) \cap I$ is contained in the union of pairwise disjoint subintervals $I_0, \dots, I_{\lfloor p/q \rfloor}$.
If $m(g) \cap I \ne \emptyset$, then $m(g) \cap I_\ell \ne \emptyset$ for $\ell=0, \dots, \lfloor p/q \rfloor$, and $S$ does not have length 0 or $p-1$.
\end{lem}

\begin{proof}
 Since $g$ has a $(p,q)$-symmetry in its minima, any minimum of $g$ in $I$ lies in a sequence of 
 $\lceil p/q \rceil$ minima that are spaced $q$ units apart.
 By the hypothesis on $q$-isolation, this whole sequence must lie in $I$.
 Any such sequence has width $q \cdot \lfloor p/q \rfloor$ units, and its last term is less than $s+p$, so its first term must be less than $s+r$.
 That is, each sequence containing an element of $m(g) \cap I$ has a representative in $[s,s+r)$, and thus in $[s+\ell q,s+\ell q+r)$, for each $\ell = 0, \dots, \lfloor p/q \rfloor$.
Furthermore, if $m(g) \cap I \ne \emptyset$, then $|m(g) \cap I| \ge \lceil p/q \rceil \ge 2$.
Hence by Lemma~\ref{lem:length0}, $S$ does not have length 0 or $p-1$, since in both of those cases $|m(g) \cap I|=1$.
\end{proof}

First we consider the case that the difference of $S$ is positive.

\begin{lem}
\label{lem:dividesqpositive}
Let $S$ be an affine progression in $I=[s,s+p)$ with difference $q>0$. Let $r=[p]_q$, and let $r'=[q]_r$.
Suppose that $g$ is an $S$-jump function with a $(p,q)$-symmetry in its minima, $m(g) \cap I \neq \emptyset$, and $I$ is $q$-isolated for $g$. Then $m(g) \cap I$ is contained in the union of pairwise disjoint subintervals $I_0, \dots, I_{\lfloor p/q \rfloor}$.
For each $\ell = 0, \dots, \lfloor p/q \rfloor$, 
$S \cap I_\ell$ is an affine progression of width $r$ and difference $r'$, $I_\ell$ is an $r'$-isolated interval for $g$, and $m(g) \cap I_\ell \neq \emptyset$.
\end{lem}

\begin{proof}
We begin by defining some auxiliary subintervals.
Define $H_\ell=[s+\ell q,s+(\ell+1)q)$ for $\ell = 0, \dots, \lfloor p/q \rfloor-1$.
As the iterates $f_{I,q}^n(s)$ first intersect $H_\ell$ in its left endpoint, and $H_\ell$ has width $q$, repeated application of Lemma~\ref{lem: first return} shows that $S \cap H_\ell$ is an affine progression of difference $-r$ (possibly length $0$). 
Let $L_{ \lfloor p/q \rfloor}$ denote the rightmost interval of width $q$ in $I$.
Then Lemma~\ref{lem: dynamics} shows that $S \cap L_{\lfloor p/q \rfloor}$ is an affine progression with difference  $-r$.   

As $I_\ell$ is the initial interval of width $r$ in $H_l$ for $\ell < \lfloor p/q \rfloor$,  Lemma~\ref{lem: dynamics} shows that $S \cap I_\ell$ is an affine progression with difference $r'$ when $\ell = 0, \dots, \lfloor p/q \rfloor-1$.  
As $I_{\lfloor p/q \rfloor}$ is the rightmost interval of width $r$ within $L_{\lfloor p/q \rfloor}$, Lemma~\ref{lem: dynamics} shows that $S \cap I_{\lfloor p/q \rfloor}$ is an affine progression with difference  $r'$. 

Lemma~\ref{lem:Icontainsall} shows that $m(g) \cap I \subset I_0 \cup \dots \cup I_{\lfloor p/q \rfloor}$, since $I$ is $q$-isolated for $g$. 
Consequently, the interval of width $q-r \ge r'$ between $I_\ell$ and $I_{\ell+1}$ is disjoint from $m(g)$, for $\ell = 0,\dots, \lfloor p/q \rfloor-1$.
Thus, each $I_\ell$ is $r'$-isolated for $g$, as required. Finally, Lemma~\ref{lem:Icontainsall} guarantees that $m(g) \cap I_\ell \neq \emptyset$ for each $\ell$.
\end{proof}

Now we consider the case that the difference of $S$ is negative.
This case requires considerably greater effort.
For the remainder of the subsection, we assume that $g$ has a $(p,q)$-symmetry in its minima, $m(g) \cap I \ne \emptyset$, $I$ is $q$-isolated for $g$, and $S$ has difference $-q$.

The first result is in direct analogy to Lemma \ref{lem:dividesqpositive}:

\begin{lem}
\label{lem:dividesqnegative1}
For $\ell = 0, \dots, \lfloor p/q \rfloor$, $S \cap I_\ell$ is an affine progression of difference $-r'$, $m(g) \cap I_\ell \ne \emptyset$, and $I_\ell$ is an $r'$-isolated interval for $g$.
\end{lem}

\begin{proof}

For $\ell = 0, \dots, \lfloor p/q \rfloor-1$, define $H_\ell=[s+lq,s+(l+1)q)$. 
As the iterates $f_{I,-q}^n(s+\ell q)$ intersect $s$ before returning to $H_l$,  Lemma~\ref{lem: first return} shows that $S \cap H_l$ is an affine progression of difference $r$.
As $I_\ell$ is the leftmost interval of width $r$ in $H_\ell$,  Lemma~\ref{lem: dynamics} shows that $I_\ell \cap S$ is an affine progression with difference $-r'$ when $\ell = 0, \dots, \lfloor p/q \rfloor-1$.
Let $L_{ \lfloor p/q \rfloor}$ denote the rightmost interval of width $q$ in $I$.
Then
Lemma~\ref{lem: dynamics} shows that $S \cap L_{\lfloor p/q \rfloor}$ is an affine progression with difference  $r$.   
As the rightmost interval of width $r$ within $L_{\lfloor p/q \rfloor}$, Lemma~\ref{lem: dynamics} shows that $I_{\lfloor p/q \rfloor}$ is an affine progression with difference $-r'$. 
The argument at the end of the proof of Lemma \ref{lem:dividesqpositive} applies verbatim to show that each $I_\ell$ is $r'$-isolated for $g$ and $m(g) \cap I_\ell \neq \emptyset$.
\end{proof}

In order to run the induction in the proof of Theorem \ref{thm: minimizer}, we need to locate a collection of disjoint subintervals of $I$ that capture the minima of $g$ in $I$ and that are furthermore left-like or right-like, assuming that $I$ is.
In certain cases, we can guarantee this property for the subintervals $I_\ell$ themselves, and in other cases, we need to consider related subintervals (see Figure \ref{fig: intervals}).

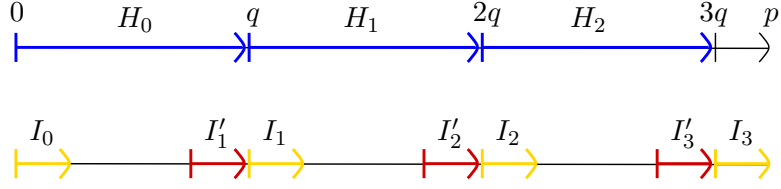
\begin{figure}
\labellist
\pinlabel 0 [Bl] at -3 50
\pinlabel $q$ [Bl] at 70 50
\pinlabel $2q$ [Bl] at 140 50
\pinlabel $3q$ [Bl] at 210 50
\pinlabel $p$ [Bl] at 230 50
\pinlabel $H_0$ [Bl] at 30 48
\pinlabel $H_1$ [Bl] at 100 48
\pinlabel $H_2$ [Bl] at 170 48
\pinlabel $I_0$ [Bl] at 3 13
\pinlabel $I_1$ [Bl] at 75 13
\pinlabel $I_2$ [Bl] at 147 13
\pinlabel $I_3$ [Bl] at 219 13
\pinlabel $I_1'$ [Bl] at 57 13
\pinlabel $I_2'$ [Bl] at 129 13
\pinlabel $I_3'$ [Bl] at 201 13
\endlabellist
\includegraphics[width=4in]{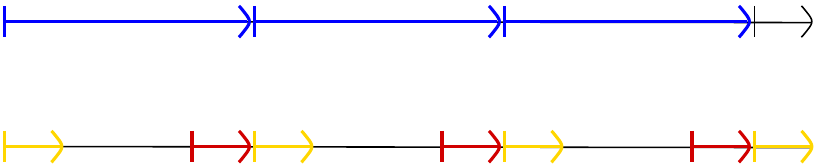}
\caption{Subintervals of $[0,p)$.}
\label{fig: intervals}
\end{figure}

In any case, we need to constrain the location of the break point and to compare the values of $g$ at points spaced $q$ apart.

\begin{lem}
\label{lem: break point location}
Either
\begin{enumerate}
\item
$b > s+p-r$, and $g(x) = g(x-q)$ for all $x,x-q \in I$, $x < b$; or else
\item
$b < s+q+r$, and $g(x) = g(x-q)$ for all $x,x-q \in I$, $x \ge b$.
\end{enumerate}
\end{lem}

\begin{proof}
First, apply Lemma \ref{lem: pattern 2-2}(3). Note that this result does not require that $\theta \nmid q$ though it appears in the earlier subsection. 
It shows that $S$ meets every interval $(x-q,x] \subset I$ in the same number of points $c$ for $x < b$ and in the same number of points $c+1$ for $x \ge b$. 
Setting $\sigma = (q-c) t + c (t - \theta) = q t - c \theta$, we have
\begin{equation}
\label{eq: r difference2}
g(x) - g(x-q) = \begin{cases} \sigma, & s+q \le x < b; \\ \sigma - \theta, & b \le x < s+p. \end{cases}
\end{equation}
Next, by Lemma~\ref{lem:Icontainsall}, there exists a value $x_0 \in m(g) \cap I_0= m(g) \cap [s, s+r)$, and we have $g(x_0+lq) = \min(g)$ for all $l = 0, \dots, \lfloor p/q \rfloor-1$.
Applying \eqref{eq: r difference2} to the values $x = x_0 + l q$, $l=1,\dots,\lfloor p/q \rfloor$, we conclude that either 

\begin{enumerate}
\item
$\sigma=0$, and $b>x_0+ \lfloor p/q \rfloor q \ge s+p-r$, or 
\item
$\sigma=\theta$, and $b \leq x_0+q < s+q+r$,
\end{enumerate} 
which give the two conclusions of the Lemma.
\end{proof}

We first dispense with the possibility that $b \le s+q$.

\begin{lem}
\label{lem: b vs s + q}
Suppose that $b \le s+q$.
If $I$ is left-like, then $I_\ell$ is left-like for all $\ell = 0,\dots,\lfloor p/q \rfloor$. If $I$ is right-like, then $I_\ell$ is right-like for all $\ell = 0,\dots,\lfloor p/q \rfloor$.
\end{lem}

\begin{proof}
We begin by noting that $b \le s+q$ implies that
\begin{enumerate}
\item[(a)]
$s+\ell q+r \in S$ for every $\ell=0, \dots, {\lfloor p/q \rfloor}-1$, and
\item[(b)]
$g(x)=g(x-q)$ for all $x, x-q \in I$.
\end{enumerate}
Deduction (a) uses the fact that the first return of $s$ to $[s,s+q]$ under $f_{I,-q}$ is $s+r$, while (b) uses Lemma \ref{lem: break point location}(2).

Now, suppose that $I$ is left-like.
Deduction (a) along with the fact that $g$ is an $S$-jump function establishes the second condition on left-likeness for each $I_\ell$, $\ell < \lfloor p/q \rfloor$, while the second condition on left-likeness for $I$ implies that for $I_{\lfloor p/q \rfloor}$, viz. $g(s+p) = g(s+p-1) + t - \theta$.
The right side equals $g(s+p-q-1) + t - \theta$ by (b), and this in turn equals $g(s+p-q)$ by (a) and that $g$ is an $S$-jump function.
We conclude that $g(s+p) = g(s+p-q)$.
Thus, (b) and this conclusion yield $g(s) = g(s+\ell q)$ and $g(s+r) = g(s+r+\ell q)$ for all $\ell=1,\dots,\lfloor p/q \rfloor$.
The first condition on left-likeness for $I$ gives $0<g(s+p)-g(s)$, which we can now equate with $g(s+r+\ell q)-g(s+\ell q)$ for all $\ell=0,\dots,\lfloor p/q \rfloor$ and thereby establish the first condition on left-likeness for each $I_\ell$.
Consequently, each $I_\ell$ is left-like.

Next, suppose that $I$ is right-like.
The second condition is equivalent to $s \notin S$ (since $g$ is an $S$-jump function).
Suppose by way of contradiction that $s + \ell q \in S$ for some $1 \le \ell \le \lfloor p/q \rfloor$.
Since $b \le s+q$, it follows that $b = s+q$ and that $S$ has length $p-1$, contradicting Lemma~\ref{lem:Icontainsall}.
Hence the second condition on right-likeness holds for each $I_\ell$.
Since $I$ is right-like $g(s) = g(s-1) + t$.
By (b) above, we have $g(s) = g(s+q)$.
Since $b \ne s+q$, we have $s+q \notin S$.
As $g$ is an $S$-jump function,  $g(s+q) = g(s+q-1) + t$.
We conclude that $g(s+q-1) = g(s-1)$.
Thus, (b) and this conclusion yield $g(s-1) = g(s+\ell q-1)$ and $g(s+r-1) = g(s+\ell q+r-1)$ for all $\ell=1,\dots,\lfloor p/q \rfloor$.
The first condition on right-likeness of $I$ gives $g(s+p-1)-g(s-1) < 0$, which we can now equate with $g(s+\ell q+r-1)-g(s+\ell q-1)$ for all $\ell=0,\dots,\lfloor p/q \rfloor$ and thereby establish the first condition on right-likeness for each $I_\ell$.
Consequently, each $I_\ell$ is right-like.
\end{proof}

Following Lemmas \ref{lem: break point location} and \ref{lem: b vs s + q}, we must now treat the pair of possibilities $b \in (s+p-r,s+p)$ and $b \in (s+q,s+q+r)$ in turn.

\begin{lem}
Suppose that $b \in (s+p-r,s+p)$.
If $I$ is left-like, then $I_{\lfloor p/q \rfloor}$ is left-like, and if $I$ is right-like, then $I_{\lfloor p/q \rfloor}$ is right-like.
\end{lem}

\begin{proof}
By Lemma \ref{lem: break point location}(1), $g$ is $q$-periodic on $[s,s+p-r]$.
In particular, $g(s) = g(s+p-r)$.

If $I$ is left-like, then $0 < g(s+p)-g(s)$ and $g(s+p) - g(s+p-1) = t - \theta$.
It follows that $0 < g(s+p)-g(s+p-r)$, so $I_{\lfloor p/q \rfloor}$ is left-like.

If $I$ is right-like, then $g(s+p-1)-g(s-1) < 0$ and $s \notin S$ (since $g$ is an $S$-jump function).
Note that $s+p-r \notin S$: otherwise, since $b > s+p-r$, we would have $s+\ell q \in S$ for all $\ell = 1,\dots, \lfloor p/q \rfloor$, and then $S$ has length $p-1$ and break point $s+q$, either of which is a contradiction.
This verifies the second condition of right-likeness for $I_{\lfloor p/q \rfloor}$.
Now, since $s, s+p-r \notin S$, we have $g(s) = g(s-1) + t$ and $g(s+p-r) = g(s+p-r-1)+t$; and, as noted at the outset, we have $g(s+p-r) = g(s)$.
Consequently, $g(s-1) = g(s+p-r-1)$, and we obtain $g(s+p-1)-g(s+p-r-1) < 0$, verifying the first condition of right-likeness for $I_{\lfloor p/q \rfloor}$.
\end{proof}

Define $I_\ell'=[s+\ell q-r,s+\ell q)$ for $\ell=1,\dots,\lfloor p/q \rfloor$, and recall that $H_\ell = [s+\ell q, s+ (\ell+1)q)$ for $\ell = 0, \dots, \lfloor p/q \rfloor -1$.
Thus, $I_\ell$ is the leftmost subinterval of width $r$ in $H_\ell$, and $I_{\ell+1}'$ is the rightmost subinterval of width $r$ in $H_\ell$, for $\ell=0,\dots,\lfloor p/q \rfloor -1$.
See Figure \ref{fig: intervals}.

\begin{lem}
\label{lem:auxintervals1}
Suppose that $b \in (s+p-r,s+p)$ and $\theta \nmid r$.
Then either $I_\ell$ is left-like, for $\ell=0,\dots,\lfloor p/q \rfloor -1$, or else $I_\ell'$ is right-like, for $\ell=1,\dots,\lfloor p/q \rfloor$.
In the latter case, $S \cap I_\ell'$ is an affine progression of width $r$ and difference $-r'$, $I_\ell'$ is $r'$-isolated for $g$, and $m(g) \cap I_\ell' \ne \emptyset$, for $\ell=1,\dots,\lfloor p/q \rfloor$.
Lastly, in this latter case, $m(g) \cap I \subset I_1' \cup \dots \cup I_{\lfloor p/q \rfloor}' \cup I_{\lfloor p/q \rfloor}$. 
\end{lem}

\begin{proof}
We focus attention on $H_0$.
It is the left-most subinterval of $I$ of width $q$.
Consequently, $S \cap H_0$ is an affine progression of difference $r$.
Since $b$ lies in the rightmost subinterval of width $r$ in $I$, it follows that the break point of $S \cap H_0$ in $H_0$ is $f^{-1}_{I, -q}(b)=f_{I,q}(b)$, and it lies in the right-most subinterval of width $r$ in $H_0$.
Apply Lemma \ref{lem: minima} to $(H_0,S \cap H_0)$.
Note that the hypothesis $\theta \nmid r$ enables us to apply it and that the Addendum holds by the remark about the break point.
It follows that one of its conclusions (1) and (3) hold.  Thus, either

\begin{enumerate}
\item
$I_0$ is left-like: $g(s+r)-g(s) > 0$ and $g(s+r)-g(s+r-1) = t - \theta$; or
\item
$I_1'$ is right-like: $g(s+q-1) - g(s+q-r-1) < 0$ and $g(s+q-r) - g(s+q-r-1) = t$;
furthermore, $m(g) \cap H_0 \subset I_1'$, and $S \cap I_1'$ is an affine progression of difference $-r'$.
\end{enumerate}
Recall that 
Lemma \ref{lem: break point location}(1) shows that $g$ is $q$-periodic on $[s,s+p-r] = H_0 \cup H_1 \cup \cdots \cup H_{\lfloor p/q \rfloor -1} \cup \{s + p - r \}$.

Assume $(1)$ holds. For $\ell=0,\dots,\lfloor p/q \rfloor$, Lemma \ref{lem:dividesqnegative1} shows $S \cap I_\ell$ is an affine progression of difference $-r'$. The $q$-periodicity on $[s,s+p-r]$ shows that the left-like criterion for $I_0$ implies it for $I_\ell=0,\dots,\lfloor p/q \rfloor-1$. This gives us one conclusion of the Lemma.

Assume $(2)$ holds. As the iterates $f_{I,-q}^n(s+\ell q)$ intersect $s$ before returning to $H_l$,  Lemma~\ref{lem: first return} shows that $S \cap H_l$ is an affine progression with difference $r$ for each $\ell = 0, \dots, \lfloor p/q \rfloor -1$.
As $I_{\ell+1}'$ is the rightmost subinterval of width $r$ in $H_\ell$, for $\ell=0,\dots,\lfloor p/q \rfloor -1$, Lemma \ref{lem: dynamics} shows that $S \cap I_{\ell}'$ is an affine progression with difference $-r'$ for $\ell = 1, \dots, \lfloor p/q \rfloor$.
As $m(g) \cap H_0 \subset I_1'$, we conclude $\emptyset \ne m(g) \cap H_\ell \subset I_{\ell+1}'$ for $\ell=0,\dots,\lfloor p/q \rfloor -1$, using the $q$-periodicity of $g$ and Lemma \ref{lem:Icontainsall}.
Again applying Lemma \ref{lem:Icontainsall}, we see $m(g) \cap I \subset I_1' \cup \dots \cup I_{\lfloor p/q \rfloor}' \cup I_{\lfloor p/q \rfloor}$.

What remains is to establish, in this latter case, the $r'$-isolation of the $I_\ell'$.
Combining Lemma \ref{lem:Icontainsall} with the above, we have $\emptyset \ne m(g) \cap H_\ell \subset I_\ell \cap I_{\ell+1}'$ for $\ell=0,\dots,\lfloor p/q \rfloor -1$.
Since $I_\ell$ is the left-most interval of width $r$ in $H_\ell$ and $I_{\ell+1}'$ is the right-most interval of width $r$ in $H_\ell$, it follows that $q = r +r'$ and that $m(g) \cap H_\ell \subset [s+ \ell q + r', s + (\ell+1) q -r')$ for $\ell=0,\dots,\lfloor p/q \rfloor -1$.
From this we see that $I_\ell'$ is $r'$-isolated for $\ell=1,\dots,\lfloor p/q \rfloor -1$, and that there are no minima of $g$ to the left of $I_{\lfloor p/q \rfloor}'$ and within $r'$ units of it.
The last thing is to show that there are no minima of $g$ to the right of $I_{\lfloor p/q \rfloor}'$ and within $r'$ units of it.
This is because of Lemma \ref{lem:Icontainsall}: if $x \in m(g)$ were such a point, then $x-q \in m(g)$ as well; but $x-q$ is in the leftmost interval of width $r'$ in $H_{\lfloor p/q \rfloor - 1}$, which we already argued is disjoint from $m(g)$.
\end{proof}

For the case $b \in (s+q,s+q+r)$, we obtain two more Lemmas in direct analogy with the previous two. The proofs are completely analogous, and we omit them.

\begin{lem}
Suppose that $b \in (s+q,s+q+r)$.
If $I$ is left-like, then $I_0$ is left-like, and if $I$ is right-like, then $I_0$ is right-like. \qed
\end{lem}

For $\ell = 1, \dots, \lfloor p/q \rfloor$, define $L_\ell=[s+(\ell-1)q+r,s+\ell q+r)$. 
As the iterates of $s$ under $f_{I,-q}$ first intersect $L_\ell$ in its left endpoint, $S \cap L_\ell$ is an affine progression of difference $r$. Define $I_\ell''=[s+(\ell-1) q+r,s+(\ell-1) q+2r)$ for $\ell=1,\dots,\lfloor p/q \rfloor$.
As the initial interval of width $r$ in $L_\ell$,  Lemma~\ref{lem: dynamics} shows that $I_\ell'' \cap S$ is an affine progression
with difference $-r'$, when $\ell = 1, \dots, \lfloor p/q \rfloor$.
Note that $I_{\ell}$ is the rightmost subinterval of width $r$ in $L_\ell$, and $I_\ell''$ is the leftmost subinterval of width $r$ in $L_\ell$, for $\ell=1,\dots,\lfloor p/q \rfloor$. The intervals $L_\ell, I_\ell''$ play the roles in the proof of Lemma~\ref{lem:auxintervals2} played by $H_\ell, I_\ell'$ in the proof of Lemma~\ref{lem:auxintervals1} .

\begin{lem}
\label{lem:auxintervals2}
Suppose that $b \in (s+q,s+q+r)$ and $\theta \nmid r$.
Then either $I_\ell$ is right-like for $\ell=1,\dots,\lfloor p/q \rfloor$, or else $I_\ell''$ is left-like for $\ell=1,\dots,\lfloor p/q \rfloor$.
In the latter case, $S \cap I_\ell''$ is an affine progression of difference $-r'$, $I_\ell''$ is $r'$-isolated for $g$, and $m(g) \cap I_\ell'' \ne \emptyset$, for $\ell=1,\dots,\lfloor p/q \rfloor$.
Lastly, in this latter case, $m(g) \cap I \subset I_0 \cup I_1'' \cup \dots \cup I_{\lfloor p/q \rfloor}''$. \qed
\end{lem}

The net result of our preparatory Lemmas in the case that the difference of $S$ is $-q$ reads thus:

\begin{lem}\label{lem: thetadividesdiffneg}
Let $S$ be an affine progression in $I=[s,s+p)$ with difference $-q$ where $q>0$. Let $r=[p]_q$, and let $r'=[q]_r$.
Suppose that $g$ is an $S$-jump function with parameters $t,\theta$ and that $\theta \nmid r$.
Assume $g$ has a $(p,q)$-symmetry in its minima, that $m(g) \cap I \neq \emptyset$, and $I$ is $q$-isolated for $g$.
Suppose that $I$ is left-like or right-like.
Then $m(g) \cap I$ is contained in the union of pairwise disjoint subintervals $J_0, \dots, J_{\lfloor p/q \rfloor}$.
For each $\ell = 0,\dots,\lfloor p/q \rfloor$, $S \cap J_\ell$ is an affine progression of width $r$ and difference $-r'$, $m(g) \cap J_\ell \ne \emptyset$, and $J_\ell$ is $r'$-isolated and either left-like or right-like for $g$. \qed
\end{lem}

\subsection{Harmonics and minima.}
\label{ss: harmonics}

In this subsection, we prove Theorem~\ref{thm: minimizer} using the notion of an affine progression and the results developed in the preceding subsections.
We refer to Subsections \ref{ss: main result} and \ref{ss: notation} of the Introduction for the salient notation and recall the statement of Theorem~\ref{thm: minimizer} here (where $r_0 = q$ and the empty product is taken to be $1$):

\thmminimizer*

\noindent
{\em Example.} Take $(p,q,k) = (33,10,11)$.
It corresponds to case (2) in the example following Lemma \ref{lem: minima}.
The pair $(33,10)$ has the sequence of remainders $(3,1)$ and sequence of coefficients $(3,3,3)$.
As indicated in that example, $|m(g) \cap I|=4$.
In this case, $\theta = 3$ divides $r_1$, so $\prod_{\{j : \theta | r_j\}} (d_{j+1}+1)=d_2+1=4$, in conformity with Theorem \ref{thm: minimizer}.
As another example, take $(p,q,k)=(46,17,23)$.
The pair $(46,17)$ has the sequence of remainders $(12,5,2,1)$ and sequence of coefficients $(2,1,2,2,2)$.
In this case, $\theta=2$ divides $r_1$ and $r_3$, leading to the product $(d_2+1)(d_4+1) = 6$, which one may check equals $|m(g) \cap I|$.

\begin{proof}
We wish to compute the number of times $S(p,q,k)$ achieves its minimum value of 0.
Thus, we wish to compute $|m(g) \cap I|$, where $I=[0,p)$.
We do so by honing in on the set $m(g) \cap I$ by an iterative procedure.
Define $\iota(\theta,p,q)= \{1 \leq j \leq n :  \theta \nmid r_{j-1}\}$.
Note that $1 \in \iota(\theta,p,q)$, by the convention that $r_0 = q$, and $n \in \iota(\theta,p,q)$, because $r_{n-1} =1$.
By convention, we take $r_n=0$ below, although we are never concerned with the divisibility of $r_n$ by $\theta$.

We recursively construct a sequence $\I_j$, $j \in \iota(\theta,p,q)$. 
Each $\I_j$  is a collection of disjoint subintervals of $I$, and it possesses the following properties:

\begin{enumerate}
\item
$\I_1$ consists of a single interval of width $q$.
In general, each $J \in \I_j$ has width $r_{j-1}$.
In particular, each $J \in \I_n$ consists of a single element.
\item
$(m(g) \cap I) \subset \bigcup_{J \in \I_j} J$; and, for each $J \in \I_j$, $m(g) \cap J  \neq \emptyset$. 
\item
For each $J \in \I_j$, $S \cap J$ is an $r_j$-isolated affine progression with width $r_{j-1}$ and difference $\pm r_j$; furthermore, if the difference is negative, then $J$ is left-like or right-like.
\item
If $j < n$ and $\theta \nmid r_j$, then $j+1 \in \iota(\theta,p,q)$ and $|\I_{j+1}| = |\I_j|$.
\item
If $j < n$ and $\theta \mid r_j$, then $j+1 \notin \iota(\theta,p,q)$, $j+2 \in \iota(\theta,p,q)$, and $|\I_{j+2}| = \lceil r_{j-1}/ r_j \rceil \cdot |\I_j|$.
\end{enumerate}

Assuming the existence of the $\I_j$, we derive the conclusion of the theorem.
Properties (1) and (2) show that $|m(g) \cap I|=|\I_n|$.
We calculate $|\I_n|$ as the telescoping product of $|\I_1|$ with the quotients $|\I_k| / |\I_j|$, as $(j,k)$ runs over pairs of consecutive indices in $\iota(\theta,p,q)$ with $j < k$.
First, we have $|\I_1| = 1$, by (1).
Next, let $j < k$ be consecutive indices in $\iota(\theta,p,q)$.
If $\theta \nmid r_j$, then $k=j+1$, while if $\theta \mid r_j$, then $k = j+2$, by (4) and (5).
The quotient $|\I_k| / |\I_j|$ accordingly equals 1 or $\lceil r_{j-1}/ r_j \rceil$, again by (4) and (5).
Thus, we get the desired expression $|m(g) \cap I| =|\I_n| = \prod_{ 0 < j < n, \, \theta | r_j} \lceil r_{j-1}/ r_j \rceil$.
Lastly,  $r_{j-1}/ r_j = d_{j+1} + (r_{j+1}/r_j)$ from the Euclidean algorithm.
Note that if $\theta \mid r_j$, $j < n$, then in fact $j < n-1$.
Hence for $0 < j < n$, $\theta \mid r_j$, we have $0 < r_{j+1}/r_j < 1$, and so $\lceil r_{j-1}/ r_j \rceil = d_{j+1}+1$, completing the identity.

To begin the construction of the $\I_j$, note that since $p$ and $q$ are relatively prime, $\theta \nmid q$.
Applying Lemma \ref{lem: minima} to $S$ and $I=[0,p)$ yields a subinterval $J \subset I$ such that $m(g) \cap I = m(g) \cap J$.
Moreover, $J \cap S$ is an affine progression of width $q$ and difference $\pm r_1$, and if the difference is $-r_1$, then $J$ is left-like or right-like.
Since $m(g) \cap I \subset J$ and $g$ is $p$-periodic, we obtain $m(g) \subset J + p \cdot \bZ$, and the translates of $J$ by $p \cdot \bZ$ are spaced at least $[p]_q=r_1$ units apart.
Hence $J$ is $r_1$-isolated.
Setting $\I_1=\{J\}$, we obtain properties (1)-(3) for $j=1$.

Suppose by induction that  $j \in \iota(\theta,p,q)$, $j < n$, and for each $k \in \iota(\theta,p,q)$, $k \le j$, the collection $\I_k$ is constructed with properties (1)-(5) for $k < j$ and (1)-(3) for $k=j$.

First, assume that $\theta \nmid r_j$.
Then $j+1 \in \iota (\theta,p,q)$, by definition.
Let $J \in \I_j$.
By (3), $J \cap S$ is an $r_j$-isolated affine progression of width $r_{j-1}$ and difference $\pm r_j$.
If the difference is $r_j$, then we apply Lemma \ref{lem: minima} to $J$, with $p=r_{j-1}, q=r_j, r=r_{j+1}$ in its statement.\footnote{
As indicated at the outset of Section \ref{s: minima}, we indulge in the following abuse of notation.
In Theorem~\ref{thm: minimizer}, $p,q$ stand for the parameters of the lens space.
In the calls to Lemmas~\ref{lem: minima}, \ref{lem: minima 2-2}, \ref{lem:dividesqpositive}, and \ref{lem: thetadividesdiffneg}, $p,q$ stand for the local variables.}
We find a subinterval $J' \subset J$ such that $J' \cap m(g) = J \cap m(g)$ and such that $J' \cap S$ is an $r_{j+1}$-isolated affine progression of width $r_j$ and difference $\pm r_{j+1}$.
Furthermore, if the difference of $J'$ is $-r_{j+1}$, then $J'$ is left-like or right-like.
If instead $J$ has difference $-r_j$, then by assumption it is left-like or right-like.
We apply Lemma~\ref{lem: minima 2-2} to $J$, with $p=r_{j-1}, q=r_j, r=r_{j+1}$.
We find a subinterval $J' \subset J$ such that $m(g) \cap J = m(g) \cap J'$ and such that $J' \cap S$ is an $r_{j+1}$-isolated affine progression of width $r_j$ and difference $r_{j+1}$.
In either case, we associate $J'$ to $J \in \I_j$.
Set $\I_{j+1}=\{J' \, | \, J \in \I_j\}$.
Then (1)-(3) hold for $\I_{j+1}$ and (4) holds for $\I_j$.  

Second, assume that $\theta \mid r_j$.
Theorem~\ref{thm:lotsofzerosg} shows that $g$ has an $(r_{j-1},r_j)$-symmetry on its minima, a hypothesis in the Lemmas we shall invoke.
Note that $j+1 \notin \iota(\theta,p,q)$, by definition.
Also, $\theta \nmid r_{j+1}$, since $\theta$ does not divide consecutive remainders.
Hence $j+2 \in \iota(\theta,p,q)$.
Let $J \in \I_j$.
Then $J \cap S$ is an $r_j$-isolated 
affine progression of width $r_{j-1}$ and difference $\pm r_j$.
If the difference of $J$ is $r_j$, then we apply Lemma~\ref{lem:dividesqpositive} to $J$, with $p=r_{j-1}, q=r_j, r=r_{j+1}, r'=r_{j+2}$.
We find disjoint subintervals $J_0', \dots, J_{\lfloor r_{j-1}/r_j \rfloor}'$ of $J$ with the properties that
\begin{itemize}
\item $J_i' \cap S$ is an $r_{j+2}$-isolated affine progression of width $r_{j+1}$ and difference $r_{j+2}$ for each $i$,
\item $m(g) \cap J_i' \neq \emptyset$ for each $i$, and
\item $m(g) \cap J \subset J_0' \cup J_1' \cup \dots \cup J_{\lfloor r_{j-1}/r_j \rfloor}'$. 
\end{itemize} 
If the difference of $J$ is $-r_j$, then $J$ is right-like or left-like.
We apply Lemma~\ref{lem: thetadividesdiffneg} to $J$, with $p=r_{j-1}, q=r_j, r=r_{j+1}, r'=r_{j+2}$.
We find disjoint subintervals $J_0', \dots, J_{\lfloor r_{j-1}/r_j \rfloor}'$ of $J$ with the properties that
\begin{itemize}
\item $J_i' \cap S$ is an $r_{j+2}$-isolated affine progression of width $r_{j+1}$ and difference $-r_{j+2}$ for each $i$,
\item $J_i'$ is left-like or right-like for each $i$,
\item $m(g) \cap J_i' \neq \emptyset$ for each $i$, and
\item $m(g) \cap J \subset J_0' \cup J_1' \cup \dots \cup J_{\lfloor r_{j-1}/r_j \rfloor}'$. 
\end{itemize}
In either case, define $\I_{j+2}$ to be the set of the $\lfloor r_{j-1}/r_j \rfloor +1=\lceil r_{j-1}/r_j \rceil$
disjoint subintervals associated to $J$ above, unioned over all $J \in I_j$.
Then (1)-(3) hold for $\I_{j+2}$ and (5) holds for $\I_j$.

This completes the inductive construction of the sequence $\I_j$, $j \in \iota(\theta,p,q)$, and verifies properties (1)-(5) for them.
The proof of the Theorem is now complete.
\end{proof}

\section{Criteria for fibering.}
\label{s: fibering}

In this Section, we prove Proposition~\ref{prop: coeffs 2} and Theorem~\ref{thm:fibersiffS}, thereby establishing Theorem \ref{thm: characterization}. We also prove Theorem \ref{thm: seifert} and Corollary~\ref{c: order bound}.
The material of this section is classical in nature.

Let $K=K(p,q,k)$ denote a simple knot in $L(p,q)$ and $X$ its exterior.
Since $L(p,q)$ is a rational homology sphere, the group $H^1(X;\bZ)$ is infinite cyclic, with a generator $\phi$ that is Poincar\'e dual to a (rational) Seifert surface for $K(p,q,k)$.
Let $\Delta(p,q,k) = \Delta_\phi \in \bZ[T,T^{-1}]$ denote its Alexander polynomial.
It is well-defined up to multiplication by a unit, and we use the symbol $\sim$ to denote equality in $\bZ[T,T^{-1}]$ up to multiplication by a unit.
Recall from Section~\ref{ss: notation} that $S(p,q,k)$ is the sequence of values output by the jump function $g : \{0,\dots,p-1\} \to \bZ$ of \eqref{eq: jump function}, and  $\theta$ denotes the order of $k \pmod p$.

We have the following refinement of Proposition \ref{prop: coeffs}:

\begin{prop}
\label{prop: coeffs 2}
The generating function of the sequence $S(p,q,k)$ equals the Alexander polynomial $\Delta(p,q,k)$, multiplied by a monic polynomial:
\[
\sum_{i=0}^{p-1} T^{g(i)} \, \sim \, \frac {T^\theta-1}{T-1} \cdot \Delta(p,q,k).
\]
In particular, the leading coefficient of $\Delta(p,q,k)$ equals the number of times the sequence $S(p,q,k)$ attains its minimum value.
\end{prop}

\noindent
{\em Example.}
For the knot $K(5,2,1)$, a glance at Figure \ref{fig: domain} shows that the left side equals $T^4+T^3+T^2+T+1 = (T^5-1)/(T-1)$.
Consequently, $\Delta(5,2,1) \sim 1$.
Indeed, as $K(5,2,1)$ is isotopic to the core of $V_\beta$, its complement is a solid torus, which matches with this calculation.

\begin{figure}
\includegraphics[width=2in]{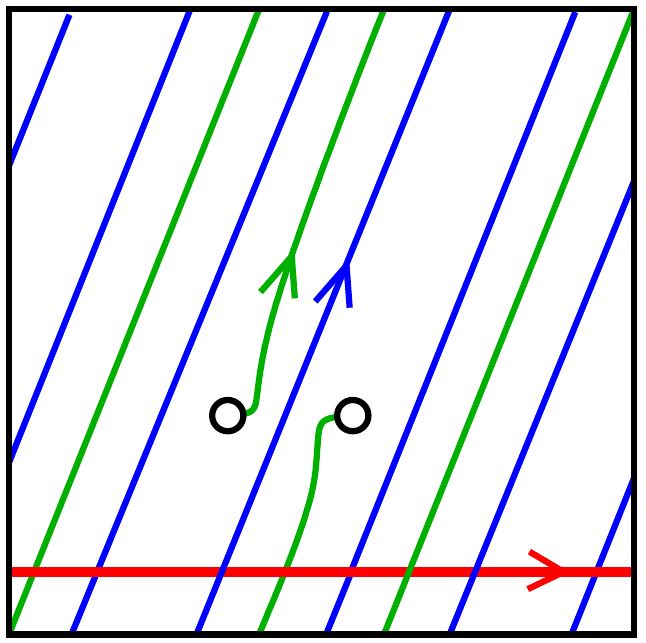}
\put(0,135){$\beta_1$}
\put(0,100){$\beta$}
\put(0,12){$\alpha$}
\caption{Heegaard diagram of the exterior $K(5,2,1)$.}
\label{fig: doubly pointed}
\end{figure}

\begin{proof}
We calculate $\Delta(p,q,k)$ from the presentation of $K(p,q,k)$ by a doubly-pointed Heegaard diagram, as in \cite[Section 5]{os:lens} and \cite[Section 3.7]{r:Lspace}.

First, we follow a familiar procedure to obtain a Heegaard decomposition of $X$ from the presentation of $K$ by a doubly-pointed Heegaard diagram.
Push $K$ off of $D_\alpha$ and $D_\beta$ so that it meets $T$ in a pair of points, up and to the right of $x_0$ and $x_k$ (i.e. in the regions $R_{ [q]_p}$ and $R_{ [k+q]_p}$).
By abuse of notation, let $D^1_\beta$ and $\beta_1$ denote shifted copies of these objects with respect to the new positioning of the knot.
The exterior of $K$ in $V_\beta$ is a genus-2 handlebody with compressing disks $D_\beta$, $D^1_\beta$.
The handlebody has a decomposition into a 0-handle and a pair of 1-handles whose cocores are these two disks.
We obtain $X$ by attaching a 2-handle with core $D_\alpha$.
See Figure \ref{fig: doubly pointed}.

Next, the Heegaard decomposition of $X$ leads, by an equally familiar procedure, to a presentation for its fundamental group $\pi$.
We record generators $x$ and $y$ for $\beta$ and $\beta_1$, respectively, corresponding to the cores of the 1-handles.
We record a single relation $r$ for the 2-handle $\alpha$, as follows.
Traverse $\alpha$ in the direction of its orientation, starting from a point just after the intersection point $x_0$ with $\beta$.
Record the letter $x^{\pm1}$ for each intersection point with $\beta$ of sign $\pm$, and record the letter $y^{\pm1}$ for each intersection point with $\beta_1$ of sign $\pm$.
The resulting word is the relation $r$, and we have $\pi \approx \langle x,y \, | \, r \rangle$.
From the construction, the relation takes the precise form
\[
r = \prod_{i=0}^{p-1} x y^{e_i},
\]
where $e_i = 1$ if $i \in S = \{ [q]_p, [2q]_p, \dots, [lq]_p\}$ and $e_i = 0$ if $i \notin S$.

To obtain the Alexander polynomial from this presentation, we first calculate the Fox free derivative
\[
\del_x r = \sum_{j=0}^{p-1} \prod_{i=0}^{j-1} x y^{e_i} \in \bZ[ \langle x,y \rangle ],
\]
understanding 1 for the empty product corresponding to $j=0$.
The generator $\phi \in H^1(X;\bZ) \approx \Hom(\pi, \bZ)$ is determined uniquely by the conditions that $\phi(r) = 0$, $\phi(x) > 0$, and $\phi(x)$ and $\phi(y)$ are coprime.
From $\phi(r)=0$, we obtain $p \phi(x) + l \phi(y) = 0$.
We solve this equation subject to the other two conditions to obtain $\phi(x) = l / \gcd(p,l) = t$ and $\phi(y) = - p / \gcd(p,k) = -\theta$.
Thus, for $j=0,\dots,p-1$, we have
\[
\phi \left( \prod_{i=0}^j x y^{e_i} \right) - \phi \left( \prod_{i=0}^{j-1} x y^{e_i} \right) = \phi(x y^{e_j}) = \begin{cases} 
t-\theta, & j \in S \\ t, & j \notin S. \end{cases}
\]
That is, there is a constant $C$ such that $\phi \left( \prod_{i=0}^j x y^{e_i} \right)+C, \,  j=0, \dots p-1,$ are the values of the jump function $g$ of Equation \eqref{eq: jump function} of Section~\ref{ss: notation} corresponding to the sequence $S(p,q,k)$.

The mapping $\phi$ extends to a map of group rings $\phi : \bZ[\langle x,y \rangle] \to \bZ[T,T^{-1}]$.
From the preceding formula, we see that $\phi(\del_x r) \in \bZ[T,T^{-1}]$ agrees, up to multiplication by a unit, with the generating function of the sequence $S(p,q,k)$.
On the other hand, $\phi(\del_x r) \sim \Delta(p,q,k) \cdot (T^{\phi(y)}-1)/(T-1)$ (cf. \cite[Proof of Proposition 3.1]{r:Lspace}; there it is assumed that $\theta=p$, but the argument goes through using the fact that $\phi$ is surjective).
Substituting $\phi(y)=-\theta$ and multiplying by the unit $-T^\theta$ yields the first assertion of the Proposition.
The second assertion of the Proposition follows easily from the first.
\end{proof}

The following result establishes the equivalence of the first two items in Theorem~\ref{thm: characterization}:

\begin{thm}
\label{thm:fibersiffS}
$K(p,q,k)$ fibers if and only if $S(p,q,k)$ achieves its minimum exactly once.
\end{thm}

\begin{proof}
Let $\phi$ be the generator of $H^1(X;\bZ)$ taken in the proof of Theorem~\ref{prop: coeffs 2}.
As in that proof, $\pi_1(X)$ can be presented as $\langle x,y \, | \, r \rangle$ where
$\phi(x) = t$, $\phi(y) = -\theta$, and
\[
r = \prod_{i=0}^{p-1} x y^{e_i},
\]
with $e_i = 1$ if $i \in S = \{ [q]_p, [2q]_p, \dots, [lq]_p\}$ and $e_i = 0$ if $i \notin S$.

Letting $y'=xy$, we rewrite the presentation $\langle x,y' \, | \, r' \rangle$ where
\[
r' = \prod_{i=0}^{p-1} x_i,
\]
with $x_i = y'$ if $i \in S = \{ [q]_p, [2q]_p, \dots, [lq]_p\}$ and $x_i = x$ if $i \notin S$.

Define $s_i=x_0 x_1 \dots x_i$ for $i=0, \dots, p-1$. The proof of Theorem~\ref{prop: coeffs 2} shows there is a constant $C$ such that $\phi(s_i)+C=g(i)$, where $g$ is the jump function from (2) of Section~\ref{ss: notation}. The sequence $\phi(s_i),  i=0, \dots ,p-1,$ assumes a unique minimum iff $S(p,q,k)$ assumes a unique minimum. By  \cite[Theorem 4.3]{brown1987}, the kernel of $-\phi$ is finitely generated iff $-\phi(s_i), i=0, \dots ,p-1,$ assumes a unique maximum -- which is then equivalent to $S(p,q,k)$ attaining a unique minimum. By \cite[Theorems 1\&2]{stallings1962}, $K$ is fibered iff $S(p,q,k)$ attains a unique minimum.
\end{proof}

Recall Theorem \ref{thm: seifert}:
\thmseifert*

\begin{proof}
For a knot $K$ in a rational homology sphere and a (rational) Seifert surface $S$ for $K$, the Alexander polynomial satisfies the bound
\begin{equation}
\label{e: alex bound}
\mathrm{breadth} \, \Delta(K) \le 1 - \chi(S).
\end{equation}
(For example, see \cite[Theorem 1.1]{mcmullen2002}.)  For a simple knot $K$, Proposition \ref{prop: coeffs 2} gives
\begin{equation}
\label{e: breadth}
\mathrm{breadth} \, \Delta(K) = 1 - \theta +  M,
\end{equation}
where $M = \max_i c_i$ is the maximum coefficient in the domain associated with $K$ of Section \ref{ss: domains}.
We obtain a 2-chain $C'$ in $L(p,q)$
by replacing each coefficient $c_i$ in the 2-chain $C$ with $M-c_i$.
Just as $C$ yields a Seifert surface $S$ for $K(p,q,k)$, the 2-chain $C'$ yields a Seifert surface $S'$ for $-K(p,q,k)$.
Applying Lemma \ref{lem:euler} to both $S$ and $S'$ and summing the results, we obtain
\begin{equation}
\label{e: chi sum}
\chi(S)+ \chi(S') = \theta - \frac14 \sum |R_i \cap \partial X_K| \cdot c_i + \theta - \frac14 \sum |R_i \cap \partial X_K| \cdot (M-c_i) = 2\theta - 2M.
\end{equation}
Since $\Delta(K)$ is insensitive to orientation reversal, \eqref{e: breadth} and \eqref{e: chi sum} yield
\[
2 \cdot \mathrm{breadth} \, \Delta(K) = (1 - \chi(S)) + (1 - \chi(S')).
\]
Combined with \eqref{e: alex bound}, we learn that $\mathrm{breadth} \, \Delta(K) = 1 - \chi(S) = 1 - \chi(S')$ and that both of $S$ and $S'$ are taut Seifert surfaces for $\pm K(p,q,k)$.
\end{proof}

We finish with the following corollary of Theorem \ref{thm: fiber}:

\orderbound*

\begin{proof}
Suppose that $K \subset L(p,q)$ does not fiber and has order $\theta$.
By Theorem \ref{thm: fiber}, there exist positive integers $b$ and $c$ and an index $i$ such that $p = \theta b$ and $r_i= \theta c$.
Substituting into Corollary \ref{cor: identities} (3) and (4) gives
\[
(-1)^i \theta c =  p_i q- (\theta b) q_i \quad \textup{and} \quad \theta b = (\theta c) p_{i-1} + p_i r_{i-1}.
\]
From the first equation and the fact that $\gcd(\theta b,q)=1$, we deduce $\theta \, | \, p_i$.
From the second equation and this deduction, we obtain $b = c p_{i-1} + (p_i / \theta) r_{i-1} \ge 1+ r_{i-1} \ge 2 + r_i = 2 + \theta c \ge 2+ \theta$.
Thus, $p = \theta b \ge \theta(\theta+2)$, which establishes the contrapositive.
\end{proof}

\bibliographystyle{amsalpha2}

\bibliography{References}

\end{document}